\documentclass[reqno,a4paper,11pt]{amsart}
\usepackage[utf8]{inputenc}

\calclayout

\usepackage{mathtools,amsthm}
\usepackage{thmtools}
\usepackage{amssymb}
\usepackage{verbatim}
\usepackage[english]{babel}
\usepackage{amsmath}
\usepackage{amsfonts}
\usepackage{amsthm}
\newcommand{\say}[1]{``#1"}
\usepackage{stmaryrd}
\usepackage{float}
\usepackage{color}
\usepackage{tikz}
\usetikzlibrary{decorations.pathreplacing,decorations.markings}
\usepackage{graphicx}
\usepackage{subcaption}
\usepackage[shortlabels]{enumitem}

\usepackage[pdftex,colorlinks]{hyperref}
\hypersetup{
	colorlinks=true,
	linkcolor=blue,
	citecolor=red
}
\tikzset{
	on each segment/.style={
		decorate,
		decoration={
			show path construction,
			moveto code={},
			lineto code={
				\path [#1]
				(\tikzinputsegmentfirst) -- (\tikzinputsegmentlast);
			},
			curveto code={
				\path [#1] (\tikzinputsegmentfirst)
				.. controls
				(\tikzinputsegmentsupporta) and (\tikzinputsegmentsupportb)
				..
				(\tikzinputsegmentlast);
			},
			closepath code={
				\path [#1]
				(\tikzinputsegmentfirst) -- (\tikzinputsegmentlast);
			},
		},
	},
	mid arrow/.style={postaction={decorate,decoration={
				markings,
				mark=at position .6 with {\arrow[#1]{stealth}}
	}}},
	rmid arrow/.style={postaction={decorate,decoration={
				markings,
				mark=at position .4 with {\arrowreversed[#1]{stealth}}
	}}},
}

\makeatletter
\def\grd@save@target#1{%
  \def\grd@target{#1}}
\def\grd@save@start#1{%
  \def\grd@start{#1}}
\tikzset{
  grid with coordinates/.style={
    to path={%
      \pgfextra{%
        \edef\grd@@target{(\tikztotarget)}%
        \tikz@scan@one@point\grd@save@target\grd@@target\relax
        \edef\grd@@start{(\tikztostart)}%
        \tikz@scan@one@point\grd@save@start\grd@@start\relax
        \draw[minor help lines] (\tikztostart) grid (\tikztotarget);
        \draw[major help lines] (\tikztostart) grid (\tikztotarget);
        \grd@start
        \pgfmathsetmacro{\grd@xa}{\the\pgf@x/1cm}
        \pgfmathsetmacro{\grd@ya}{\the\pgf@y/1cm}
        \grd@target
        \pgfmathsetmacro{\grd@xb}{\the\pgf@x/1cm}
        \pgfmathsetmacro{\grd@yb}{\the\pgf@y/1cm}
        \pgfmathsetmacro{\grd@xc}{\grd@xa + \pgfkeysvalueof{/tikz/grid with coordinates/major step}}
        \pgfmathsetmacro{\grd@yc}{\grd@ya + \pgfkeysvalueof{/tikz/grid with coordinates/major step}}
        \foreach \x in {\grd@xa,\grd@xc,...,\grd@xb}
        \node[anchor=north] at (\x,\grd@ya) {\pgfmathprintnumber{\x}};
        \foreach \y in {\grd@ya,\grd@yc,...,\grd@yb}
        \node[anchor=east] at (\grd@xa,\y) {\pgfmathprintnumber{\y}};
      }
    }
  },
  minor help lines/.style={
    help lines,
    step=\pgfkeysvalueof{/tikz/grid with coordinates/minor step}
  },
  major help lines/.style={
    help lines,
    line width=\pgfkeysvalueof{/tikz/grid with coordinates/major line width},
    step=\pgfkeysvalueof{/tikz/grid with coordinates/major step}
  },
  grid with coordinates/.cd,
  minor step/.initial=.2,
  major step/.initial=1,
  major line width/.initial=0.25mm,
}

\makeatletter
\def\l@subsection{\@tocline{2}{0pt}{2.5pc}{5pc}{}}

\DeclareMathOperator{\diag}{diag}

\newcommand{\N}{\mathbb{N}}
\newcommand{\C}{\mathbb{C}}
\newcommand{\R}{\mathbb{R}}
\newcommand{\Z}{\mathbb{Z}}

\usepackage{todonotes}

\newtheorem{theorem}{Theorem}[section]
\newtheorem{prop}[theorem]{Proposition}
\newtheorem{lemma}[theorem]{Lemma}
\newtheorem{corollary}[theorem]{Corollary}

\declaretheoremstyle[
spaceabove=\topsep, spacebelow=\topsep,
headfont=\normalfont\bfseries,
notefont=\bfseries, notebraces={}{},
bodyfont=\normalfont\itshape,
postheadspace=0.5em,
name={\ignorespaces},
numbered=no,
headpunct=.]{mystyle}
\declaretheorem[style=mystyle]{named}

\theoremstyle{definition}

\theoremstyle{remark}
\newtheorem{remark}{Remark} [section]

\numberwithin{equation}{section} 
\begin{document}
\title{Global Phase Portrait and Large Degree Asymptotics for the Kissing Polynomials }

\author[A.~Barhoumi]{Ahmad Barhoumi}

\address[AB]{University of Michigan, Ann Arbor, MI, United States}

\email{barhoumi@umich.edu}

\author[A.~F.~Celsus]{Andrew F. Celsus}

\address[AC]{University of Cambridge, Cambridge, UK.}

\email{a.f.celsus@maths.cam.ac.uk }

\author[A.~Dea\~no]{Alfredo Dea\~no}

\address[AD]{University of Kent, Canterbury, UK, and Universidad Carlos III de Madrid, Spain} 

\email{alfredo.deanho@uc3m.es}

\date{}

\begin{abstract}
	We study a family of monic orthogonal polynomials which are orthogonal with respect to the varying, complex valued weight function, $\exp(nsz)$, over the interval $[-1,1]$, where $s\in\C$ is arbitrary. This family of polynomials originally appeared in the literature when the parameter was purely imaginary, that is $s\in i \R$, due to its connection with complex Gaussian quadrature rules for highly oscillatory integrals. The asymptotics for these polynomials as $n\to\infty$ have been recently studied for $s\in i\R$, and our main goal is to extend these results to all $s$ in the complex plane.
	
	We first use the technique of continuation in parameter space, developed in the context of the theory of integrable systems, to extend previous results on the so-called modified external field from the imaginary axis to the complex plane minus a set of critical curves, called breaking curves. We then apply the powerful method of nonlinear steepest descent for oscillatory Riemann-Hilbert problems developed by Deift and Zhou in the 1990s to obtain asymptotics of the recurrence coefficients of these polynomials when the parameter $s$ is away from the breaking curves. We then provide the analysis of the recurrence coefficients when the parameter $s$ approaches a breaking curve, by considering double scaling limits as $s$ approaches these points. We shall see a qualitative difference in the behavior of the recurrence coefficients, depending on whether or not we are approaching the points $s=\pm 2$ or some other points on the breaking curve. 	
\end{abstract}


\vspace*{-2.5cm}

\maketitle

\tableofcontents

\section{Introduction}
	The main goal of this paper is to determine the asymptotic behavior of the recurrence coefficients of polynomials satisfying the following non-Hermitian, degree dependent, orthogonality conditions:
	\begin{equation}\label{eq: ortho p defn}
		\int_{-1}^1 p_n(z;s) z^k e^{-nf(z;s)}\, dz = 0, \qquad k= 0, 1, \dots, n-1, 
	\end{equation}
    	where $p_n$ is a monic polynomial of degree $n$ in the variable $z$, $f(z;s)=sz$, and $s\in\C$ is arbitrary. Polynomial sequences satisfying non-Hermitian orthogonality conditions similar to \eqref{eq: ortho p defn} first appeared in the literature in the context of approximation theory (c.f. \cite{aptekarev2002sharp,aptekarev1992asymptotics,gonchar1989equilibrium,nikishin1991rational}). In the present day, complex orthogonal polynomials with respect to exponential weights have been studied in \cite{bertola2015asymptotics,bertola2016asymptotics} (with quartic potential) and \cite{bleher2013cubic,bleher2015cubic_ds} (with cubic potential). They have found uses in various areas of mathematics including random matrix theory and theoretical physics \cite{alvarez2013determination,alvarez2014partition,alvarez2015fine, bertola2011boutroux}, rational solutions of Painlev\'e equations \cite{balogh2016hankel,bertola2015zeros,bertola2018painleve,bertola2016asymptotics}, and, of particular interest in the present work, numerical analysis \cite{asheim2012gaussian,deano2015kissing,deano2017computing}. 
	
	Indeed, motivation for the present work is concerned with the numerical treatment of highly oscillatory integrals of the form
	\begin{equation*}
		I_\omega[f]:= \int_{-1}^1 f(z) e^{i\omega z}\, dz,\qquad \omega>0,
	\end{equation*}
	where for sake of exposition, we take $f$ to be an entire function. Historically, the numerical treatment of such integrals falls into two regimes, as explained in the monograph \cite{deano2017computing}. The first regime occurs when $\omega$ is relatively small, and the weight function is not highly oscillatory. In this regime, traditional methods of numerical analysis based on Taylor's Theorem, such as Gaussian quadrature, are adequate and provide a suitable means of evaluating such integrals. However, methods such as Gaussian quadrature require exceedingly many quadrature points as the parameter $\omega$ grows large, and as such, the second regime concerns the treatment of $I_\omega[f]$ when the parameter $\omega$ is large. Here, numerical methods based on the asymptotic analysis of such integrals take over, and methods such as numerical steepest descent are preferred. In order to address this apparent schism between the two regimes, the authors of \cite{asheim2012gaussian} proposed a new quadrature rule based on monic polynomials which satisfy
	\begin{equation}\label{eq: original kps}
		\int_{-1}^1 p_n^\omega(z) z^k e^{i\omega z} \, dz= 0, \qquad k = 0, 1, \dots, n-1.
	\end{equation}
	Note in \eqref{eq: original kps}, the weight function no longer depends on the degree of the polynomial $n$. Letting $\{z_i\}_{i=1}^{2n}$ be the $2n$ complex zeros of $p_{2n}^\omega$, the quadrature rule proposed in \cite{asheim2012gaussian} is to approximate the integral via
	\begin{equation}\label{eq: CGQ}
		\int_{-1}^1 f(z)e^{i\omega z}\, dz \approx \sum_{j=1}^{2n} w_j f(z_j),
	\end{equation}
	where the weights $w_j$ are the standard weights used for Gaussian quadrature. Note that as $\omega\to 0$, the rule \eqref{eq: CGQ} reduces elegantly to the classical method of Gauss-Legendre quadrature. Moreover, \cite[Theorem~4.1]{asheim2012gaussian} shows us that
	\begin{equation}
		\int_{-1}^1 f(z)e^{i\omega z}\, dz - \sum_{j=1}^{2n} w_j f(z_j) = \mathcal{O}\left(\frac{1}{\omega^{2n+1}}\right), \qquad \omega\to\infty, 
	\end{equation}
	showing that the proposed quadrature method attains high asymptotic order as $\omega$ grows, especially when compared to other methods, such as Filon rules, used to handle the numerical treatment of highly oscillatory integrals. For more information on the numerical analysis of oscillatory integrals, the reader is referred to \cite{deano2017computing}, and in particular Chapter 6 for the relations to non-Hermitian orthogonality.  
	
	Despite the theoretical successes of numerical methods based on non-Hermitian orthogonal polynomials listed above, many questions about the polynomials themselves remain open. For instance, as the weight function in \eqref{eq: original kps} is now complex valued, questions such as existence of the polynomials and the location of their zeros can no longer be taken for granted. However, provided the polynomials exist for the corresponding values of $n$ and $\omega$, all of the classical algebraic results on orthogonal polynomials will continue to apply. This is due to the fact that the bilinear form
	\begin{equation}
		\langle f, g \rangle := \int_{-1}^1 f(z) g(z) e^{i\omega z}\, dz
	\end{equation}
	still satisfies the relation $\langle z f,g\rangle =\langle f, z g\rangle$. Indeed, there will still be a Gaussian quadrature rule and the polynomials will still satisfy the famous three term recurrence relation
	\begin{equation}\label{eq: ttrr fake}
		z p_n^\omega(z) = p_{n+1}^\omega(z) + \alpha_n(\omega) p_n^\omega(z) +\beta_{n}(\omega) p_{n-1}^\omega(z).
	\end{equation}
	We restate that the weight function for the polynomials $p_n^\omega$ does not depend on $n$, which is why relations such as \eqref{eq: ttrr fake} continue to hold in the complex setting. 
	
	From a different perspective, we observe that the weight of orthogonality in \eqref{eq: ortho p defn} can be seen as a deformation of the Legendre weight by the exponential of a polynomial potential. Such deformations, in this case with the parameter $s$, have been considered in the context of integrable systems. Following the general theory presented in \cite{BEH}, the Hankel determinant of the corresponding family of orthogonal polynomials (or equivalently, the partition function) is closely related to isomonodromic (i.e. monodromy preserving) deformations of a certain system of ODEs; more precisely, we consider the vector $\mathbf{p}_n(z;s)=[p_n(z,s), p_{n-1}(z;s)]^T$, that satisfies both a linear system of ODEs in the variable $z$, as well as an auxiliary linear system of ODEs in the parameter $s$; then, compatibility between these two systems of ODEs characterizes the isomodromic deformations of the differential system in $z$, see \cite{JMU} and also \cite[Chapter 4]{FIKN2007}. In this case, both linear systems can be obtained by standard techniques from the Riemann--Hilbert problem for the OPs, that we present below, and they can be checked to coincide with the linear system corresponding to the Painlev\'e V equation, as given by Jimbo and Miwa in \cite{JM}, with suitable changes of variable to locate the Fuchsian singularities at $z=0,1,\infty$. We refer to reader to \cite{deano2015kissing} for details of this calculation in the case of purely imaginary $s$. As a consequence, the results of this paper also provide information about solutions (special function solutions, in fact) of Painlev\'e V. For the sake of brevity we do not include the details of this connection here.
	
	The results of \cite{asheim2012gaussian} kick-started the study of the polynomials in \eqref{eq: original kps}, and the authors of \cite{deano2015kissing} dubbed such polynomials the Kissing Polynomials on account of the behavior of their zero trajectories in the complex plane. In particular, the work \cite{deano2015kissing} provides the existence of the even degree Kissing polynomials, along with the asymptotic behavior of the polynomials as $\omega\to\infty$ with $n$ fixed. On the other hand, the asymptotic analysis of the Kissing polynomials for fixed $\omega$ as $n\to\infty$ can be handled via the Riemann-Hilbert techniques discussed in \cite{arnojacobi} or the appendix of \cite{deano2014large}, where it was shown that the zeros of the Kissing polynomials accumulate on the interval $[-1,1]$ as $n\to\infty$ with $\omega>0$ fixed. 
	
	One can also let both $n$ and $\omega$ tend to infinity together, by letting $\omega$ depend on $n$. In order to get a nontrivial limit as the parameters tend to infinity, one sets $\omega=\omega(n)=t n$, where $t\in\R_+$. This leads to the varying-weight Kissing polynomials which satisfy the following orthogonality conditions
	\begin{equation}\label{eq: varying weight kissing polynomials}
		\int_{-1}^1 p_n^t(z)z^k  e^{i n t z}\, dz = 0, \qquad k = 0, 1, \dots, n-1. 
	\end{equation}
	Thus, studying the behavior of the Kissing polynomials in \eqref{eq: original kps} as both $n$ and $\omega$ go to infinity at the rate $t$ is equivalent to studying the behavior of the polynomials in \eqref{eq: varying weight kissing polynomials} as $n\to\infty$. 
	
	The varying-weight Kissing polynomials were first studied in \cite{deano2014large}, where it was shown that for $t<t_c$, the zeros of $p_n^t$ accumulate on a single analytic arc connecting $-1$ and $1$, which we denote here to be $\gamma_{m}(t)$. Here, $t_c$ is the unique positive solution to the equation
	\begin{equation}\label{eq: tc def}
		2 \log\left(\frac{2+\sqrt{t^2+4}}{t}\right) - \sqrt{t^2+4}=0,
	\end{equation}
	numerically given by $t_c\approx 1.32549$. In \cite{deano2014large}, strong asymptotic formulas for $p_n^t$ in the complex plane and asymptotic formulas for the recurrence coefficients were given as $n\to\infty$ with $t< t_c$. Moreover, the curve $\gamma_m(t)$ can be defined as the trajectory of the quadratic differential
	\begin{equation}\label{eq: QD genus 0}
		\varpi^{(0)}_t := -\frac{(2+itz)^2}{z^2-1}\,dz^2
	\end{equation}
	which connects $-1$ and $1$, as shown in \cite[Section~3.2]{deano2014large}.	These results all followed in a standard way from the nonlinear steepest descent analysis of the Riemann-Hilbert problem for these polynomials, to be discussed in Section~\ref{sec: steepest descent}. To cast these results in a manner amenable to our analysis, we restate one of the main results of \cite{deano2014large} below. To establish notation, we define $\gamma_{c,0}:=(-\infty, -1]$. 
	\newline
	\begin{named}[Restatement of Results in \cite{deano2014large}]\label{thm: restatement of alfredo h function}
		Let $t<t_c$. There exists an analytic arc, $\gamma_{m}(t)$, that is the trajectory of the quadratic differential
		\begin{equation*}
			\varpi^{(0)}_t := -\frac{(2+itz)^2}{z^2-1}\,dz^2
		\end{equation*}
		which connects $-1$ and $1$. Furthermore, there exists a function $h(z,t)$ such that
		\begin{subequations}
			\label{eq: subcritical h function}
			\begin{alignat}{2}
			&h(z;t) \text{ is analytic for } z\in\C\setminus\left(\gamma_{c,0}\cup\gamma_m(t)\right), \qquad && \\
			&h_+(z;t) - h_-(z;t) = 4\pi i, \qquad && z\in \gamma_{c,0}, \\
			&h_+(z;s) + h_-(z;s) = 0, \qquad && z\in \gamma_{m}(t),\\
			&h(z;s) = itz + 2\log 2  + 2\log z + \mathcal{O}\left(\frac{1}{z}\right), \qquad && z \to \infty\\ 
			&\Re h(z;s) =\mathcal{O}\left((z\mp 1)^{1/2}\right), \qquad && z\to \pm 1.
			\end{alignat}
		\end{subequations}
		Moreover, 
		\begin{equation}
			\Re h(z;t)=0, \qquad z\in \gamma_{m},
		\end{equation}
		and $\Re h(z)>0$ for $z$ in close proximity on either side of $\gamma_m$. 
	\end{named}
	\begin{remark}
		The function $h$ above is called $\phi$ in the notation of \cite{deano2014large}. Properties of this $h$ function listed above can be found in Section~3 of \cite{deano2014large}. The existence and description of the contour $\gamma_m(t)$ is provided in \cite[Section~3.2]{deano2014large}.
	\end{remark}
	
	The analysis of the varying-weight kissing polynomials for $t>t_c$ was undertaken in \cite{celsus2020supercritical}. Again using the Riemann-Hilbert approach for these polynomials, the authors were able to show that there exist analytic arcs $\gamma_{m,0}(t)$ and $\gamma_{m,1}(t)$ such that the zeros of the varying-weight Kissing polynomials accumulate on $\gamma_{m,0}\cup\gamma_{m,1}$ as $n\to\infty$. We restate some of the main results of \cite{celsus2020supercritical} below.
	\newline
	\begin{named}[Restatement of Results in \cite{celsus2020supercritical}]\label{thm: restatement of celsus/guilherme h function}
		Let $t>t_c$. There exist two analytic arcs, $\gamma_{m,0}(t)$ and $\gamma_{m,1}(t)$, that are trajectories of the quadratic differential
		\begin{equation}\label{eq: QD genus 1}
		\varpi^{(1)}_t := -Q(z,t)\,dz^2,
		\end{equation}
		where
		\begin{equation}
			Q(z,t):= -\frac{t^2 (z-\lambda_0)(z-\lambda_1)}{z^2-1}.
		\end{equation}
		Above, $\lambda_0, \lambda_1\in\C$ uniquely satisfy
		\begin{equation}\label{eq: lambda conditions}
			\lambda_0 + \lambda_1 = \frac{4i}{t}, \qquad \lambda_0=-\overline{\lambda_1}, \qquad \Re \oint_C Q^{1/2}(z;t)\, dz = 0, 
		\end{equation}
		where $C$ is any closed loop on the Riemann surface associated to the algebraic equation $y^2 = Q(z;t)$. The trajectory $\gamma_{m,0}$ connects $-1$ to $\lambda_0$ and the trajectory $\gamma_{m,1}$ connects $\lambda_1$ to $1$. Furthermore, there exists a function $h(z,t)$ such that
		\begin{subequations}
			\label{eq: supercritical h function}
			\begin{alignat}{2}
			&h(z;t) \text{ is analytic for } z\in\C\setminus\left(\gamma_{c,0}\cup\gamma_{m,0}(t)\cup\gamma_{c,1}(t)\cup\gamma_{m,1}(t)\right), \qquad && \\
			&h_+(z;t) - h_-(z;t) = 4\pi i, \qquad && z\in \gamma_{c,0}, \\
			&h_+(z;s) + h_-(z;s) = 4\pi i \omega_0, \qquad && z\in \gamma_{m,0}(t),\\
			&h_+(z;t) - h_-(z;t) = 4\pi i\eta_1, \qquad && z\in \gamma_{c,1}(t), \\
			&h_+(z;s) + h_-(z;s) = 0, \qquad && z\in \gamma_{m,1}(t),\\
			&h(z;s) = itz - \ell  + 2\log z + \mathcal{O}\left(\frac{1}{z}\right), \qquad && z \to \infty\\ 
			&\Re h(z;s) =\mathcal{O}\left((z-\lambda_0)^{3/2}\right), \qquad && z\to \lambda_0,\\
			&\Re h(z;s) =\mathcal{O}\left((z- \lambda_1)^{3/2}\right), \qquad && z\to \lambda_1\\
			&\Re h(z;s) =\mathcal{O}\left((z\mp 1)^{1/2}\right), \qquad && z\to \pm 1.
			\end{alignat}
		\end{subequations}
		Above, $\gamma_{c,1}$ is an analytic arc connecting $\lambda_0$ and $\lambda_1$, and $\ell, \omega_0, \eta_1\in\R$. Moreover, 
		\begin{subequations}
			\begin{alignat}{2}
			&\Re h(z;t)=0, \qquad &&z\in \gamma_{m,0}(t)\cup\gamma_{m,1}(t),\\
			&\Re h(z;t) < 0, \qquad &&z\in \gamma_{c,1}(t),
			\end{alignat}
		\end{subequations}
		and $\Re h(z)>0$ for $z$ in close proximity on either side of $\gamma_{m,0}$ and $\gamma_{m,1}$. 
	\end{named}
	\begin{remark}
		The $h$ function described above is given by $h(z;s)=-2\phi(z)+i\kappa$ in the notation of \cite{celsus2020supercritical}, where $\kappa\in\R$ is a real constant of integration. Moreover, the quadratic differential listed above differs from that of \cite{celsus2020supercritical} by a factor of $4$. For more details, we refer the reader to Sections 4 and 5 of \cite{celsus2020supercritical}. Moreover, we note that if we let $\lambda_0 = \frac{2i}{t}$ in \eqref{eq: lambda conditions}, the quadratic differential $\varpi_t^{(1)}$ defined in \eqref{eq: QD genus 1} coincides with the quadratic differential $\varpi_t^{(0)}$ defined in \eqref{eq: QD genus 0}.
	\end{remark}
	We also point out that a further continuation of the work in \cite{celsus2020supercritical,deano2014large} was carried out in \cite{ahmadjacobikissing}, where varying weight Kissing polynomials with a Jacobi type weight were considered. 
	
	Another natural generalization of the works \cite{celsus2020supercritical,deano2014large} is to allow $t$ to take on complex values. That is, instead of considering the polynomials defined in \eqref{eq: varying weight kissing polynomials} with $t \in \R_+\setminus\{t_c\}$, we consider monic polynomials
	\begin{equation*}
		\int_{-1}^1 p_n(z;s) z^k e^{-nf(z;s)}\, dz = 0, \qquad k= 0, 1, \dots, n-1, 
	\end{equation*}
	where $f(z;s)=sz$ and $s\in\C$ is arbitrary, as introduced in \eqref{eq: ortho p defn}. As stated at the beginning of this introduction, these polynomials will be investigated throughout this work, and we particularly concern ourselves with the asymptotics of the recurrence coefficients of these polynomials as $n\to\infty$.  
		
\section{Statement of Main Results}
	In this section, we discuss the necessary background on non-Hermitian orthogonality and state our main findings. 
	
	We first note that as everything in the integrand of \eqref{eq: ortho p defn} is analytic, Cauchy's Theorem gives us complete freedom to choose a contour connecting $-1$ and $1$ to integrate over. However, in light of the asymptotic behavior of the zeros of the polynomials as $n\to\infty$, it is expected that there exists a \say{correct} contour over which to take the integration in \eqref{eq: ortho p defn}. This contour should be the one on which the zeros of $p_n$ accumulate as $n\to\infty$. The study of this intuitive notion of the \say{correct} curve was started by Nuttall, who conjectured that in the case where the weight function does not depend on the degree $n$, the correct curve should be one of minimal capacity (see also \cite{nuttall1977orthogonal}). Nuttall's conjectures were then established rigorously by Stahl in  \cite{stahl1985extremal,stahl1986orthogonal}, where the correct curve was shown to satisfy a certain max-min variational problem. After Stahl's contributions, such curves became known in the literature as S-curves (where the S stands for \say{symmetric}) or curves which possess the S-property. 
	
	The attempt to adapt Stahl's work to account for orthogonality with respect to varying weights, as is considered in the present work, was first undertaken by Gonchar and Rakhmanov. In \cite{gonchar1989equilibrium}, Gonchar and Rakhmanov obtained the asymptotic zero distribution of a particular class of non-Hermitian orthogonal polynomials with varying weights, but took the existence of a curve with the S-property for granted. The question of the existence of S-curves was considered by Rakhmanov in \cite{rakhmanov2012orthogonal}, where he outlined a general max-min formulation for obtaining S-contours. In both the context of varying and non-varying weights, the probability measure which minimizes a certain energy functional on the S-curve (known as the equilibrium measure) governs the weak limit of the empirical counting measure for the zeros of the orthogonal polynomials. Indeed, the main technical differences between the subcritical case for the Kissing polynomials in \cite{deano2014large} and the supercritical case of \cite{celsus2020supercritical} is that for $t<t_c$, the equilibrium measure is supported on one analytic arc, whereas for $t>t_c$, the measure is supported on two arcs, as depicted in Figure~\ref{fig: sub and super zeros}. We shall see that this distinction between the one and two cut regimes will also play a fundamental role in the present analysis, as hinted at by Figure~\ref{fig: move across breaking curve}. This potential theoretic approach, known now as the Gonchar-Rakhmanov-Stahl (GRS) program, has been carried out in various scenarios, and we refer the reader to many excellent works on the subject \cite{aptekarev2015pade, deano2014large,kuijlaars2015s, martinez2011critical,martinez2016orthogonal,martinez2016critical,martinez2019critical,yattselev2018symmetric}. 
		\begin{figure}[h!]
		\centering
		\begin{subfigure}{.45\textwidth}
			\centering
			\includegraphics[width=0.65\linewidth]{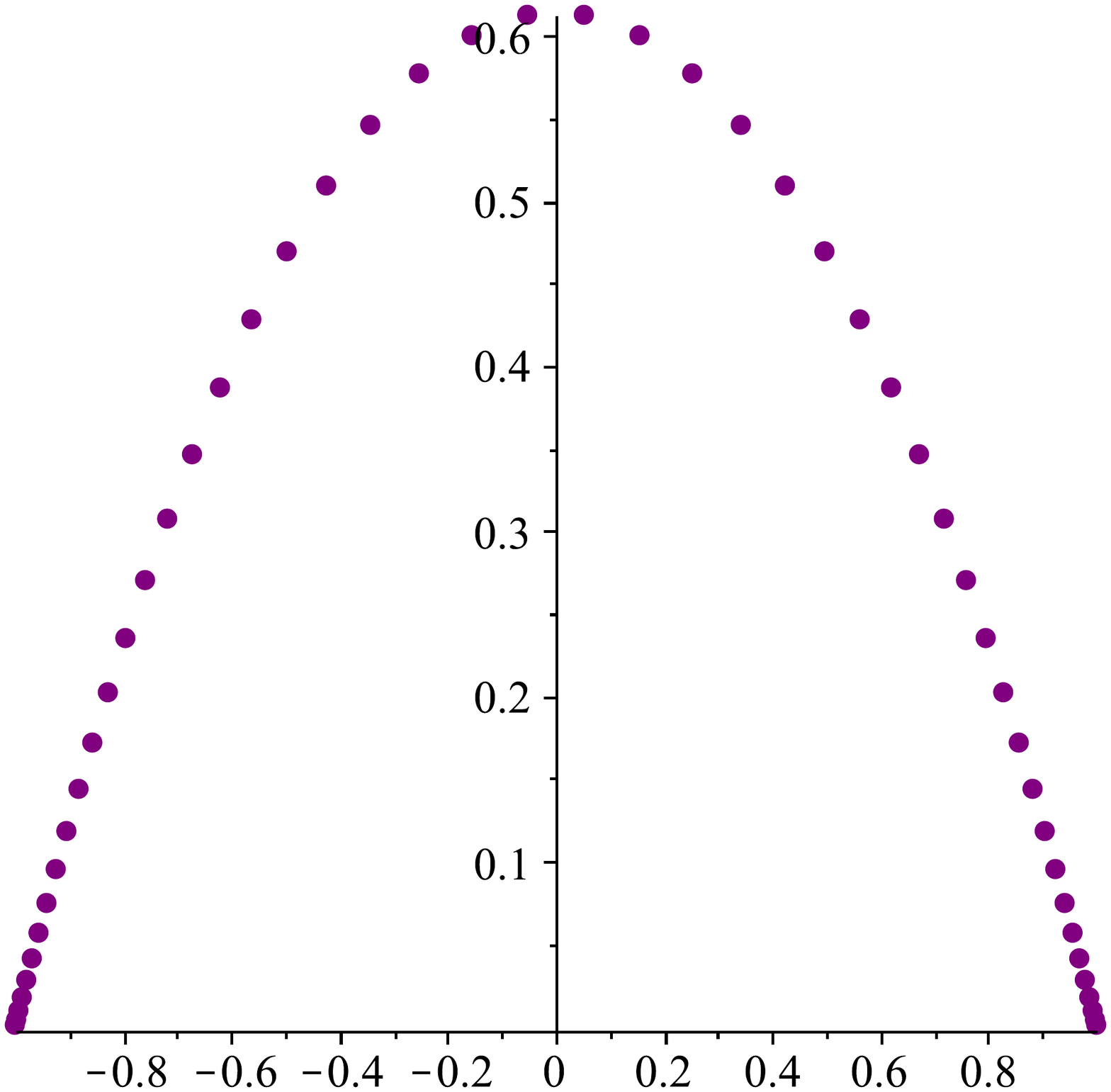}
			\caption{Zeros of $p_{50}^{-i}(z)$.}
			\label{fig:s=-i}
		\end{subfigure}%
		\qquad
		\begin{subfigure}{.45\textwidth}
			\centering
			\includegraphics[width=0.65\linewidth]{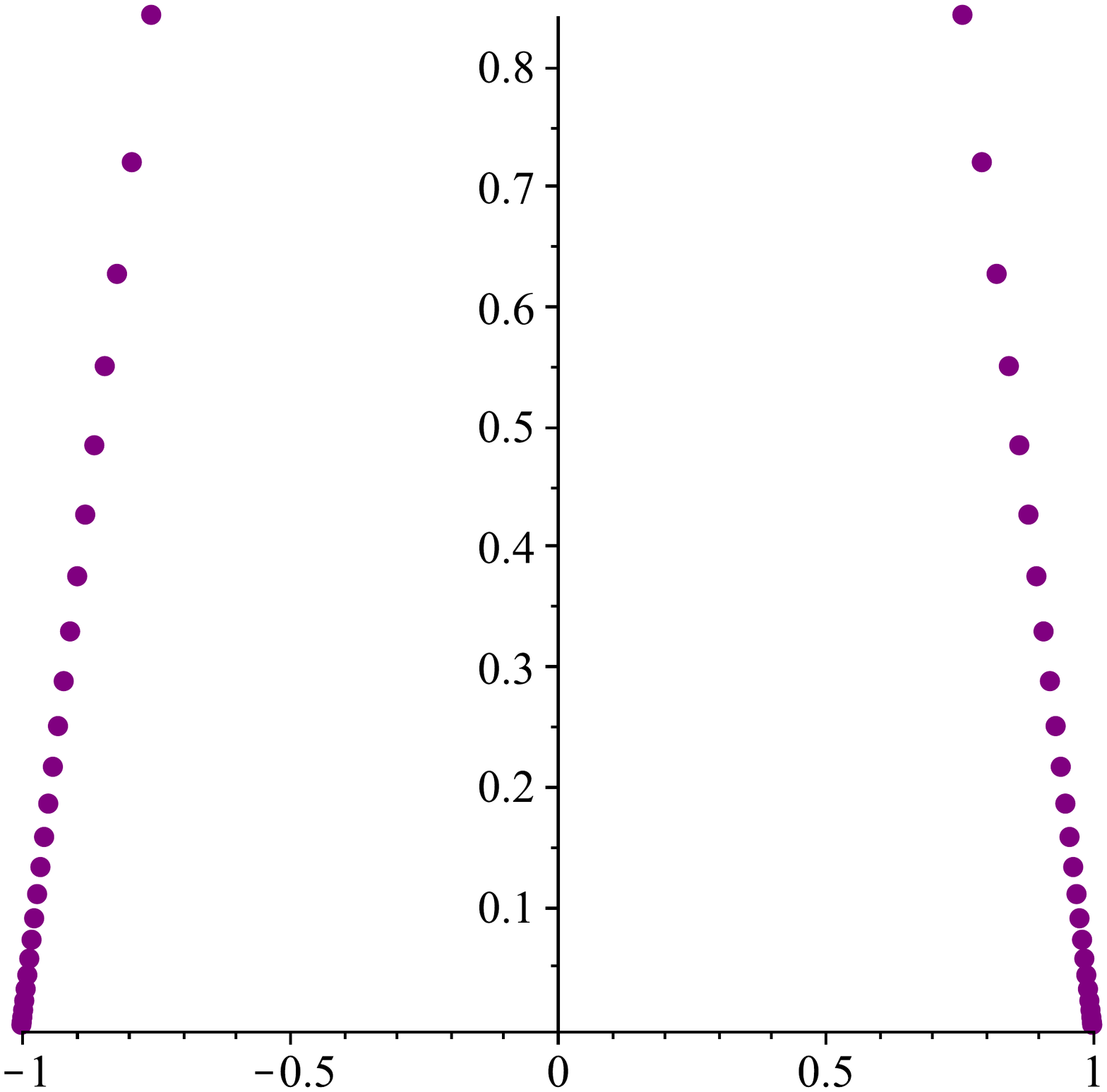}
			\caption{Zeros of $p_{50}^{-2i}(z)$.}
			\label{fig:s=-2i}
		\end{subfigure}
		\caption{Zeros of $p_{50}^{-t}$ defined in \eqref{eq: varying weight kissing polynomials} for $t=i<t_c$ and $t=2i>t_c$, where $t_c$ is the unique positive solution to \eqref{eq: tc def}.}
		\label{fig: sub and super zeros}
	\end{figure}
	\begin{figure}[h!]
		\centering
		\begin{subfigure}{.245\textwidth}
			\centering
			\includegraphics[width=\linewidth]{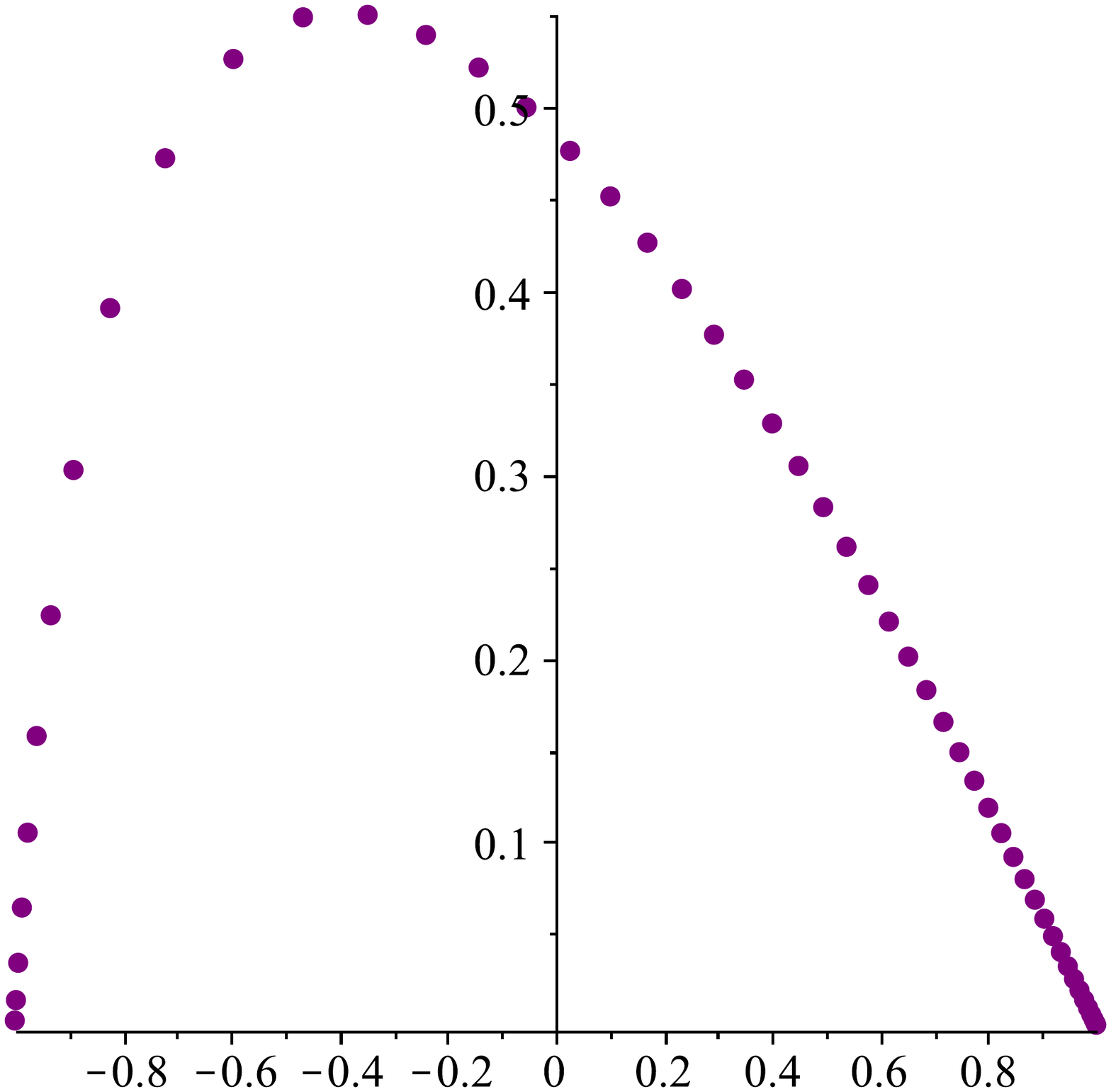}
			\caption{$s=-1-0.85i$.}
			\label{fig:s=085}
		\end{subfigure}%
		\begin{subfigure}{.245\textwidth}
			\centering
			\includegraphics[width=\linewidth]{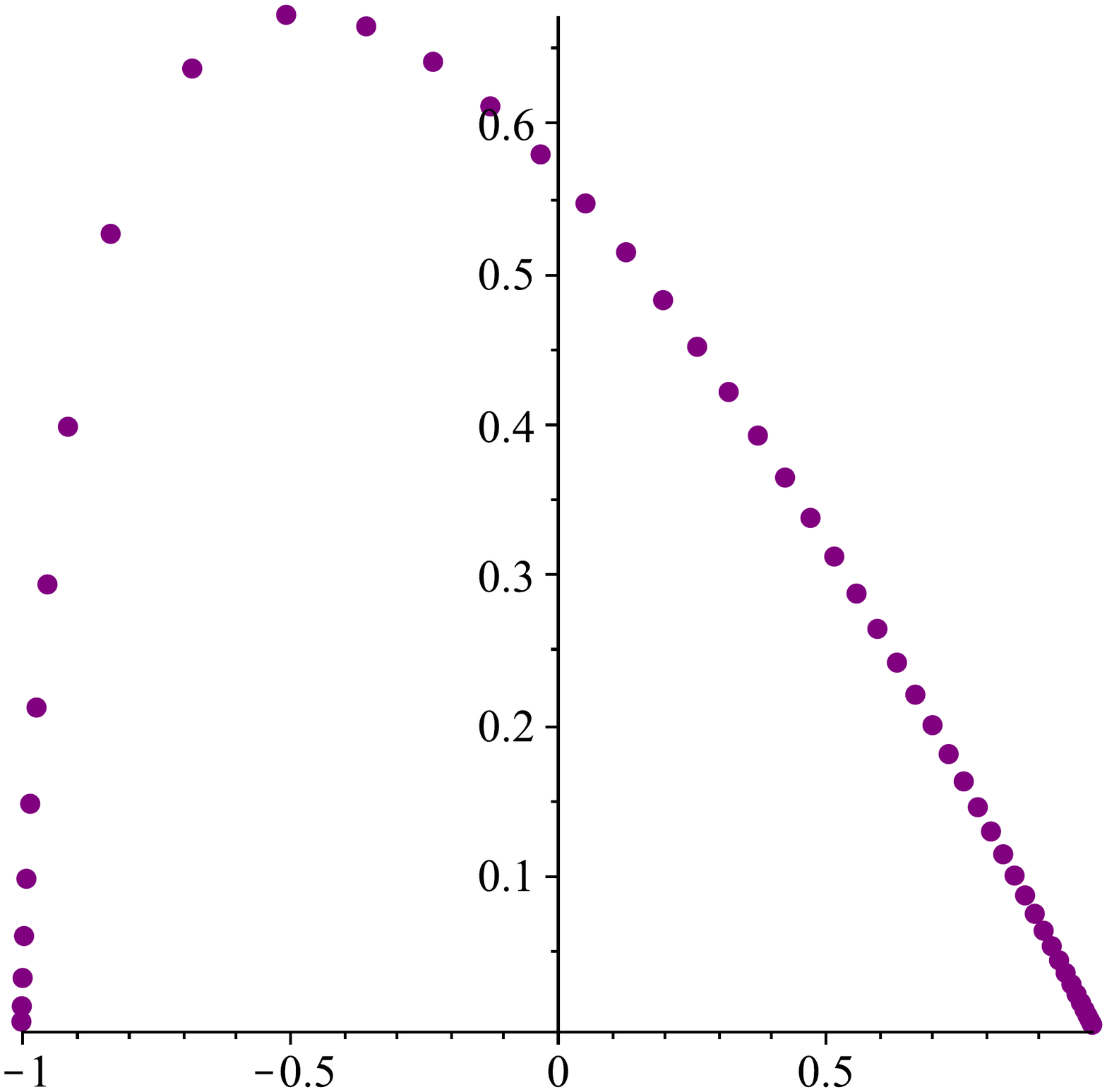}
			\caption{$s=-1-0.95i$.}
			\label{fig:s=095}
		\end{subfigure}
		\begin{subfigure}{.245\textwidth}
			\centering
			\includegraphics[width=\linewidth]{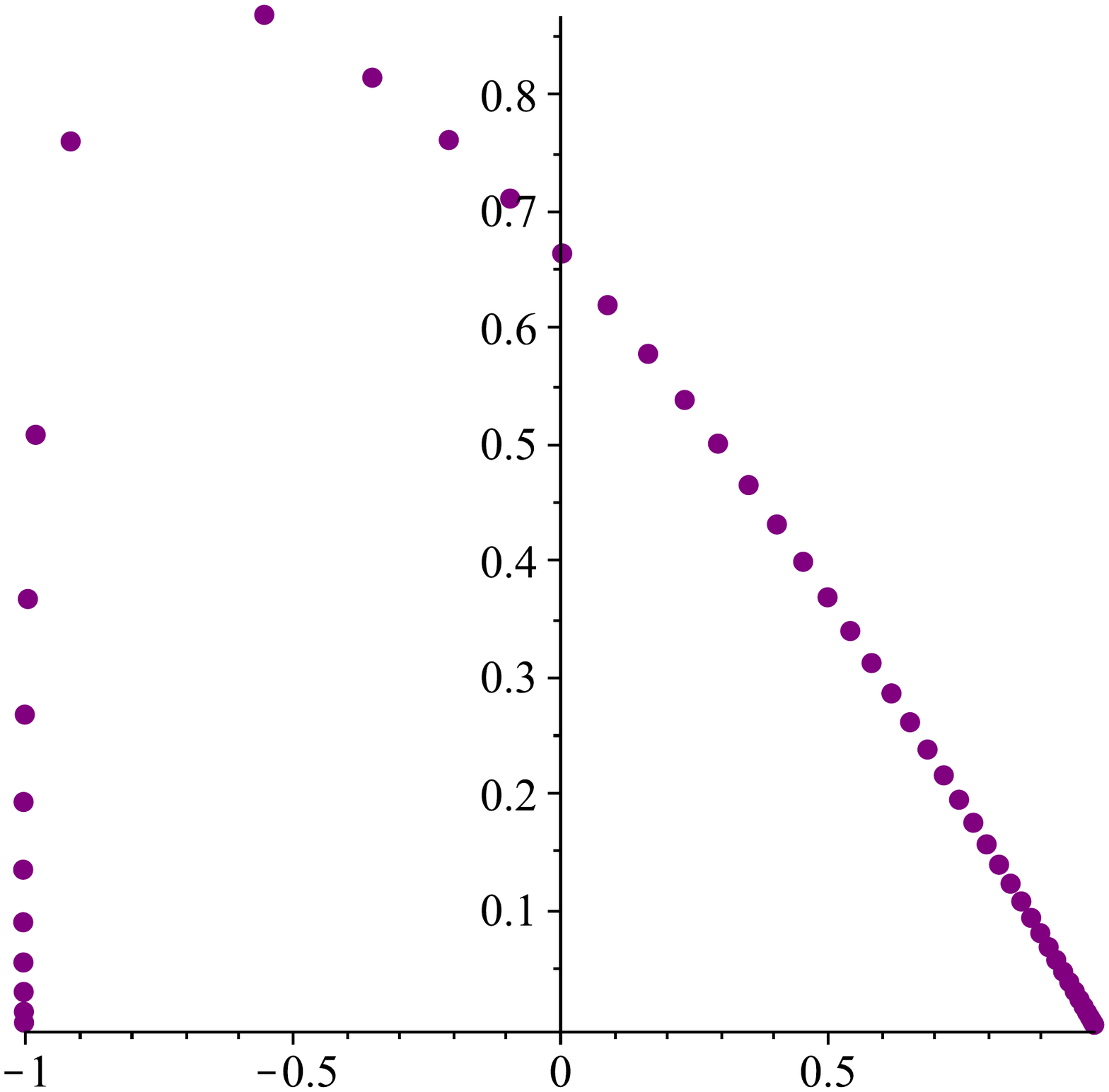}
			\caption{$s=-1-1.05i$.}
			\label{fig:s=105}
		\end{subfigure}%
		\begin{subfigure}{.245\textwidth}
			\centering
			\includegraphics[width=\linewidth]{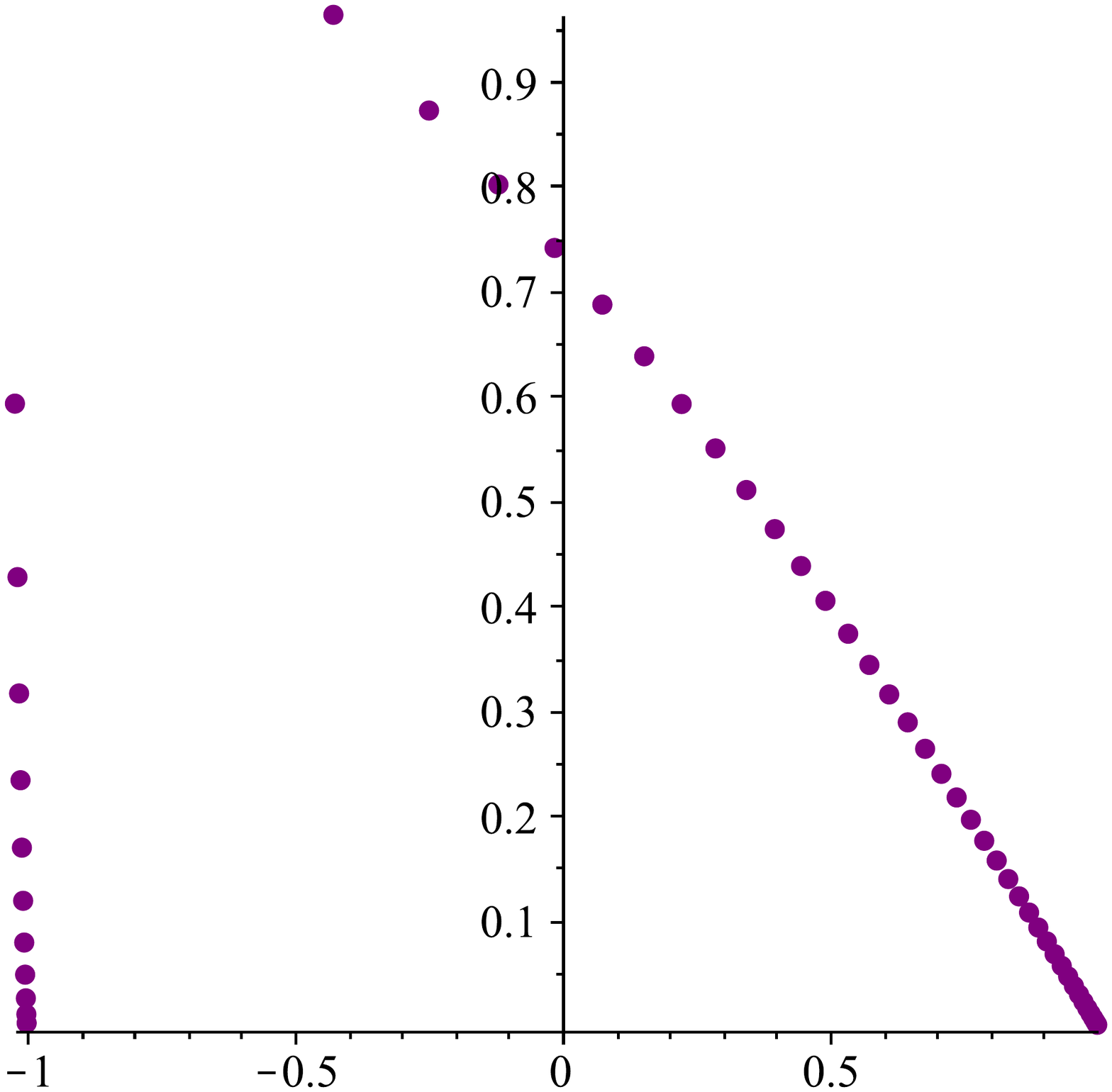}
			\caption{$s=-1-1.15i$.}
			\label{fig:s=115}
		\end{subfigure}
		\caption{Zeros of $p_{50}(z;s)$ defined in \eqref{eq: ortho p defn} as $s$ moves from $s=-1-0.85i\in\mathfrak{G}_0$ to $s=-1-1.15i\in\mathfrak{G}_1^-$.}
		\label{fig: move across breaking curve}
	\end{figure}

	Despite many successful applications of potential theory to the analysis of non-Hermitian orthogonal polynomials via the GRS program, we adopt an alternate viewpoint based on deformation techniques born from advances in the theory of random matrices and integrable systems. We will make heavy use of the technique known as \textit{continuation in parameter space}, first developed in the context of integrable systems (c.f. \cite{kamvissis2003semiclassical,tovbis2010nonlinear,tovbis2004semiclassical}), but which has only recently been applied in the field of orthogonal polynomials \cite{bertola2011boutroux,bertola2015asymptotics,bertola2016asymptotics}. In contrast to the GRS program, where one constructs a so-called  $g$-function as a solution to a certain variational problem, now one constructs a scalar function which solves a certain Riemann-Hilbert problem, which we call the \textit{$h$-function} or \textit{modified external field}. 

	We quickly note that as the weight function we consider, $\exp(-n f(z;s))$, depends on the parameter $s$, the scalar Riemann-Hilbert problem also depends on the parameter $s$. Importantly, the number of arcs over which this Riemann-Hilbert problem is posed, or equivalently the genus of the underlying Riemann surface, is also to be determined. Indeed, we will see that $h$-functions corresponding to Riemann surfaces of different genus lead to asymptotic expansions which possess markedly different behavior as $n\to\infty$. This difference is analogous to the difference in asymptotic behavior of the polynomials (and their recurrence coefficients) in the one cut and two cut cases, as described above for the GRS program. However, once one proves that for a specified genus and corresponding $s\in\C$ the scalar problem has a solution, one may continue with the process of steepest descent as will be outlined in Section~\ref{sec: steepest descent} below. 
	
	We will see that the $h$-functions constructed in \eqref{eq: subcritical h function} and \eqref{eq: supercritical h function} are the desired $h$-functions corresponding to genus 0 and 1 regimes, respectively, when $s\in i\R_-$. 
	
	In order to establish the global phase portrait for all $s\in\C$, we deform these solutions off of the imaginary axis using the technique of continuation in parameter space discussed above. During this deformation process, we will encounter curves in the parameter space which separate regions of different genera. These curves in parameter space are called \textit{breaking curves} and we denote the set of breaking curves, along with their endpoints, as $\mathfrak{B}$. For our purposes, breaking curves can only originate and terminate at what are called \textit{critical breaking points}, and we will see that the only critical breaking points we encounter in the present work are $s=\pm 2$. The description of the breaking curves in the parameter space forms our first main result.
	\begin{theorem}\label{thm: global phase portrait}
		There are two critical breaking points at $s=\pm 2$ and $\mathfrak{B}=\mathfrak{b}_{-\infty}\cup\mathfrak{b}_{\infty}\cup \mathfrak{b}_{+}\cup\mathfrak{b}_{-}\cup\{\pm 2\}$. Here, $\mathfrak{b}_{-\infty} = \left(-\infty, -2\right)$ and $\mathfrak{b}_{\infty} = (2, \infty)$. The breaking curve $\mathfrak{b}_{+}$ connects $-2$ and $2$ while remaining in the upper half plane, and the breaking curve $\mathfrak{b}_{-}$ is obtained by reflecting $\mathfrak{b}_{+}$ about the real axis. 
	\end{theorem}
	As seen in Figure~\ref{fig: breaking curve}, the set $\mathfrak{B}$ divides the parameter space into three connected components: $\mathfrak{G}_0$ and $\mathfrak{G}_1^\pm$. 
	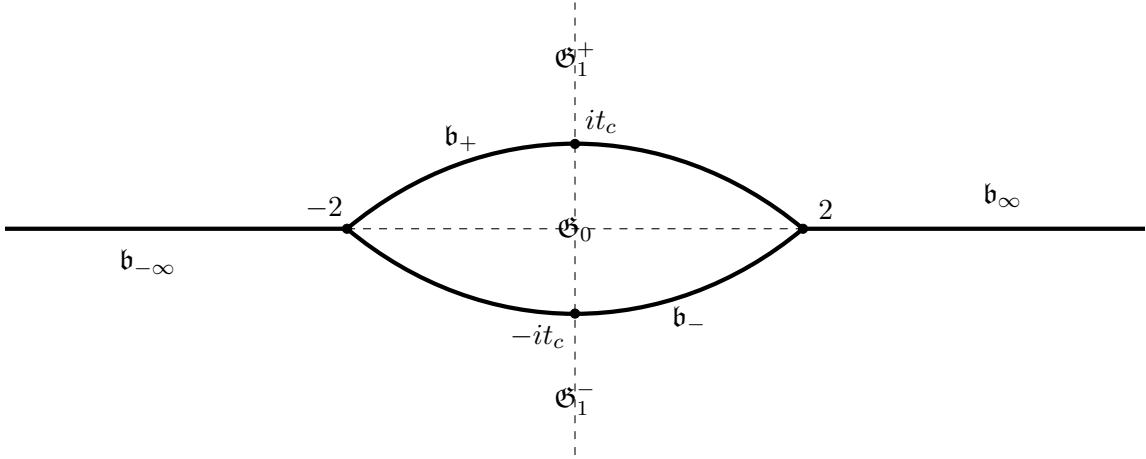
\begin{figure}[h]
		\centering
		\begin{tikzpicture}[scale=1.5]
		\draw [ultra thick] (-2,0) to (-5,0);
		\draw [ultra thick] (2,0) to (5, 0);
		\node [above] at (2.2, 0) {$2$};
		\node [above] at (-2.2, 0) {$-2$};
		\node [right] at (0, 0.95) {$it_c$};
		\node [left] at (0, -0.95) {$-it_c$};
		\node at (0,0) {$\mathfrak{G}_0$};
		\node at (0,1.5) {$\mathfrak{G}_1^+$};
		\node at (0,-1.5) {$\mathfrak{G}_1^-$};
		\node at (-3.75, -0.3) {$\mathfrak{b}_{-\infty}$};
		\node at (3.75,0.3) {$\mathfrak{b}_{\infty}$};
		\node at (-1,0.8) {$\mathfrak{b}_{+}$};
		\node at (1,-0.8) {$\mathfrak{b}_{-}$};
		
		\draw [ultra thick] (-2,0) to [bend right=40] (2, 0);
		\draw [ultra thick] (-2,0) to [bend left=40] (2, 0);
		
		\draw [dashed] (-2,0) to (2,0);
		\draw [dashed] (0,-2) to (0,2);
		
		\draw [fill] (2,0) circle [radius=0.04];
		\draw [fill] (-2,0) circle [radius=0.04];
		\draw [fill] (0,.75) circle [radius=0.04];
		\draw [fill] (0,-.75) circle [radius=0.04];
		\end{tikzpicture}
		\caption{Definitions of the regions $\mathfrak{G}_0$ and $\mathfrak{G}_1^\pm$ in the $s$-plane. The set $\mathfrak{B}$ is drawn in bold. The regular breaking points $\pm i t_c$ are indicated on the breaking curves $\mathfrak{b}^\pm$, where we recall that $t_c$ was defined in \eqref{eq: tc def}.}
		\label{fig: breaking curve}
	\end{figure}
	We will see that the region $\mathfrak{G}_0$ corresponds to the genus $0$ region, whereas the regions $\mathfrak{G}_1^\pm$ correspond to genus $1$ regions. 
	
	Having determined the description of the set $\mathfrak{B}$, we will be able to deduce asymptotic formulas for the recurrence coefficients for the orthogonal polynomials defined in \eqref{eq: ortho p defn} for all $s\in\C\setminus\mathfrak{B}$ via deformation techniques. We quickly digress to discuss notation before stating these results. We first introduce monic polynomials, $p_n^N(z;s)$ which satisfy the following orthogonality conditions.
	\begin{equation}\label{eq: ortho p defn N n} 
	\int_{-1}^1 p_n^N(z;s) z^k e^{-Nf(z;s)}\, dz = 0, \qquad k= 0, 1, \dots, n-1,
	\end{equation}
	where $N$ is a fixed integer. Note that for each $N\in\N$, we have a family of polynomials $\{p_n^N(z;s)\}_{n=0}^\infty$. The polynomials that we consider in \eqref{eq: ortho p defn} are given by $p_n(z;s) = p_n^n(z;s)$; that is, we consider the polynomials along the diagonal where $N=n$.  Now, provided the polynomials exist for the appropriate values for $n, N$, and $s$, they satisfy the following three term recurrence relation
	\begin{equation}\label{eq: ttrr}
		z p_n^N(z;s) = p_{n+1}^N(z;s) + \alpha_n^N(s) p_n^N(z;s) +\beta_{n}^N(s) p_{n-1}^N(z;s).
	\end{equation}
	In the present work, we concern ourselves with the situation $N=n$, and for sake of notation we set $\alpha_n := \alpha_n^n$ and $\beta_n := \beta_n^n$. It should be stressed that the polynomials $p_{n-1}$, $p_n$ and $p_{n+1}$ do \textit{not} satisfy the recurrence relation \eqref{eq: ttrr}. We now state our second result, on the asymptotics of the recurrence coefficients in the region $\mathfrak{G}_0$. 
	\begin{theorem}\label{thm: recurrence coeffs genus 0}
		Let $s\in \mathfrak{G}_0$. Then the recurrence coefficients $\alpha_n$ and $\beta_n$ exist for large enough $n$, and they satisfy, as $n\to\infty$,
		\begin{equation}
				\alpha_n(s) = \frac{2s}{(s^2-4)^2}\frac{1}{n^2} + \mathcal{O}\left(\frac{1}{n^3}\right), \qquad \beta_n(s) = \frac{1}{4} + \frac{s^2+4}{4(s^2-4)^2}\frac{1}{n^2}  + \mathcal{O}\left(\frac{1}{n^4}\right).
			\end{equation}
	\end{theorem}

	As mentioned above, for $s\in \mathfrak{G}_1^\pm$, the underlying Riemann surface has genus $1$. Indeed, the Riemann surface corresponds to the algebraic equation $\xi^2 = Q(z;s)$, where $Q$ is a rational function, and we take the branch cuts for the Riemann surface on two arcs - one connecting $1$ to $\lambda_0(s)$, labeled $\gamma_{m ,0}$, and the other connecting $-1$ to $\lambda_1(s)$,  labeled $\gamma_{m ,1}$, where $\lambda_0$ and $\lambda_1$ will be determined. Moreover, for $s\in\mathfrak{G}_1^\pm$, the asymptotics of the recurrence coefficients will depend on theta functions on our Riemann surface. These theta functions will be used to construct functions $\mathcal{M}_1(z, k)$ and $\mathcal{M}_2(z, k)$, along with a constant $d$, whose precise descriptions we provide in Section~\ref{sub: Genus 1 Global Parametrix}. The functions $\mathcal{M}_{1, n}(z, d)  \equiv \mathcal{M}_{1}(z, d)$ and  $\mathcal{M}_{2, n}(z, d)\equiv \mathcal{M}_{2}(z, d)$ are holomorphic in $\C \setminus (\gamma_{m, 0} \cup \gamma_{m, 1} \cup \gamma_{c, 1})$, where $\gamma_{c, 1}$ is a to be determined curve connecting $\lambda_0(s)$ to $\lambda_1(s)$, and have at most one simple zero there. Furthermore, for $N=n$ and given $\epsilon>0$, we will need to consider asymptotic results on a subsequence $\N(s, \epsilon)$, whose precise definition we defer to Section \ref{sub: Genus 1 Global Parametrix}. However, to make use of this subsequence, we need to know that the cardinality of the set $\N(s, \epsilon)$ is infinite, which we prove in Lemma \ref{seq-prop}.
	These functions $\mathcal{M}_{1, n}(z, d)$ and $\mathcal{M}_{2, n}(z, d)$ arise in the asymptotics of the recurrence coefficients for $s\in\mathfrak{G}_1^\pm$, which we state below. 
	\begin{theorem}\label{thm: recurrence coeffs genus 1}
		Let $s\in \mathfrak{G}_1^\pm$ and $n \in \N(s, \epsilon)$. Then the recurrence coefficients $\alpha_n$ and $\beta_n$ exist for large enough $n$, and they satisfy, as $n\to\infty$, 
		\begin{equation}
			\alpha_n(s) = \frac{\lambda_1^2(s)-\lambda_0^2(s)}{4+2\lambda_0(s)-2\lambda_1(s)} + \frac{d}{dz}\left[\log\mathcal{M}_{2, n}(1/z,d)-\log\mathcal{M}_{2, n}(1/z,-d)\right]\Big|_{z=0}+\mathcal{O}_{\epsilon}\left(\frac{1}{n}\right)
		\end{equation}
		and
		\begin{equation}
			\beta_n(s) = \frac{(2+\lambda_0(s)-\lambda_1(s))^2}{16}\frac{\mathcal{M}_{1, n}(\infty, -d)\mathcal{M}_{2, n}(\infty, d)}{\mathcal{M}_{1, n}(\infty, d)\mathcal{M}_{2, n}(\infty, -d)}  + \mathcal{O}_{\epsilon}\left(\frac{1}{n}\right),
		\end{equation}
	\end{theorem}
	Above, the notation $f(n)=\mathcal{O}_\epsilon(1/n)$ indicates that there exists a constant which depends only on $\epsilon$, $M=M(\epsilon)$, such that $\left|f(n)\right|\leq M/n$ for large enough $n$. We recall that the parameter $\epsilon$ is used to define the set of valid indices, $\N(s,\epsilon)$, along which we take limits. Having determined the asymptotics of the recurrence coefficients of the polynomials in \eqref{eq: ortho p defn} when $s\in\C\setminus\mathfrak{B}$, our final two results recover these asymptotics when $s\in\mathfrak{B}$.
	
	As seen in Theorem~\ref{thm: global phase portrait}, the breaking curves $\mathfrak{b}_{-\infty}$ and $\mathfrak{b}_{\infty}$ are the intervals $(-\infty,-2)$ and $(2, \infty)$, respectively. The theory of orthogonal polynomials with respect to real weights, varying or otherwise, has been written about extensively in the literature, most notably from the viewpoint of potential theory. In particular, the results of Deift, Kriecherbauer, and McLaughlin in \cite{deift1998new} can be applied in conjunction with the GRS program to show that the empirical zero counting measure of the polynomials in \eqref{eq: ortho p defn} converge to a continuous measure supported on the interval $[-1,1]$ as $n\to\infty$, when $s\in\R$ and $|s|<2$. The results of \cite{deift1998new} can also be used to show that the corresponding limit measure is supported on $[-1,a)$ for some $a<1$ when $s>2$. Similarly, one also has that this measure is supported on $(b,1]$ for some $b>-1$ when $s\in\R$ is such that $s<-2$. The difference in the support of the limiting measure when $|s|>2$ and $|s|<2$ is also of interest in random matrix theory, and occurs when the soft edge meets the hard edge (see the work of Claeys and Kuijlaars \cite{claeys2008universality}). The asymptotics of orthogonal polynomials corresponding to the case $s\in\mathfrak{b}_{\infty}\cup\mathfrak{b}_{-\infty}$ follow from \cite[Theorem 2]{aptekarev2002sharp}. From the viewpoint of the present work, the transitions at $s=\pm 2$ can be seen to come from the fact these are critical breaking points. 
	
	As the case where $s\in\R\cap\mathfrak{B}$ has been extensively studied, we next consider the asymptotic behavior of the recurrence coefficients as we approach a regular breaking point which is not on the real line. More precisely, we let $s_*$ be a regular breaking point in $\mathfrak{b}_+\cup\mathfrak{b}_-$ and we let $s$ approach $s_*$ as
	\begin{equation}\label{def: ds regular breaking point}
		s= s_* + \frac{L_1}{n},
	\end{equation}
	where $L_1\in\C$ is such that $s=s(n)\in\mathfrak{G}_0$ for large enough $n$. The scaling limit \eqref{def: ds regular breaking point} is referred to as the double scaling limit, as it describes the behavior of the polynomials as both $n\to\infty$ and $s\to s_*$. This formulation then leads to the following description of the recurrence coefficients in the aforementioned double scaling limit.  
	
	\begin{theorem}\label{thm: rec coeff double scaling regular}
		Let $s^*\in\mathfrak{b}_+\cup\mathfrak{b}_-$ and let $s\to s_*$ as described in \eqref{def: ds regular breaking point}. Then the recurrence coefficients exist for large enough $n$, and they satisfy
		\begin{subequations}
			\begin{equation}
			\alpha_n(s)=\frac{\delta_n\left(2-\sqrt{4-s^2_*}\right)\sqrt{4-s_*^2}}{\sqrt{\pi}s_*^3 n^{1/2}}\frac{1}{n^{1/2}}
			-
			\frac{2\delta_n^2\left(2-\sqrt{4-s_*^2}\right)^2}{\pi s_*^5}\frac{1}{n}+\mathcal{O}\left(\frac{1}{n^{3/2}}\right)
			\end{equation}
			and
			\begin{equation}
			\beta_n(s) = \frac{1}{4} + \frac{\sqrt{4-s_*^2}}{2\sqrt{\pi}s_*^2}\frac{1}{n^{1/2}}-\frac{\delta_n^2}{2\pi s_*^2}\frac{1}{n} + \mathcal{O}\left(\frac{1}{n^{3/2}}\right),
			\end{equation}
		\end{subequations}
		as $n\to\infty$, where 
		\begin{equation}
			\delta_n =\delta_n(L_1)= e^{-in\kappa}\exp\left(L_1\left(\frac{4}{s_*^2}-1\right)^{1/2}\right), \qquad \kappa\in\R,
		\end{equation}
		and $\kappa$ is a constant given in \eqref{kappa}.
	\end{theorem}
	Note above that
	\begin{equation}
		\left|\delta_n\right| =  \exp\left(\Re \left[L_1\left(\frac{4}{s_*^2}-1\right)^{1/2}\right]\right),
	\end{equation}
	as $\kappa\in\R$ and that the recurrence coefficients decay at a rate of $n^{1/2}$. In particular, the modulus of $\delta_n$ does not depend on $n$. 
	
	Now, we are just left with investigating the behavior of the recurrence coefficients for $s$ near the critical breaking points $s=\pm 2$. For brevity, we focus just on the case $s=2$, although we note that the case $s=-2$ can be handled similarly via reflection, as $\exp(-nf(z;-s))=\exp(-nf(-z;s))$. To state our results, we consider the Painlev\'e II equation \cite[Chapter~32]{olver2016nist}:
	\begin{equation}\label{eq: PII}
		q''(x) = x q(x) + 2 q^3(x) -\alpha, \qquad \alpha\in\C.
	\end{equation}
	Next, let $q=q(w)$ be the generalized Hastings-McLeod solution to Painlev\'e II with the parameter $\alpha=1/2$, which is characterized by the following asymptotic behavior
	\begin{align}
		q(x) = \begin{cases}
		\displaystyle \sqrt{-\frac{x}{2}} + \mathcal{O}\left(\frac{1}{x}\right), \qquad & x\to -\infty \\[2mm]
		\displaystyle \frac{1}{2x} + \mathcal{O}\left(\frac{1}{x^4}\right) & x\to\infty.
		\end{cases}
	\end{align}
	
	In order to study the asymptotics of the recurrence coefficients as $s\to 2$, we take $s$ in a double scaling regime near this critical point as
	\begin{equation}\label{def: ds crit breaking point}
		s = 2 + \frac{L_2}{n^{2/3}},
	\end{equation}
	where we impose that $L_2<0$. This leads us to our final main finding. 
	\begin{theorem}\label{thm: rec coeff double scaling critical}
		Let $s\to 2$ as described in \eqref{def: ds crit breaking point}. Then the recurrence coefficients exist for large enough $n$, and they satisfy
			\begin{equation}
			\alpha_n(s)= -\frac{q^2(-L_2)+q'(-L_2)}{n^{2/3}} + \mathcal{O}\left(\frac{1}{n}\right), \qquad \beta_n(s) = \frac{1}{4} -\frac{q^2(-L_2)+q'(-L_2)}{2}\frac{1}{n^{2/3}} + \mathcal{O}\left(\frac{1}{n}\right),
			\end{equation}
		as $n\to\infty$, where $q$ is the generalized Hastings-McLeod solution to Painlev\'e II with parameter $\alpha=1/2$. Furthermore, the function $U(w)=q^2(w)+q'(w)$ is free of poles for $w\in\mathbb{R}$.
	\end{theorem}

	Plots of the recurrence coefficients are given in Figures~\ref{fig: rec coeff} and \ref{fig: rec coeff imag}, and should be compared with Theorems~\ref{thm: recurrence coeffs genus 0} and \ref{thm: rec coeff double scaling critical}.
	\begin{figure}[h!]
		\centering
		\begin{subfigure}{.245\textwidth}
			\centering
			\includegraphics[width=\linewidth]{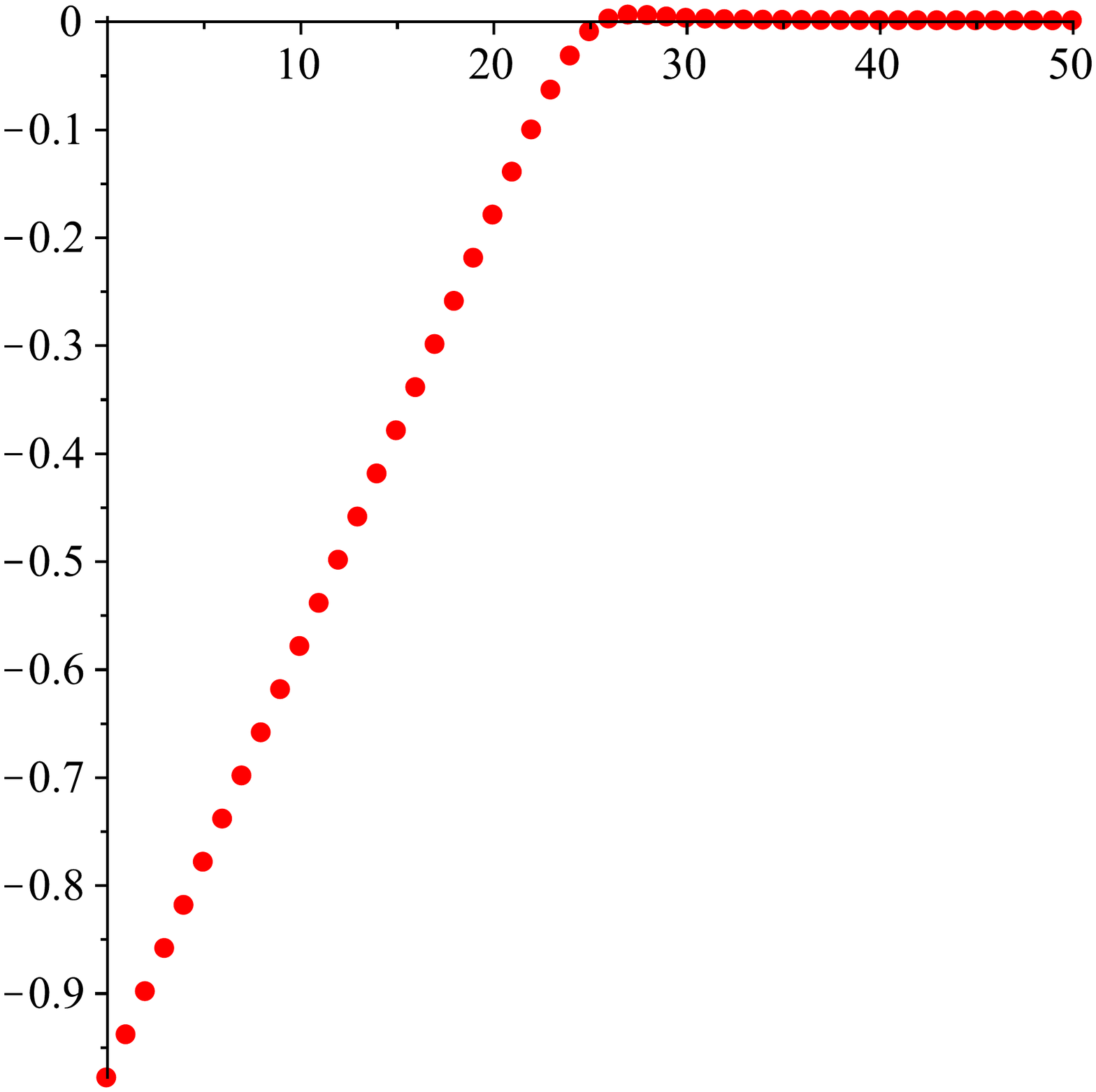}
			\caption{$\alpha_n(1)$.}
			\label{fig:s1alpha}
		\end{subfigure}%
		\begin{subfigure}{.245\textwidth}
			\centering
			\includegraphics[width=\linewidth]{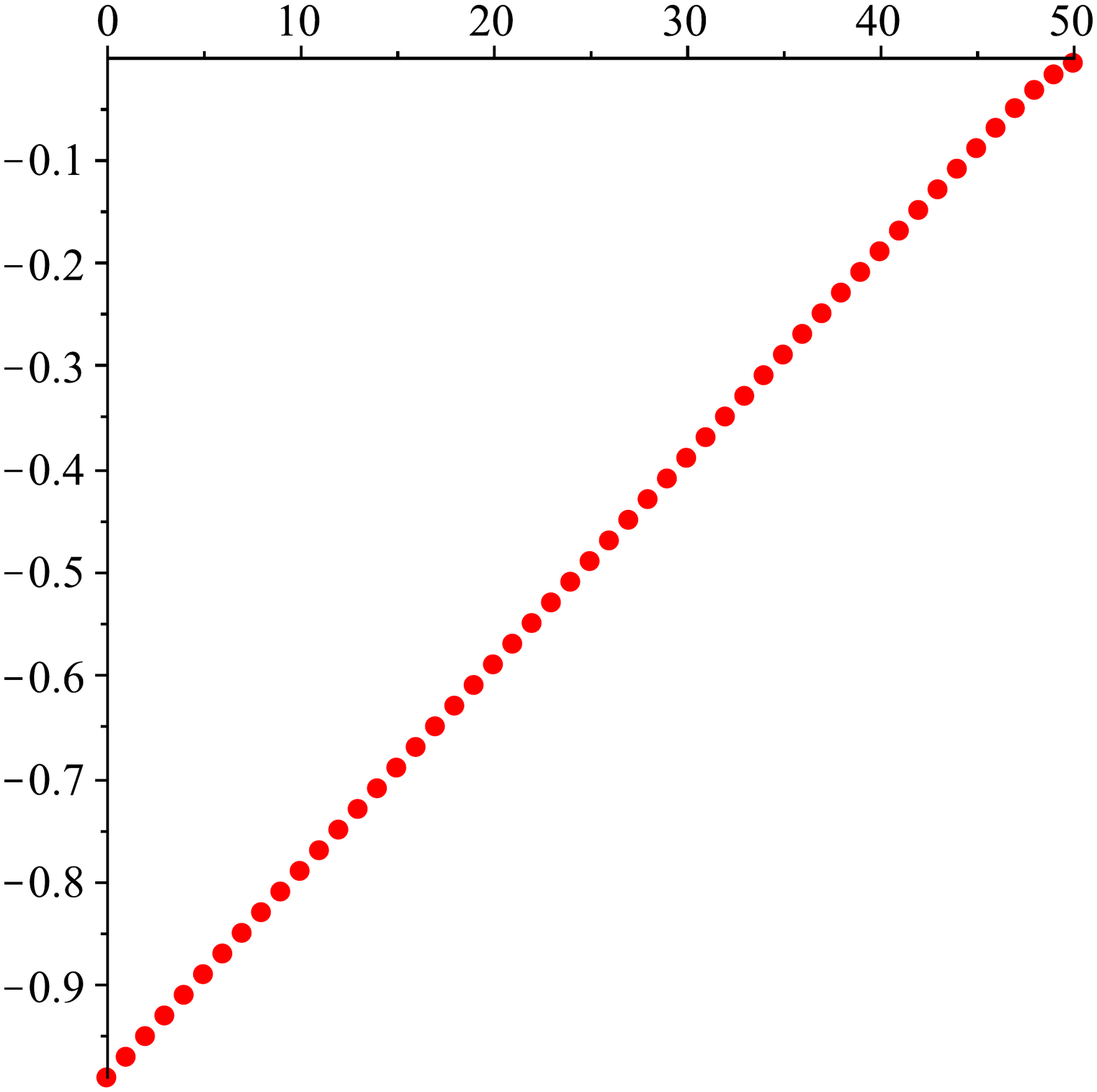}
			\caption{$\alpha_n(2)$.}
			\label{fig:s2lpha}
		\end{subfigure}
		\begin{subfigure}{.245\textwidth}
			\centering
			\includegraphics[width=\linewidth]{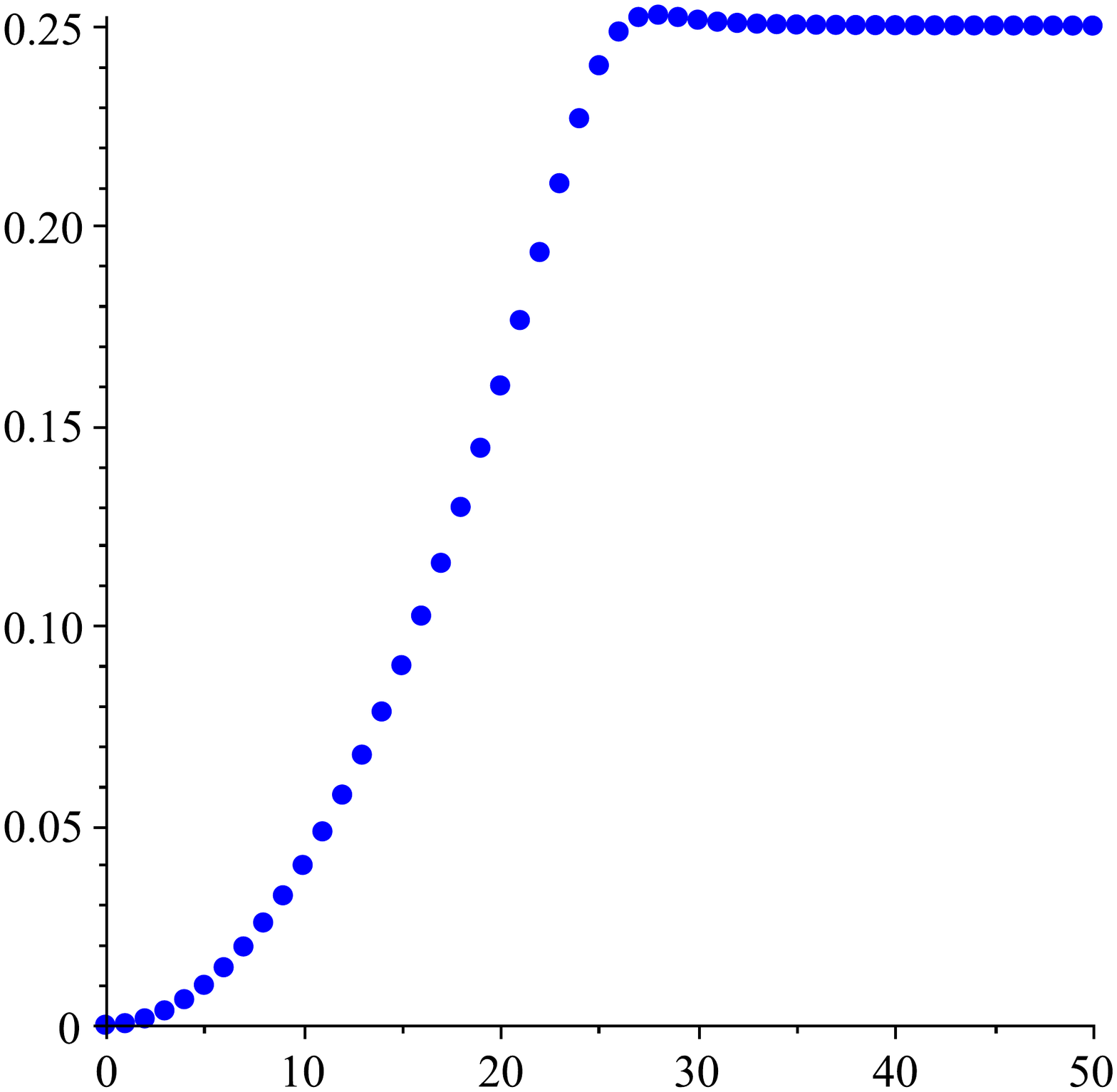}
			\caption{$\beta_n(1)$.}
			\label{fig:s1beta}
		\end{subfigure}%
		\begin{subfigure}{.245\textwidth}
			\centering
			\includegraphics[width=\linewidth]{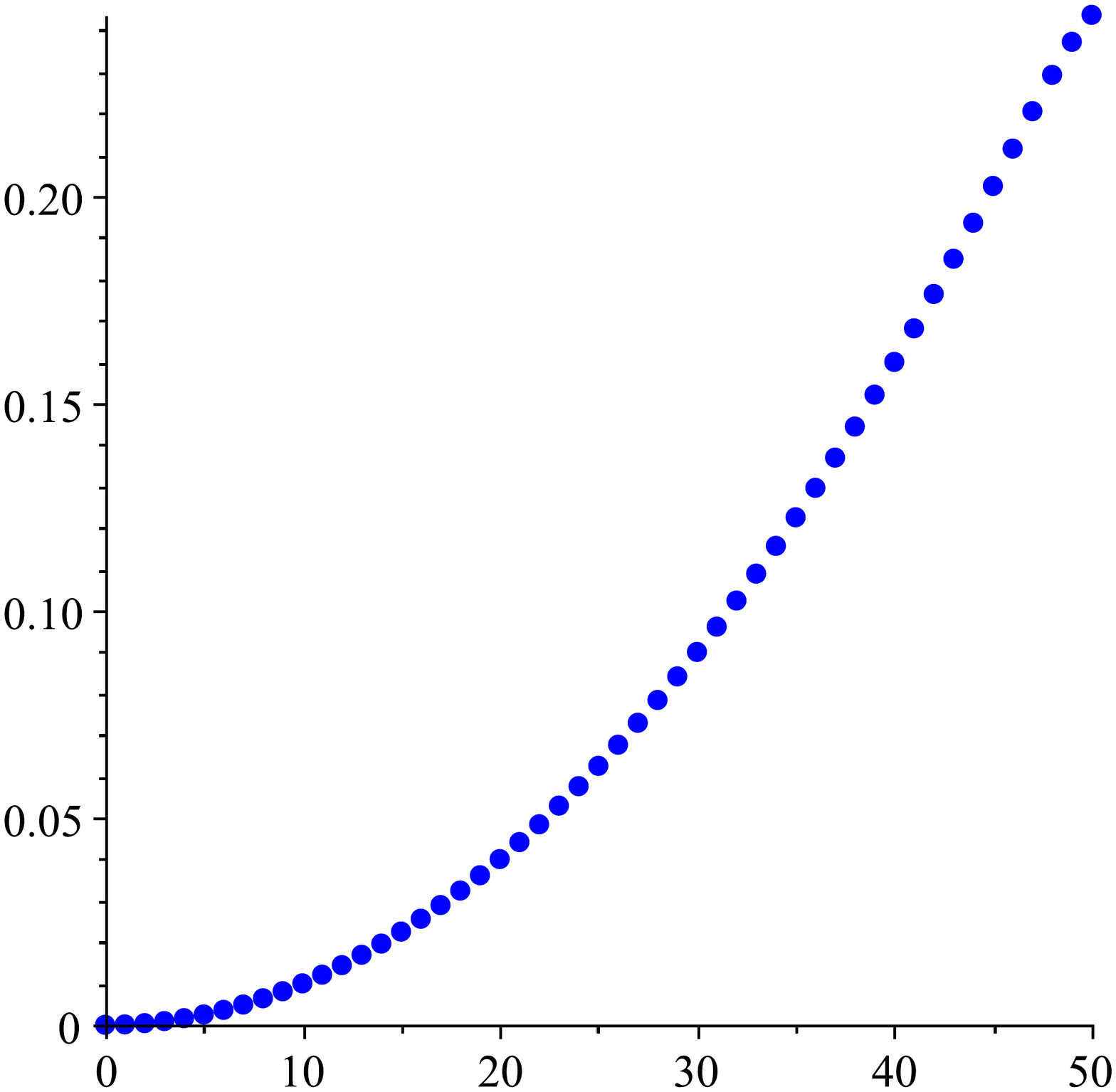}
			\caption{$\beta_n(2)$.}
			\label{fig:s2beta}
		\end{subfigure}
		\caption{Plots of $\alpha_n(s)$ and $\beta_n(s)$ for $n=0, \dots, 50$, with $s=1, 2$.}
		\label{fig: rec coeff}
	\end{figure}
	\begin{figure}[h!]
		\centering
		\begin{subfigure}{.245\textwidth}
			\centering
			\includegraphics[width=\linewidth]{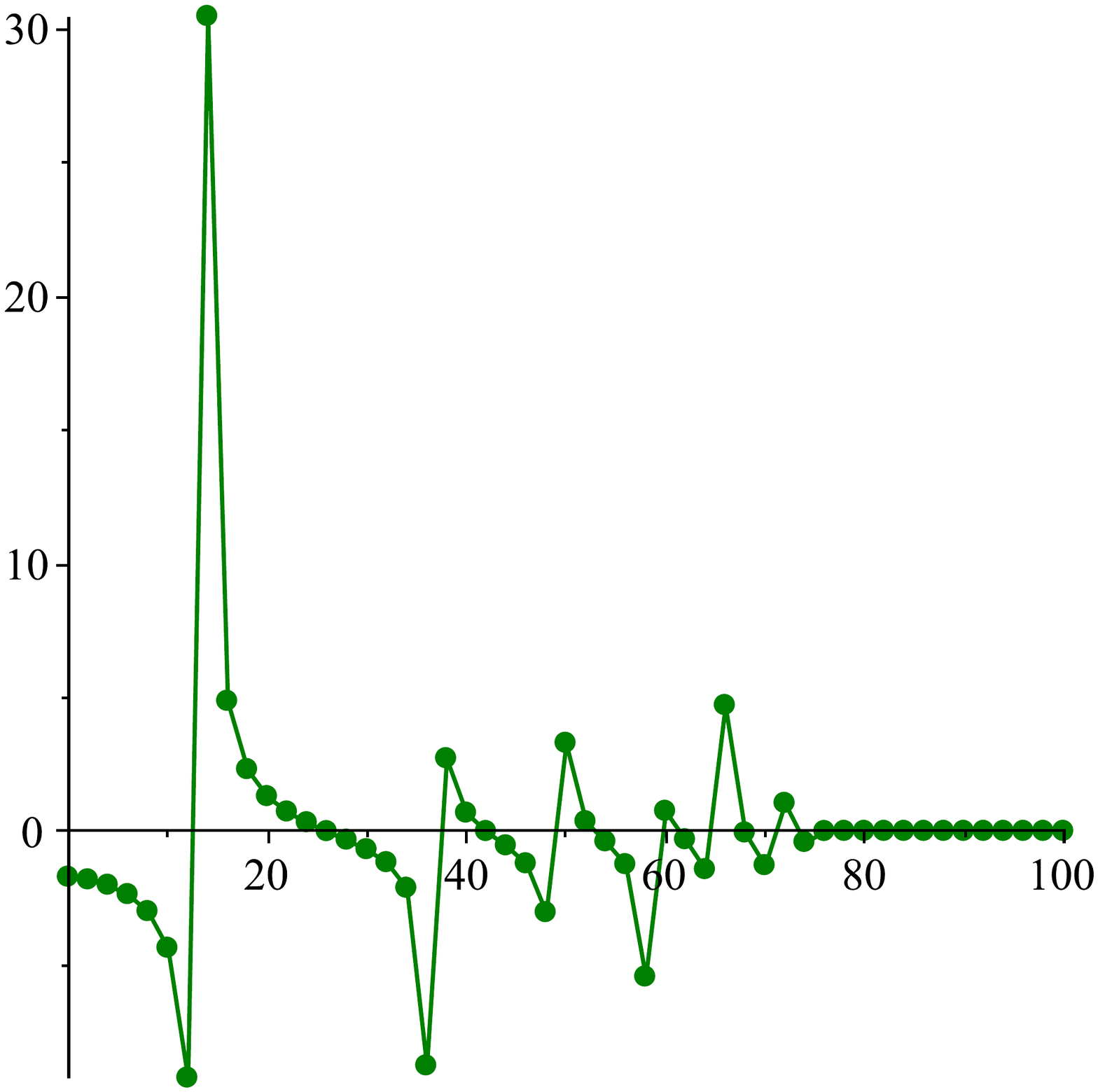}
			\caption{$\Im\alpha_{2n}(i)$.}
			\label{fig:sialpha}
		\end{subfigure}%
		\begin{subfigure}{.245\textwidth}
			\centering
			\includegraphics[width=\linewidth]{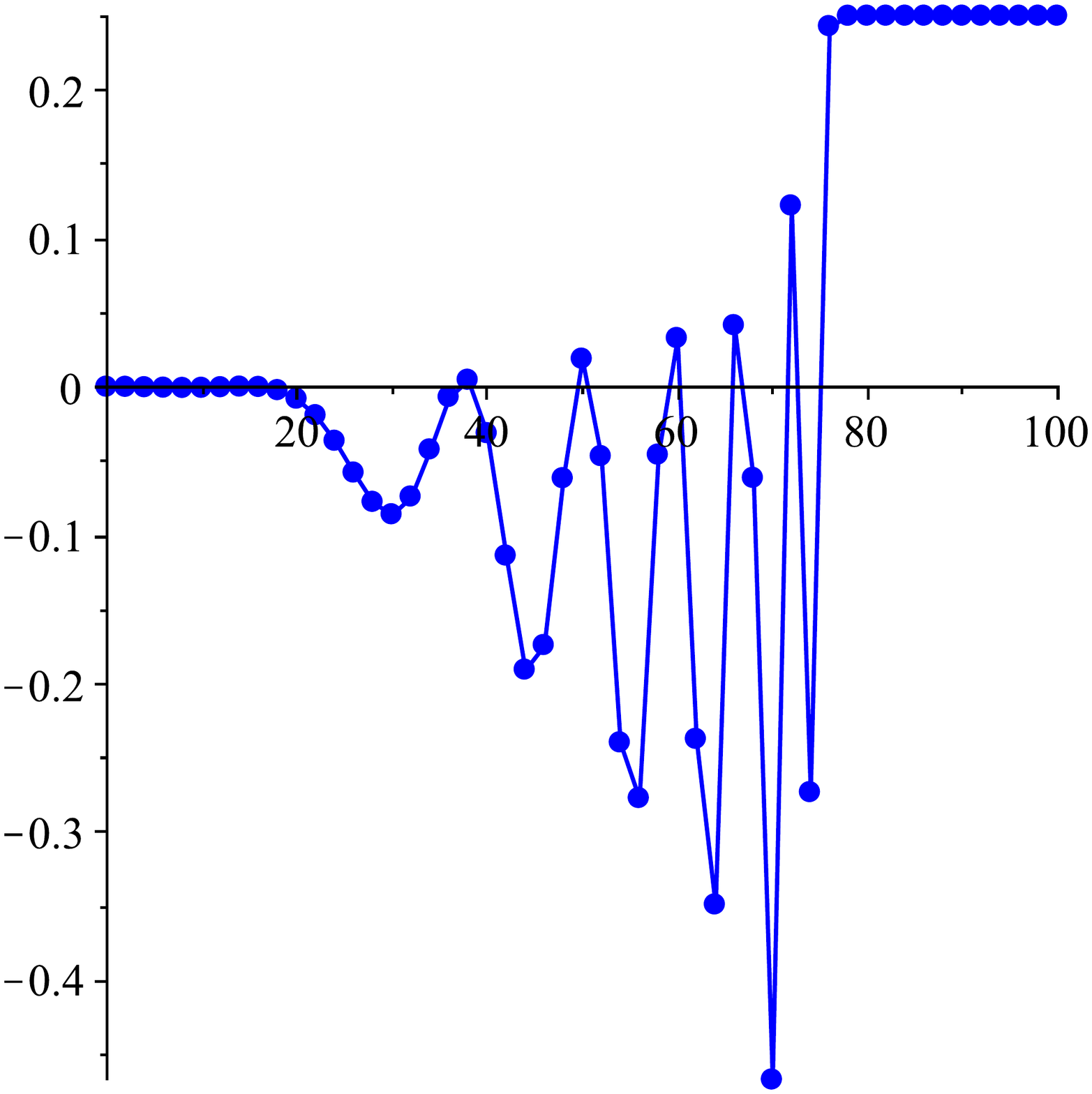}
			\caption{$\Re \beta_{2n}(i)$.}
			\label{fig:sibeta}
		\end{subfigure}
		\begin{subfigure}{.245\textwidth}
			\centering
			\includegraphics[width=\linewidth]{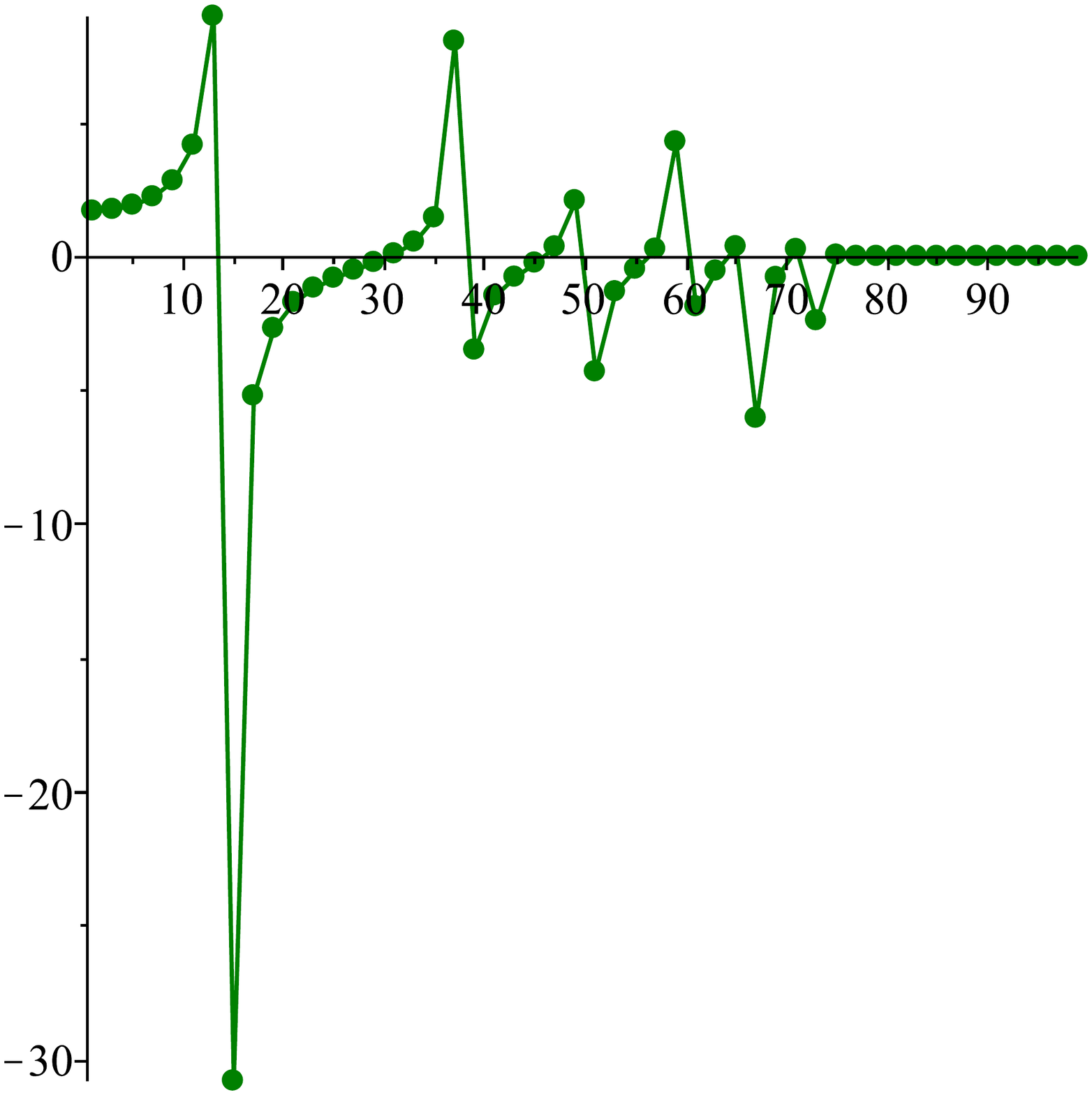}
			\caption{$\Im\alpha_{2n+1}(i)$.}
			\label{fig:sialphaodd}
		\end{subfigure}%
		\begin{subfigure}{.245\textwidth}
			\centering
			\includegraphics[width=\linewidth]{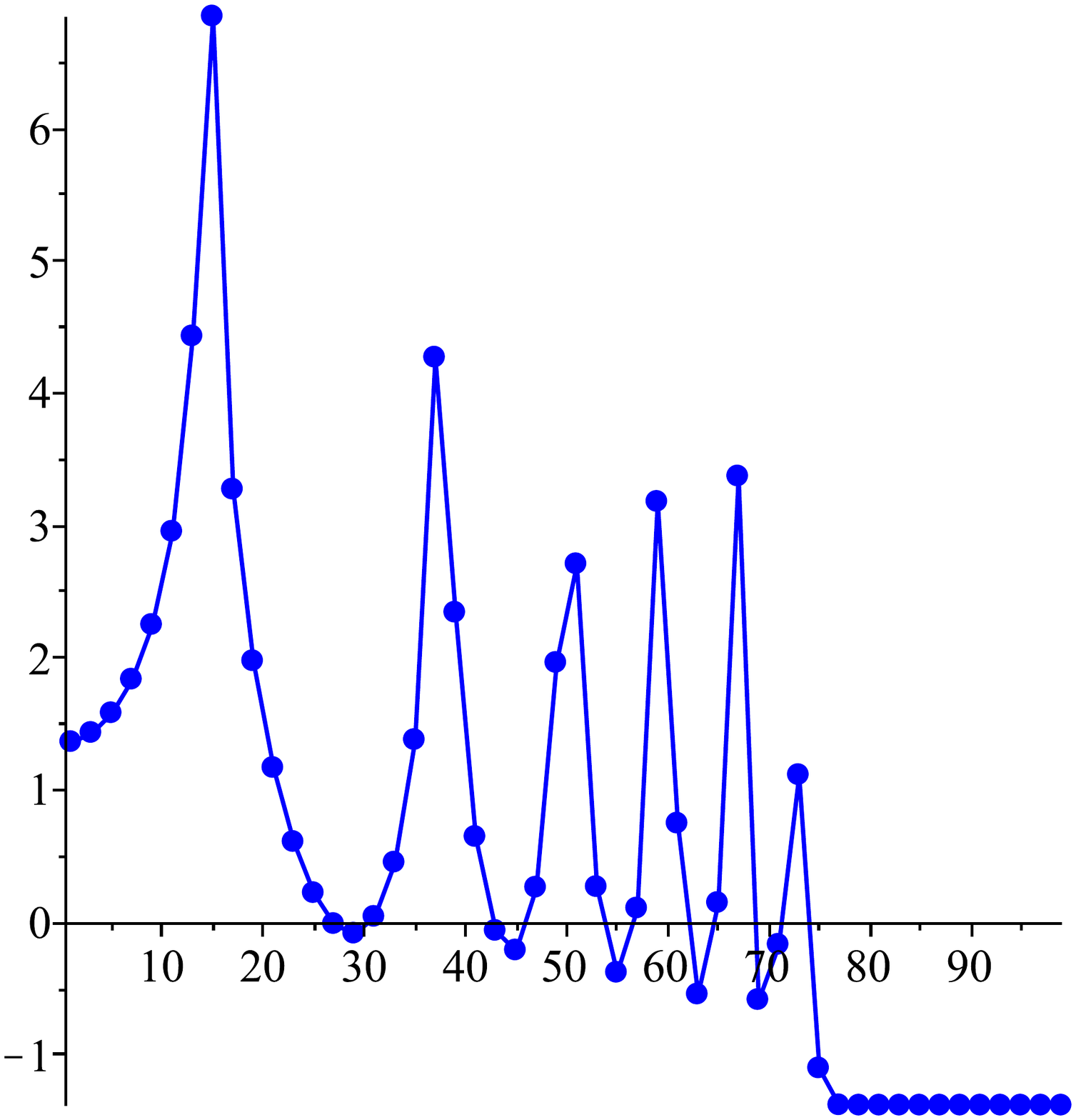}
			\caption{$\log\Re \beta_{2n+1}(i)$.}
			\label{fig:sibetalog}
		\end{subfigure}
		\caption{Plots of $\Im \alpha_n(s)$ and $\Re \beta_n(s)$ for $n=0, \dots, 100$, with $s=i$.}
		\label{fig: rec coeff imag}
	\end{figure}

	Figures \ref{fig: sub and super zeros}, \ref{fig: move across breaking curve}, \ref{fig: rec coeff}, and \ref{fig: rec coeff imag} have been computed using the nonlinear discrete string equations for the recurrence coefficients presented in \cite[Theorem 2, Theorem 4]{BCE_Jacobi}, see also \cite[\S 5.2]{Walter_book}. In Figure~\ref{fig: rec coeff imag}, we have used from \cite{deano2015kissing} that $\beta_n(s)\in\R$ and $\alpha_n(s)\in i\R$ when $s\in i\R$. Moreover, it was also shown in \cite{deano2015kissing} that for fixed $n$, $\alpha_n(t)$ and $\beta_{2n+1}(it)$ will have poles (as a function of $t$) for $t\in\R$. As such, we have plotted $\Re \beta_{2n+1}$ on a log scale in Figure~\ref{fig:sibetalog}. Once the recurrence coefficients $\alpha_n(s)$ and $\beta_n(s)$ have been computed, we assemble the Jacobi matrix for the orthogonal polynomials and calculate its eigenvalues, which correspond to the zeros of $p_{50}(z;s)$, as explained for instance in \cite{Chihara_OPs,Gautschi_OPs}. Calculations have been done in Maple, using an extended precision of 100 digits.
	
	\subsection*{Overview of Paper}
		The outline of the present paper is as follows. In Section~\ref{sec: steepest descent}, we provide a review on the Riemann-Hilbert problem for the orthogonal polynomials and the method of nonlinear steepest descent pioneered by Deift and Zhou in the early 1990s. In particular, we show how the existence of a suitable $h$-function can be used to obtain strong asymptotics of the polynomials throughout the complex plane and asymptotics of the recurrence coefficients. Moreover, we also prove Lemma~\ref{seq-prop} when discussing solutions to the global parametrix in Section~\ref{sub: Genus 1 Global Parametrix}.
		
		In Section~\ref{sec: Existence of h}, we implement the technique of continuation in parameter space to obtain the desired $h$-function for $s\in\C\setminus\mathfrak{B}$. In this section, we prove Theorems~\ref{thm: global phase portrait}, \ref{thm: recurrence coeffs genus 0}, and \ref{thm: recurrence coeffs genus 1}. 
		
		In Section~\ref{sec: DS Reg}, we study the double scaling limit as $s\to s_*$, where 
		$s^*\in\mathfrak{b}_+\cup\mathfrak{b}_-$. Moreover, we prove Theorem~\ref{thm: rec coeff double scaling regular} in the final part of this section.
		
		Finally, in Section~\ref{sec: DS Crit}, we complete our analysis by investigating the double scaling limit as $s\to 2$ via \eqref{def: ds crit breaking point}. We end the paper with a proof of Theorem~\ref{thm: rec coeff double scaling critical}.
		
	\subsection*{Acknowledgments}
		The authors would like to thank Arieh Iserles, Guilherme Silva, and Maxim Yattselev for their guidance and discussions on all aspects of the Kissing polynomials and the Riemann-Hilbert approach to these polynomials. This work was carried out while A.F.C. was a PhD student at the University of Cambridge, and he is thankful for his current support by the Cantab Capital Institute for the Mathematics of Information and the Cambridge Centre for Analysis.  A. D. gratefully acknowledges financial support from EPSRC, grant EP/P026532/1, Painlev\'e equations: analytical properties and numerical computation.

\section{The Riemann-Hilbert Problem and Overview of Steepest Descent}\label{sec: steepest descent}
	The formulation of the orthogonal polynomials as a solution to a Riemann-Hilbert problem was first given by Fokas, Its, and Kitaev in the early 1990s \cite{fokas1992isomonodromy}. This formulation became even more powerful in the late 1990s due to the development of the nonlinear steepest descent method to obtain asymptotic solutions to Riemann-Hilbert problems, developed by Deift and Zhou \cite{deift1999uniform,deift1999strong,deift1993steepest}. In this section, we review the Riemann-Hilbert problem and nonlinear steepest descent as it relates to the polynomials defined in \eqref{eq: ortho p defn}. We refer the reader to the works \cite{bleher2011lectures,deift1999orthogonal} for more details on these issues.

	\subsection{The Riemann-Hilbert Problem and The Modified External Field}
		Given a smooth curve $\Sigma$ connecting $-1$ to $1$ in $\C$, oriented from $-1$ to $1$, consider the following Riemann-Hilbert problem for $Y: \C\setminus\Sigma\to \C^{2\times 2}$, 
		\begin{subequations}
			\label{eq: Y RHP}
			\begin{alignat}{2}
				&Y_n^N(z;s) \text{ is analytic for } z\in \C\setminus \Sigma, \qquad && \\
				&Y_{n,+}^N(z;s) = Y_{n,-}^N(z;s) \begin{pmatrix}
					1 & e^{-N f(z;s)} \\
					0 & 1
				\end{pmatrix}, \qquad && z \in \Sigma, \\
				&Y_n^N(z;s) = \left(I + \mathcal{O}\left(\frac{1}{z}\right)\right) z^{n \sigma_3}, \qquad && z \to \infty, \\
				&Y_n^N(z;s) = \mathcal{O}\begin{pmatrix}
					1 & \log\left|z\mp 1\right| \\
					1 & \log\left|z\mp 1\right| 
				\end{pmatrix}, \qquad && z\to \pm 1.
			\end{alignat}
		\end{subequations}
		Above, $\sigma_3$ is the Pauli matrix given by $\sigma_3 = \diag(1,-1)$. For notational convenience, we drop the dependence of the above RHP and its solution on $s$ and $n$. We also define $\kappa_{n,N}$ as the normalizing constant for $p_n^N$, obtained via
		\begin{equation}\label{eq: kappa n}
			\int_{-1}^{1}(p_n^N(z;s))^2e^{-Nf(z;s)}dz=\frac{1}{\kappa_{n,N}^2}.
		\end{equation}
		The existence of $Y$ is equivalent to the existence of the monic orthogonal polynomial $p_n^N$ defined in \eqref{eq: ortho p defn N n}, of degree \textit{exactly} $n$, and, furthermore, if $\kappa_{n-1,N}$ is finite and non-zero, then $Y$ is explicitly given by
		\begin{equation}\label{eq: Y eqn}
			Y(z) = \begin{pmatrix}
			p_n^N(z) & \left(\mathcal{C}p_n^N e^{-Nf}\right)(z) \\
			-2\pi i \kappa_{n-1,N}^2 p_{n-1}^N(z) & -2\pi i \kappa_{n-1,N}^2 \left(\mathcal{C}p_{n-1}^N e^{-Nf}\right)(z)
		\end{pmatrix}.
		\end{equation}
		We recall that throughout the present analysis we take $N=n$, and we also drop the dependence of $Y$ on $N$ for notational convenience. In \eqref{eq: Y eqn}, $\mathcal{C}g$ denotes the Cauchy transform of the function $g$ along $\Sigma$, i.e.
		\begin{equation*}
			\left(\mathcal{C}g\right)(z) = \frac{1}{2\pi i}\int_{\Sigma} \frac{g(u)}{u-z}\, du,
		\end{equation*}
		which is analytic in $\mathbb{C}\setminus \Sigma$. The Deift-Zhou method of nonlinear steepest descent is a powerful method of determining large $n$ asymptotics of solutions to these types of Riemann-Hilbert problems, and as such, we can use it to determine asymptotics of the polynomials $p_n$ and related quantities. 
		
		The first transformation requires the existence of a \textit{modified external field}, or \textit{h-function}, whose properties we describe below. First, we define $\gamma_{c,0}:=(-\infty, -1]$ and set $\Omega = \Omega(s)=\gamma_{c,0} \cup \Sigma$. Next we partition $\Omega$ into two subsets as $\Omega= \mathfrak{M}\cup\mathfrak{C}$, where the arcs in $\mathfrak{M}$ are denoted \textit{main arcs} and those in $\mathfrak{C}$ are denoted \textit{complementary arcs}. Once this partitioning has been completed, we may define a hyperelliptic Riemann surface $\mathfrak{R} = \mathfrak{R}(s)$ whose branchcuts are precisely the main arcs in $\mathfrak{M}$ and whose branchpoints we define to be the set $\Lambda=\Lambda(s)$ . If the genus of $\mathfrak{R}$ is $L$, we may write $\mathfrak{M} = \cup_{j=0}^L \gamma_{m,j}$ and $\mathfrak{C}=\cup_{j=0}^L \gamma_{c,j}$. Moreover, when we refer to the genus of $h$, we are referring to the genus of the associated Riemann surface. Finally, we impose that all arcs in $\Omega$ are bounded, aside from the one complementary arc $\gamma_{c,0}$. 
		
		The question remains: how do we partition $\Sigma$ and choose the arcs in $\mathfrak{M}$ and $\mathfrak{C}$? Equivalently, we may ask: how do we choose the appropriate genus $L$? In order to make the appropriate first transformation to \eqref{eq: Y RHP} to begin the process of steepest descent, we must first construct a function $h$ that satisfies both a scalar Riemann-Hilbert problem on $\Omega$ and certain inequalities, to be described below. Therefore, the arcs in $\mathfrak{M}$ and $\mathfrak{C}$, and also the genus $L$, are chosen so that we can construct a suitable $h$-function. The modified external field must satisfy:
		\begin{subequations}\label{eq: RHP for h}
			\begin{alignat}{2}
				&h(z;s) \text{ is analytic for } z\in\C\setminus\Omega, \qquad && \\
				&h_+(z;s) - h_-(z;s) = 4\pi i \eta_j, \qquad && z\in \gamma_{c,j}, \\
				&h_+(z;s) + h_-(z;s) = 4\pi i \omega_j, \qquad && z\in \gamma_{m,j}, \label{eq: h jump main arc}\\
				&h(z;s) = - f(z;s) - \ell + 2\log z + \mathcal{O}\left(\frac{1}{z}\right), \qquad && z \to \infty\label{eq: h asymptotics}\\ 
				&\Re h(z;s) =\mathcal{O}\left((z\mp 1)^{1/2}\right), \qquad && z\to \pm 1, \\
				&\Re h(z;s) = \mathcal{O}\left((z-\lambda)^{3/2}\right), \qquad && z \to \lambda, \quad \lambda \in \Lambda\setminus\{\pm1 \},
			\end{alignat}
		\end{subequations}
		for $j = 0, 1, \dots, L$. Above, we impose that $\omega_L=0$ and $\eta_0=1$; the remaining real constants $\eta_j$ and $\omega_j$ (which only depend on $s$) can be chosen arbitrarily to satisfy \eqref{eq: RHP for h}. Furthermore, the constant $\ell=\ell(s)$ depends only on the parameter $s$.
		\begin{remark}
			Given any genus $L$ and arbitrary constants $\ell, \eta_j, \omega_j \in \R$, there is no guarantee that a solution to \eqref{eq: RHP for h} even exists. However, if such a solution does exist, it will be unique.
		\end{remark}
		\begin{remark}\label{rem: Riemann surface}
			Assuming that $L=0$ or $L=1$ and provided we are able to construct a solution to \eqref{eq: RHP for h}, we define the Riemann surface $\mathfrak{R}$ to be the two-sheeted genus $L$ Riemann surface associated to the algebraic equation
			\begin{equation}
				\xi^2 = h'(z)^2 = \Pi^2(z)R(z),
			\end{equation}
			where 
			\begin{equation}
				R(z) = \frac{1}{z^2-1} \prod_{j=0}^{2L-1} (z-\lambda_j),
			\end{equation}
			and $\Pi(z)$ is a polynomial of degree $1-L$, chosen so that $h'$ possess the correct asymptotics at infinity. The branch cuts of $\mathfrak{R}$ are taken along $\gamma_{m,j}$, $j = 0, \dots, L$, and the top sheet is fixed so that 
			\begin{equation}
				\xi(z) =-f'(z;s) + \mathcal{O}\left(\frac{1}{z}\right)= -s + \mathcal{O}\left(\frac{1}{z}\right), \qquad z\to \infty_1.
			\end{equation}
		\end{remark}
		In addition to solving the above scalar Riemann-Hilbert problem, $h$ must also satisfy the following inequalities
		\begin{subequations}\label{eq: both inequalities}
			\begin{align}
			&\Re h(z)<0 \text{ if } z \text{ is an interior point of any bounded complementary arc }\gamma_c \in \mathfrak{C}, \label{eq: complementary arc inequality}\\
			&\Re h(z_0)>0 \text{ for } z_0 \text{ in close proximity to any interior point of a main arc } \gamma_m \in \mathfrak{M}. \label{eq: main arc inequality}
			\end{align}
		\end{subequations}
		If $s\in \C$ is such that we can construct a function $h(z;s)$ satisfying both \eqref{eq: RHP for h} and \eqref{eq: both inequalities}, we call $s$ a \textit{regular point}. Now, provided $s$ is a regular point, we may proceed with the process of nonlinear steepest descent as follows.

		\subsection{Overview of Deift-Zhou Nonlinear Steepest Descent}
			The first transformation of steepest descent aims to normalize the Riemann-Hilbert problem at infinity. To do so, we define
			\begin{equation}\label{eq: T def}
				T(z) := e^{-n\ell \sigma_3/2} Y(z) e^{-\frac{n}{2}\left[h(z)+f(z)\right]\sigma_3},
			\end{equation}
			where we recall that $\ell\in\R$ is defined by \eqref{eq: h asymptotics} and $f(z;s)=sz$. By making this transformation, we see that $T$ satisfies the following Riemann-Hilbert problem:
			\begin{subequations}
				\label{eq: T RHP}
				\begin{alignat}{2}
				&T(z) \text{ is analytic for } z\in \C\setminus \Sigma, \qquad && \\
				&T_+(z) = T_-(z) \begin{pmatrix}
					e^{-\frac{n}{2}\left(h_+(z)-h_-(z)\right)} & e^{\frac{n}{2}\left(h_+(z)+h_-(z)\right)}  \\
					0 & e^{\frac{n}{2}\left(h_+(z)-h_-(z)\right)} 
				\end{pmatrix}, \qquad && z \in \Sigma, \\
				&T(z) = I + \mathcal{O}\left(\frac{1}{z}\right), \qquad && z \to \infty, \\
				&T(z) = \mathcal{O}\begin{pmatrix}
					1 & \log\left|z\mp 1\right| \\
					1 & \log\left|z\mp 1\right| 
				\end{pmatrix}, \qquad && z\to \pm 1.
				\end{alignat}
			\end{subequations}
			Note that the Riemann-Hilbert problem above also depends on $s$, but we have again dropped this dependence for notational convenience. We also remark that \eqref{eq: h jump main arc} and \eqref{eq: main arc inequality} imply that $\Re h(z) =0$ for $z\in\mathfrak{M}$. As $\mathfrak{M}$ is part of the zero level set of $\Re h$, the jump matrix for $T$ has highly oscillatory diagonal entries when $z \in \mathfrak{M}$.  Furthermore, if $z \in \mathfrak{C}\setminus \gamma_{c,0}$, the diagonal entries of the jump matrix will be constant and purely imaginary. Moreover, the $(1,2)$-entry of the jump matrix will decay exponentially quickly to $0$ by \eqref{eq: complementary arc inequality}. The next transformation of the steepest descent process deforms the jump contours so that the highly oscillatory entries of the jump matrix decay exponentially quickly, and is referred to as the \textit{opening of lenses}.
			
			The opening of lenses relies on the following factorization of the jump matrix across a main arc
			\begin{equation}\label{eq: jump factorization}
				\begin{pmatrix}
				e^{-nH(z)} & e^{2\pi i n \omega_j}  \\
				0 & e^{nH(z)} 
				\end{pmatrix} = \begin{pmatrix}
					1 & 0 \\
					e^{n\left(H(z)-2\pi i \omega_j\right)} & 1 
				\end{pmatrix}\begin{pmatrix}
					0 & e^{2\pi in \omega_j} \\
					-e^{-2\pi in \omega_j} & 0 
				\end{pmatrix} \begin{pmatrix}
					1 & 0 \\
					e^{n\left(-H(z)-2\pi i \omega_j\right)} & 1
				\end{pmatrix},
			\end{equation}
			where we have defined $H(z) = \left(h_+(z)-h_-(z)\right)/2$. On the $+$-side ($-$-side) of each main arc, we define $\gamma_{m,j}^+$ ($\gamma_{m,j}^-$) to be an arc which starts and ends at the endpoints of $\gamma_{m,j}$ and remains entirely on the $+$($-$) side of $\gamma_{m,j}$. For now we do not impose any restrictions on the precise description of these arcs, but we enforce that they remain in the region where $\Re h >0$, which is possible due to \eqref{eq: main arc inequality}. We define $\mathcal{L}_j^\pm$ to be the region bounded between the arcs $\gamma_{m,j}$ and $\gamma_{m,j}^\pm$, respectively, and set $\hat{\Sigma} := \Sigma \cup_{j=0}^L \left(\gamma_{m,j}^+ \cup \gamma_{m,j}^-\right)$ as in Figure~\ref{fig: opening of lenses}. 
			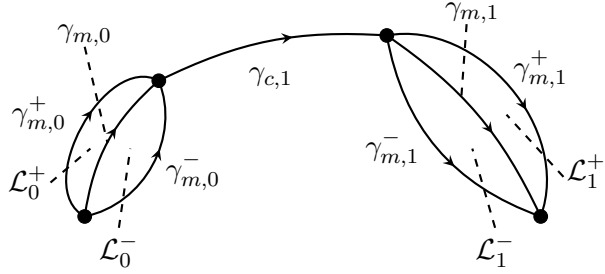
\begin{figure}[t]
				\centering
				\begin{tikzpicture}[scale=1.5]
				\draw [thick,postaction={mid arrow=black}] (-2,0) to [bend left=20] (-1.35, 1.2);
				\draw [thick,postaction={mid arrow=black}] (0.65,1.6) to [bend left=15] (2, 0);
				\draw [thick,postaction={mid arrow=black}] (-1.35, 1.2) to [bend left=15] (0.65, 1.6);
				\node [right] at (-.65, 1.22) {$\gamma_{c,1}$};
				
				\draw [thick,postaction={mid arrow=black}] (-2,0) to [bend left=90] (-1.35, 1.2);
				\draw [thick,postaction={mid arrow=black}] (-2,0) to [bend right=50] (-1.35, 1.2);
				
				\draw [thick,postaction={mid arrow=black}] (0.65,1.6) to [bend left=60] (2, 0);
				\draw [thick,postaction={mid arrow=black}] (0.65,1.6) to [bend right=30] (2, 0);
				
				\draw [fill] (0.65,1.6) circle [radius=0.06];
				\draw [fill] (-1.35,1.2) circle [radius=0.06];
				\draw [fill] (2,0) circle [radius=0.06];
				\draw [fill] (-2,0) circle [radius=0.06];
				
				\node  at (-1.05, 0.4) {$\gamma_{m,0}^-$};
				\node  at (-2.4, 0.9) {$\gamma_{m,0}^+$};	
				\node  at (-2, 1.6) {$\gamma_{m,0}$};
				\draw [thick, dashed] (-2, 1.5) to (-1.8,0.65);
				\node  at (-2.5, .3) {$\mathcal{L}_{0}^+$};
				\draw [thick, dashed] (-2.3, .3) to (-1.95,0.6);
				\node  at (-1.7, -.3) {$\mathcal{L}_{0}^-$};
				\draw [thick, dashed] (-1.7, -.1) to (-1.6,0.6);
				
				\node  at (0.7, 0.6) {$\gamma_{m,1}^-$};
				\node  at (2, 1.3) {$\gamma_{m,1}^+$};		
				\node  at (1.4, 1.8) {$\gamma_{m,1}$};		
				\draw [thick, dashed] (1.35, 1.7) to (1.3,1);		
				\node  at (1.6, -.3) {$\mathcal{L}_{1}^-$};
				\draw [thick, dashed] (1.6, -.1) to (1.4,0.6);		
				\node  at (2.4, .4) {$\mathcal{L}_{1}^+$};
				\draw [thick, dashed] (2.2, .35) to (1.7,0.9);		
				\end{tikzpicture}
				\caption{The contour $\hat{\Sigma}$ after opening lenses in the case $L=1$.}
				\label{fig: opening of lenses}
			\end{figure}
			We can now define the third transformation of the steepest descent process as 
			\begin{equation}
				S(z):= \begin{cases}
					T(z)\begin{pmatrix}
						1 & 0 \\
						\mp e^{-nh(z)} & 1
					\end{pmatrix}, \qquad &z \in \mathcal{L}_j^\pm, \\
					T(z), \qquad &\text{otherwise}.
				\end{cases}
			\end{equation}
			We then have that $S$ solves the following Riemann-Hilbert problem on $\hat{\Sigma}$:
			\begin{subequations}
				\label{eq: S RHP}
				\begin{alignat}{2}
					&S(z) \text{ is analytic for } z\in \C\setminus \hat{\Sigma}, \qquad && \\
					&S_+(z) = S_-(z) j_S(z), \qquad && z \in \hat{\Sigma}, \\
					&S(z) = I + \mathcal{O}\left(\frac{1}{z}\right), \qquad && z \to \infty, 
				\end{alignat}
			\end{subequations}
			Note that for $z\in \gamma_{m,j}^\pm$, 
			\begin{equation}
				j_S(z) = \begin{pmatrix}
					1 & 0 \\
					e^{-nh(z)} & 1
				\end{pmatrix}, 
			\end{equation}
			which decays exponentially quickly to the identity as $n\to \infty$, due to \eqref{eq: main arc inequality}. As $S=T$ outside of the lenses, we see that there are no changes to the jump matrix across a complementary arc, so that
			\begin{equation}
				j_S(z) = \begin{pmatrix}
					e^{-2\pi i n \eta_j} & e^{\frac{n}{2}(h_+(z)+h_-(z))} \\
					0 & e^{2\pi i n \eta_j}
				\end{pmatrix},  \qquad z \in \gamma_{c,j},
			\end{equation}
			which again tends exponentially quickly to a diagonal matrix as $n\to\infty$. Finally, we see that over $\gamma_{m,j}$, the jump matrix is given by
			\begin{equation}
				j_S(z) = \begin{pmatrix}
					0 & e^{2\pi i n \omega_j} \\
					-e^{-2\pi i n \omega_j} & 0
				\end{pmatrix}, \qquad z \in \gamma_{m,j},
			\end{equation}
			which follows from the factorization \eqref{eq: jump factorization}. Now consider the following model Riemann-Hilbert problem for the global parametrix, $M$, which is obtained by neglecting those entries in the jump matrices which are exponentially close to the identity in the Riemann-Hilbert problem for $S$,
			\begin{subequations}
				\label{eq: M RHP}
				\begin{alignat}{2}
					&M(z) \text{ is analytic for } z\in \C\setminus \Sigma, \qquad && \\
					&M_+(z) = M_-(z) \begin{pmatrix}
						e^{-2\pi i n \eta_j} & 0 \\
						0 & e^{2\pi i n \eta_j}
					\end{pmatrix},  \qquad &&z \in \gamma_{c,j}, \,\, j=1, \dots, L, \\
					&M_+(z) = M_-(z) \begin{pmatrix}
					0 & e^{2\pi i n \omega_j} \\
					-e^{-2\pi i n \omega_j} & 0
					\end{pmatrix},  \qquad &&z \in \gamma_{m,j}, \,\, j=0, \dots, L, \\
					&M(z) = I + \mathcal{O}\left(\frac{1}{z}\right), \qquad && z \to \infty, 
				\end{alignat}
			\end{subequations}
			Assuming we are able to solve the model Riemann-Hilbert problem, we would like to make the final transformation by setting $R = SM^{-1}$, however, this will turn out not to be valid near the endpoints. As such, we will need a more refined local analysis near these points. More precisely, we will solve the Riemann-Hilbert problem for $S$ \textit{exactly} near these points, and impose further that it matches with the global parametrix as $n\to\infty$. Therefore, define  $D_{\lambda}=D_\delta(\lambda)$ to be discs of fixed radius $\delta$ around each endpoint $\lambda \in \Lambda$. For each $\lambda \in \Lambda$, we seek a local parametrix $P^{(\lambda)}$, dependent on $n$, which solves
				\begin{subequations}
				\label{eq: P RHP}
					\begin{alignat}{2}
						&P^{(\lambda)}(z) \text{ is analytic for } z\in D_\lambda\setminus \hat{\Sigma}, \qquad && \\
						&P^{(\lambda)}_+(z) = P^{(\lambda)}_-(z) j_S(z),  \qquad &&z \in D_\lambda\cap \hat{\Sigma} \\
						&P^{(\lambda)}(z)= M(z)\left(I + \mathcal{O}\left(\frac{1}{n}\right)\right), \qquad && n \to \infty, \quad z\in \partial D_\lambda
					\end{alignat}
			\end{subequations}
			We also require that $P^{(\lambda)}$ has a continuous extension to $\overline{D_\delta(\lambda)}\setminus\hat{\Sigma}$ and remains bounded as $z\to\lambda$. The construction of both the global and local parametrices are now standard, but are included below for completeness. Near the hard edges at $\pm 1$, the local parametrix can be constructed with the help of Bessel functions. Near the soft edges, if any, the local parametrices can be constructed using Airy functions. 
		
		\subsection{Small Norm Riemann-Hilbert Problems}\label{sec: Small Norm}
			We may complete the process of nonlinear steepest descent by defining the final transformation as
			\begin{equation}
			R(z) = \begin{cases}
			S(z)M(z)^{-1}, \qquad & z \in \C\setminus \left(\hat{\Sigma}\cup_{\lambda\in\Lambda} D_\lambda)\right) \\
			S(z)P^{(\lambda)}(z)^{-1}, \qquad & z\in D_\lambda\setminus\hat{\Sigma}, \, \lambda\in\Lambda. 
			\end{cases}
			\end{equation}
			Provided we were able to appropriately construct both the local and global parametrices, the matrix $R$ will satisfy a \say{small norm} Riemann-Hilbert problem on a new contour, $\Sigma_R$, whose jumps decay to the identity in the appropriate sense. The contour $\Sigma_R$ will consist of the oriented arcs forming the boundaries $\partial D_\lambda$ about each $\lambda \in \Lambda$ and the portions of $\gamma_{m,L}^\pm$ which are not in the interior of $D_\lambda$, as illustrated in Figure~\ref{fig: Sigma R} for the genus $L=1$ case.  
			\begin{figure}[t]
				\centering
				\begin{tikzpicture}[scale=1.5]
				\draw [ultra thick,postaction={mid arrow=black}] (-2.2,0.35) to [bend left=30] (-1.75, 1.5);
				\draw [ultra thick,postaction={mid arrow=black}] (-1.6,0.15) to [bend right=30] (-1.25, 1.25);
				
				\draw [ultra thick,postaction={mid arrow=black}] (1.75,2.35) to [bend left=30] (2.2, 0.35);
				\draw [ultra thick,postaction={mid arrow=black}] (1.1,2.0) to [bend right=30] (1.6, 0.15);
				
				\draw [ultra thick, postaction={mid arrow=black}] (-1.15,2) to [bend left =40] (1.0,2.5);
				
				\draw[ultra thick, postaction={mid arrow=black}] (-1.6,0) arc (0:-360:0.4 and 0.4);
				\draw[ultra thick, postaction={mid arrow=black}] (1.6,0) arc (-180:-540:0.4 and 0.4);
				\draw[ultra thick, postaction={mid arrow=black}] (-0.95,1.65) arc (0:-360:0.4 and 0.4);
				\draw[ultra thick, postaction={mid arrow=black}] (1.75,2.34) arc (0:-360:0.4 and 0.4);
				
				\node at (2,0) {$D_{1}$};
				\node at (-2,0) {$D_{-1}$};
				\node at (-1.35,1.65) {$D_{\lambda_0}$};
				\node at (1.35,2.35) {$D_{\lambda_1}$};
				\end{tikzpicture}
				\caption{The contour $\Sigma_R$ in the case $L=1$. Note that we have chosen the contours $\partial D_\lambda$ to have clockwise orientation.}
				\label{fig: Sigma R}
			\end{figure}
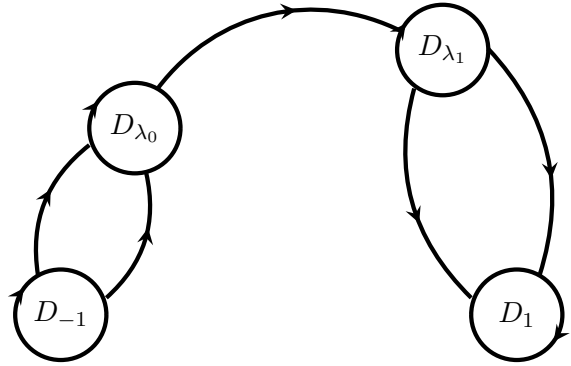
			Moreover, the jump matrix $j_R(z)$ will satisfy
			\begin{align}
				j_R(z) = \begin{cases}
				I + \mathcal{O}\left(e^{-cn}\right), \qquad &z \in \Sigma_R\setminus\bigcup_{\lambda\in\Lambda}\partial D_\lambda \\
				I +\mathcal{O}\left(\frac{1}{n}\right), \qquad &z\in\bigcup_{\lambda\in\Lambda}\partial D_\lambda
			\end{cases},
			\end{align}
			for some $c>0$ with uniform error terms. In particular, we may write the jump matrix as $j_R(z) = I + \Delta(z)$, where
			\begin{equation}\label{eq: Delta expansion}
				\Delta(z) \sim \sum_{k=1}^{\infty} \frac{\Delta_k(z)}{n^k}, \qquad n\to\infty,\, z\in \Sigma_R.
			\end{equation}
			By \cite[Theorem~7.10]{deift1999strong}, this behavior then implies that $R$ has an asymptotic expansion of the form
			\begin{equation}\label{eq: R expansion}
				R(z) \sim I + \sum_{k=1}^\infty \frac{R_k(z)}{n^k}, \qquad n\to\infty, 
			\end{equation}
			valid uniformly for $z \in\C\setminus\cup_{\lambda\in\Lambda}\partial D_\lambda$. Above, the $R_k(z)$ are solutions to the following Riemann-Hilbert problem (c.f \cite[Section~8.2]{arnojacobi}),
			\begin{subequations}
				\label{eq: R_k RHP}
				\begin{alignat}{2}
				&R_k(z) \text{ is analytic for } z \in \C\setminus  \bigcup_{\lambda\in\Lambda}\partial D_\lambda \qquad &&\\
				& R_{k,+}(z) = R_{k,-}(z) + \sum_{j=1}^{k-1}R_{k-j,-}\Delta_j(z), \qquad &&z\in \bigcup_{\lambda\in\Lambda}\partial D_\lambda\\
				& R_k(z) = \frac{R_k^{(1)}}{z}+\frac{R_k^{(2)}}{z^2}+\mathcal{O}\left(\frac{1}{z^3}\right), \qquad &&z \to\infty,
				\end{alignat}
			\end{subequations}
			where the $\Delta_j$ are given by \eqref{eq: Delta expansion}. Therefore, if we are able to determine the $\Delta_k$ in \eqref{eq: Delta expansion}, we will be able to sequentially solve for the $R_k$ in the expansion for $R$ in \eqref{eq: R expansion} via the Riemann-Hilbert problem \eqref{eq: R_k RHP}.
			
		\subsection{Unwinding the Transformations}\label{sec: unwinding transormations}
		The process of retracing the steps of Deift-Zhou steepest descent to obtain uniform asymptotics of the orthogonal polynomials in the plane is now standard. Of particular interest to us is to obtain the asymptotics of the recurrence coefficients. 
		 Unwinding the transformations away from the lenses, we see that
		\begin{align}
		Y(z)= e^{n\ell\sigma_3/2} T(z) e^{\frac{n}{2}\left[h(z)+sz\right]\sigma_3}
		&= e^{n\ell\sigma_3/2} S(z) e^{\frac{n}{2}\left[h(z)+sz\right]\sigma_3},\notag \\
		&= e^{n\ell\sigma_3/2} R(z)M(z) e^{\frac{n}{2}\left[h(z)+sz\right]\sigma_3}, \label{eq: Y away from cut}
		\end{align}
		where $M(z)$ above is the appropriate global parametrix. 
%
%
		We recall that the three term recurrence relation is given by
		\begin{equation*}
		z p_n^n(z) = p_{n+1}^n(z) + \alpha_n p_n^n(z) + \beta_n p_{n-1}^n(z).
		\end{equation*}
		To state the recurrence coefficients in terms of $Y$, we first note that from \eqref{eq: Y RHP} that we may write
		\begin{equation}
		Y(z) z^{-n\sigma_3} = I + \frac{Y^{(1)}}{z} + \frac{Y^{(2)}}{z^2} + \mathcal{O}\left(\frac{1}{z^3}\right), \qquad z\to\infty. 
		\end{equation}
		Then, we may write the recurrence coefficients as 
		\begin{equation}\label{eq: recurrence coeffs in terms of Y}
		\alpha_n = \frac{Y^{(2)}_{12}}{Y^{(1)}_{12}} - Y^{(1)}_{22}, \qquad \beta_n = Y^{(1)}_{12}Y^{(1)}_{21}. 
		\end{equation}
		see \cite[Theorem 3.1]{deift1999strong}, noting also that the matrix $Y^{(1)}$ is traceless, so $Y^{(1)}_{11}=-Y^{(1)}_{22}$. 
		As before, we will unwind these transformations until we are able to express the recurrence coefficients in terms of the global parametrix and the matrix valued $R(z)$. We continue by writing
		\begin{equation}
		T(z) = I + \frac{T^{(1)}}{z} + \frac{T^{(2)}}{z^2} + \mathcal{O}\left(\frac{1}{z^3}\right), \qquad z\to\infty. 
		\end{equation}
		Using \eqref{eq: h asymptotics}, we recall that
		\begin{equation}
		h(z;s) = -f(z;s)-l + 2\log(z) + \frac{c_1}{z}+\frac{c_2}{z^2} + \mathcal{O}\left(\frac{1}{z^3}\right), \qquad z\to\infty, 
		\end{equation}
		so that
		\begin{equation}
		e^{-\frac{n}{2}\left(h(z;s)+f(z;s)\right)}= z^{-n} e^{\frac{n\ell}{2}}\left(1-\frac{nc_1}{2z} + \frac{nc_1^2-4nc_2}{8z^2} +\mathcal{O}\left(\frac{1}{z^3}\right)\right), \qquad z\to\infty. 
		\end{equation}
		Next, using \eqref{eq: T def} we compute
		\begin{subequations}
			\begin{equation}
			T^{(1)}_{12} = e^{-n \ell} Y^{(1)}_{12}, \qquad T^{(1)}_{21} = e^{n \ell} Y^{(1)}_{21}
			\end{equation}
			\begin{equation}
			T^{(1)}_{22} = Y^{(1)}_{22} + \frac{n c_1}{2},\qquad T^{(2)}_{12} = e^{-n \ell} \left(\frac{n c_1}{2}Y^{(1)}_{12}+Y^{(2)}_{12}\right).
			\end{equation}
		\end{subequations}
		Then, it easily follows that \eqref{eq: recurrence coeffs in terms of Y} becomes
		\begin{equation}\label{eq: recurrence coeffs in terms of T}
		\alpha_n = \frac{T^{(2)}_{12}}{T^{(1)}_{12}} - T^{(1)}_{22}, \qquad \beta_n = T^{(1)}_{12}T^{(1)}_{21}. 
		\end{equation}
		The above equation will be the starting point of our analysis in Sections~\ref{sec: asymptotics I} and \ref{sec: asymptotics II}, where we prove Theorems~\ref{thm: recurrence coeffs genus 0} and \ref{thm: recurrence coeffs genus 1} , respectively, providing the asymptotics of the recurrence coefficients. 
		
		Below, we give a detailed description on how to solve the model problem \eqref{eq: M RHP} in the genus $0$ and genus $1$ cases, which will be the only two regimes we see for the linear weight under consideration. The arguments below can be easily adapted to cases of higher genera corresponding to other weights, as in \cite{bertola2016asymptotics}.
		
		\subsection{The Global Parametrix in genus 0}\label{sub: Global parametrix genus 0}	
				In the genus $0$ regime, $\Sigma=\gamma_{m,0}(s)$, where $\gamma_{m,0}$ is chosen so that we may construct a suitable $h$ function satisfying both \eqref{eq: RHP for h} and \eqref{eq: both inequalities}. The model Riemann-Hilbert problem \eqref{eq: M RHP} in the genus $0$ case takes the following form. We seek $M: \C\setminus\gamma_{m,0} \to \C^{2\times 2}$ such that
				\begin{subequations}
					\label{eq: M RHP genus 0}
					\begin{alignat}{2}
						&M(z) \text{ is analytic for } z\in \C\setminus \gamma_{m,0}, \qquad && \\
						&M_+(z) = M_-(z) \begin{pmatrix}
							0 & 1 \\
							-1 & 0
						\end{pmatrix},  \qquad &&z \in \gamma_{m,0}, \label{eq: genus 0 model problem jump}\\
						&M(z) = I + \mathcal{O}\left(\frac{1}{z}\right), \qquad && z \to \infty. 
					\end{alignat}
				\end{subequations}
				This can be solved explicitly \cite{bleher2011lectures,deano2014large} as
				\begin{equation}\label{eq: global parametrix eq genus 0}
					M(z) = \frac{1}{\sqrt{2}\left(z^2-1\right)^{1/4}}\begin{pmatrix}
						\varphi(z)^{1/2} & i \varphi(z)^{-1/2} \\
						-i\varphi(z)^{-1/2} & \varphi(z)^{1/2}
					\end{pmatrix},
				\end{equation}
				where $\varphi(z) = z+(z^2-1)^{1/2}$, with branch cuts taken on $\gamma_{m,0}$ so that $\varphi(z) = 2z + \mathcal{O}\left(1/z\right)$, $(z^2-1)^{1/4} = z^{1/2} +\mathcal{O}(z^{-3/2})$ as $z\to\infty$. 
		
			\subsection{The Global Parametrix in genus $1$}\label{sub: Genus 1 Global Parametrix}
			In the genus $1$ regime, we have that $\Sigma=\gamma_{m,0}(s)\cup\gamma_{c,1}(s)\cup\gamma_{m,1}(s)$, and the set of branchpoints is given by $\Lambda(s)=\{-1,1,\lambda_0(s),\lambda_1(s)\}$, where the arcs and endpoints are chosen so that we may construct a suitable $h$-function. Now, the model problem \eqref{eq: M RHP} takes the form
			\begin{subequations}
				\label{eq: M RHP genus 1}
				\begin{alignat}{2}
					&M(z) \text{ is analytic for } z\in \C\setminus \Sigma, \qquad && \\
					&M_+(z) = M_-(z) \begin{pmatrix}
						e^{-2\pi i n \eta_1} & 0\\
						0 & e^{2\pi i n \eta_1}
					\end{pmatrix},  \qquad &&z \in \gamma_{c,1}, \\
					&M_+(z) = M_-(z) \begin{pmatrix}
						0 & 1 \\
						-1& 0
					\end{pmatrix},  \qquad &&z \in \gamma_{m,1},  \\
					&M_+(z) = M_-(z) \begin{pmatrix}
						0 & e^{2\pi i n \omega_0} \\
						-e^{-2\pi i n \omega_0} & 0
					\end{pmatrix},  \qquad &&z \in \gamma_{m,0},  \\
					&M(z) = I + \mathcal{O}\left(\frac{1}{z}\right), \qquad && z \to \infty.
				\end{alignat}
			\end{subequations}
			We follow the approach of \cite{bertola2016asymptotics,deift1999strong, tovbis2004semiclassical}, and solve this problem in four steps. We recall from Remark~\ref{rem: Riemann surface} that $\mathfrak{R}$ is the hyperelliptic Riemann surface associated with the algebraic equation
			\begin{equation}\label{eq: xi eqn}
				\xi^2(z) =  \frac{s^2(z-\lambda_0)(z-\lambda_1)}{z^2-1}, 
			\end{equation}
			whose branchcuts are taken along $\gamma_{m,0}$ and $\gamma_{m,1}$, and whose top sheet fixed so that 
			\begin{equation}
				\xi(z) = -s + \mathcal{O}\left(\frac{1}{z}\right),
			\end{equation}
			as $z\to \infty$ on the top sheet of $\mathfrak{R}$. We form a homology basis on $\mathfrak{R}$ using the $A$ and $B$ cycles defined in Figure~\ref{fig: homology basis}.
			\begin{figure}[t]
				\centering
				\begin{tikzpicture}[scale=1.5]
				\draw [thick] (-2,0) to [bend left=40] (-1.35, 0.9);
				\draw [thick] (0.65,1.6) to [bend left=50] (2, 0);
				\node [right] at (1.9, 0.75) {$\gamma_{m,1}$};
				\node [right] at (-2, 0) {$\gamma_{m,0}$};
				
				\draw [ultra thick,postaction={mid arrow=black}] (-2.2,-0.5) to [bend right=90] (-1.15, 1.4);
				\draw [ultra thick,postaction={mid arrow=black}] (-1.15, 1.4) to [bend right=90] (-2.2,-0.5);
				\node [left] at (-2.6, 0.1) {$A$};
				
				\draw [ultra thick,postaction={mid arrow=black}] (-1.8,0.65) to [bend left=90] (1.15, 1.5);
				\draw [dashed, ultra thick,postaction={mid arrow=black}]  (1.15, 1.5) to [bend left=90] (-1.8,0.65);
				\node [right] at (-0.4, 1.7) {$B$};
				
				\draw [fill] (0.65,1.6) circle [radius=0.06];
				\draw [fill] (-1.35,0.9) circle [radius=0.06];
				\draw [fill] (2,0) circle [radius=0.06];
				\draw [fill] (-2,0) circle [radius=0.06];
				
				\end{tikzpicture}
				\caption{The homology basis on $\mathfrak{R }$. The bold contours are on the top sheet of $\mathfrak{R}$, and the dashed contours are on the second sheet of $\mathfrak{R}$. }
				\label{fig: homology basis}
			\end{figure}
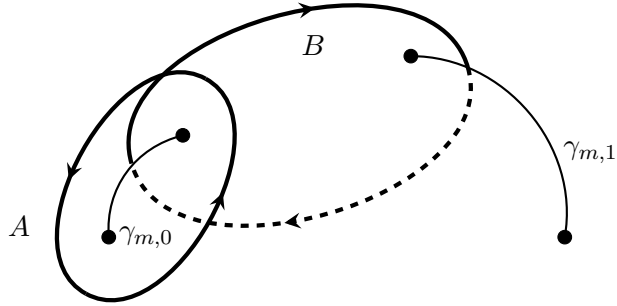
			We also recall that as $\mathfrak{R}$ is of genus $1$, the vector space of holomorphic differentials on $\mathfrak{R}$ has dimension $1$ and is linearly generated by 
			\begin{equation}
				\Omega_0 = \frac{dz}{\xi(z)\left(z^2-1\right)}.
			\end{equation}
			We then define $\omega := b \Omega_0$, with $b$ chosen to normalize $\omega$ so that
			\begin{equation}\label{eq: normalization of omega}
				\oint_A \omega = 1. 
			\end{equation}
			Moreover, if we define
			\begin{equation}\label{eq: tau def}
				\tau := \oint_B \omega, 
			\end{equation}
			it is well known that $\Im \tau > 0$, see \cite[Chapter~III.2]{farkas1992riemann}.  
			
			\subsubsection{Step One - Remove Jumps on Complementary Arcs}
				The first step aims to remove the jumps over the complementary arcs and we will follow the procedure outlined in \cite{tovbis2004semiclassical}. First, we introduce the function
				\begin{equation}\label{varXi}
					\varXi(z) = \left[\left(z^2-1\right)\left(z-\lambda_0\right)\left(z-\lambda_1\right)\right]^{1/2}, 
				\end{equation}
				with a branch cut taken on $\gamma_{m,0}$ and ${\gamma_{m,1}}$ and branch chosen so that $\varXi(z)\to z^2$ as $z\to\infty$. Next, define
				\begin{equation}\label{eq: tilde g eqn}
					\tilde{g}(z) = \varXi(z)\left[\int_{\gamma_{c,1}} \frac{ \eta_1\, d\zeta}{(\zeta-z)\varXi(\zeta)}-\int_{\gamma_{m,0}}\frac{\Delta_0\, d\zeta}{(\zeta-z)\varXi_+(\zeta)}\right],
				\end{equation}
				The constant $\Delta_0$ is chosen so that $\tilde{g}$ is analytic at infinity. More precisely, $\Delta_0$ is defined so that 
				\begin{equation}
					\int_{\gamma_{c,1}} \frac{ \eta_1\, d\zeta}{\varXi(\zeta)}- \int_{\gamma_{m,0}}\frac{\Delta_0\, d\zeta}{\varXi_+(\zeta)} = 0. 
				\end{equation}
				Note that by \eqref{eq: normalization of omega} and the definition of $\omega$, it follows that $\Delta_0 = -\eta_1 \tau$. It follows that $\tilde{g}$ is bounded near each $\lambda \in \Lambda$ and satisfies
				\begin{subequations}\label{eq : g-tide-jump}
					\begin{alignat}{2}
						&\tilde{g}_+(z) - \tilde{g}_-(z) = 2\pi i \eta_1, \qquad &&z\in \gamma_{c,1} \\
						&\tilde{g}_+(z) + \tilde{g}_-(z) = -2\pi i\Delta_0, \qquad &&z\in \gamma_{m,0},\\
						&\tilde{g}_+(z) + \tilde{g}_-(z) = 0, \qquad &&z\in \gamma_{m,1}.
					\end{alignat}
				\end{subequations}
				Next, we define
				\begin{equation}\label{eq: M0 eqn}
					M_{0}(z) = e^{-n \tilde{g}(\infty)\sigma_3} M(z) e^{n\tilde{g}(z)\sigma_3}.
				\end{equation}
				Then, $M_0$ solves the following Riemann-Hilbert problem
				\begin{subequations}
					\label{eq: M0 RHP genus 1}
					\begin{alignat}{2}
						&M_0(z) \text{ is analytic for } z\in \C\setminus \mathfrak{M}, \qquad && \\
						&M_{0,+}(z) = M_{0,-}(z) \begin{pmatrix}
							0 & e^{2\pi i n (\omega_0+\Delta_0)} \\
							-e^{-2\pi i n (\omega_0+\Delta_0)} & 0
						\end{pmatrix},  \qquad &&z \in \gamma_{m,0},  \label{eq: model problem jump condition 1}\\
						&M_{0,+}(z) = M_{0,-}(z) \begin{pmatrix}
							0 & 1 \\
							-1 & 0
						\end{pmatrix},  \qquad &&z \in \gamma_{m,1}, \label{eq: model problem jump condition 2}\\
						&M_0(z) = I + \mathcal{O}\left(\frac{1}{z}\right), \qquad && z \to \infty.
					\end{alignat}
				\end{subequations}
				Note that $M_0$ has no longer has any jumps over the complementary arcs. 
				
			\subsubsection{Step Two - Solve $n = 0$}
				In the case that $n=0$, the model problem for $M_0$ takes the form
				\begin{subequations}
					\label{eq: M0 RHP genus 1 n = 0}
					\begin{alignat}{2}
						&M_0^{(0)}(z) \text{ is analytic for } z\in \C\setminus \mathfrak{M}, \qquad && \\
						&M_{0,+}^{(0)}(z) = M_{0,-}^{(0)}(z) \begin{pmatrix}
							0 & 1 \\
							-1 & 0
					\end{pmatrix},  \qquad &&z \in \mathfrak{M}, \label{eq: model n = 0 jump}\\
					&M_0^{(0)}(z) = I + \mathcal{O}\left(\frac{1}{z}\right), \qquad && z \to \infty.
					\end{alignat}
				\end{subequations}
				The solution to \eqref{eq: M0 RHP genus 1 n = 0} is well known (see for instance \cite{bleher2011lectures}), and is given by
				\begin{equation}\label{eq: M00 eqn}
					M_0^{(0)}(z) = 
					\frac{1}{2}
					\begin{pmatrix}
						\phi(z)+\phi(z)^{-1} & i(\phi(z)-\phi(z)^{-1})\\
						-i(\phi(z)-\phi(z)^{-1}) & \phi(z)+\phi(z)^{-1}
					\end{pmatrix},
				\end{equation}
				where 
				\begin{equation}\label{eq: beta eqn}
					\phi(z) = \left(\frac{\left(z+1\right)\left(z-\lambda_1\right)}{\left(z-\lambda_0\right)\left(z-1\right)}\right)^{1/4}
				\end{equation}
				with branch cuts on $\gamma_{m,0}$ and $\gamma_{m,1}$ and the branch of the root chosen so that
				\begin{equation}\label{eq:normalization_eta}
					\lim\limits_{z\to \infty} \phi(z) = 1. 
				\end{equation}
				
				It is important to understand the location of the zeros of the entries of $M_0^{(0)}(z)$, as they will play a role later on in this construction. Note first that the zeros of $\phi(z) +\phi^{-1}(z)$ are the zeros of $\phi^4(z) - 1=\left(\phi^2(z)-1\right)\left(\phi^2(z)+1\right)$, which is meromorphic on $\mathfrak{R}$, with a zero at $\infty_1$ and one simple zero on each sheet of $\mathfrak{R}$. If we denote by $z_1$ the zero of $\phi^2(z)-1$, then $\hat{z}_1$, which denotes the projection of $z_1$ onto the opposite sheet of $\mathfrak{R}$, solves $\phi^2(z)+1$.

			\subsubsection{Step Three - Match the jumps on $\mathfrak{M}$}
				The next step in the solution is to match the jump conditions \eqref{eq: model problem jump condition 1} ans \eqref{eq: model problem jump condition 2}. We will do this by constructing two scalar functions, $\mathcal{M}_1(z, d)$ and $\mathcal{M}_2(z, d)$ which satisfy 
				\begin{equation}\label{eq: Mcal jumps}
					\mathcal{M}_+ =\begin{cases}
					\mathcal{M}_-\begin{pmatrix}
					0 & e^{2\pi i W} \\
					e^{-2\pi i W} & 0
					\end{pmatrix}, \qquad & z \in \gamma_{m,0}, \\
					\mathcal{M}_-\begin{pmatrix}
					0 & 1 \\
					1 & 0
					\end{pmatrix}, \qquad & z \in \gamma_{m,1},
					\end{cases}
				\end{equation}
				where 
				\begin{equation}
					\mathcal{M}(z,d) = \left(\mathcal{M}_1(z, d), \mathcal{M}_2(z,d)\right), 
				\end{equation}
				$W = n(\omega_0+\Delta_0)$, and $d\in\C$ is a yet to be defined constant that will be chosen to cancel the simple poles of the entries of $M_0^{(0)}$. If we can construct such functions, then it is immediate from \eqref{eq: model n = 0 jump} and \eqref{eq: Mcal jumps} that
				\begin{equation}\label{eq: L cal defn}
					\mathcal{L}(z) := \frac{1}{2}\begin{pmatrix}
						\left(\phi(z)+\phi(z)^{-1}\right)\mathcal{M}_1(z,d) & i\left(\phi(z)-\phi(z)^{-1}\right) \mathcal{M}_2(z,d) \\
						-i\left(\phi(z)-\phi(z)^{-1}\right)\mathcal{M}_1(z,-d) & \left(\phi(z)+\phi(z)^{-1}\right) \mathcal{M}_2(z,-d) \\
					\end{pmatrix}
				\end{equation}
				satisfies \eqref{eq: model problem jump condition 1} and \eqref{eq: model problem jump condition 2}. We can construct $\mathcal{M}_1$ and $\mathcal{M}_2$ with the help of theta functions on $\mathfrak{R}$. We define the Riemann theta function associated with $\tau$ in \eqref{eq: tau def} in the standard way
				\begin{equation}\label{eq: theta function def}
					\Theta(\zeta) = \sum_{m\in\Z} e^{2\pi im\zeta + \pi i \tau m^2}, \qquad \zeta\in \C. 
				\end{equation}
				The following properties of the theta function follow immediately from \eqref{eq: theta function def}:
				\begin{subequations}\label{eq: theta properties}
					\begin{align}
						& \Theta \text{ is analytic in }\C,\\
						&\Theta(\zeta) = \Theta(-\zeta),\\
						&\Theta(\zeta + 1) = \Theta(\zeta),\\
						&\Theta(\zeta+\tau) = e^{-2\pi i\zeta-\pi i \tau} \Theta(\zeta).
					\end{align}
				\end{subequations} 
				Associated with $\Theta$ is the \textit{period lattice}, $\Lambda_\tau:= \Z+\tau\Z$. The function $\Theta(\zeta)$ has a simple zero at 
				$\zeta = \frac{1}{2} + \frac{\tau}{2} \mod \Lambda_\tau$. We remark that in genus $\geq 2$ one needs to be careful as the $\Theta$ function could vanish identically.
				Next we define the Abel map as
				\begin{equation}\label{eq: Abel map}
					u(z) = -\int_{1}^z \omega, \qquad z \in \C\setminus\Sigma, 
				\end{equation}
				where we recall $\omega$ was normalized to satisfy \eqref{eq: normalization of omega}. Above, we take the path of integration on the upper sheet of $\mathfrak{R}$ in the complement of $\mathfrak{C} \cup \mathfrak{M} \cup [1, \infty)$. By \eqref{eq: normalization of omega}, we have that $u$ is well defined on $\C\setminus\mathfrak{M} \cup \gamma_{c, 1}$. We emphasize here that $u(z)$ defined in such a way has no jumps on $(-\infty, -1) \cup (1, \infty)$. From \eqref{eq: normalization of omega} and \eqref{eq: tau def} it follows that
				\begin{subequations}\label{eq: u jumps}
					\begin{alignat}{2}
						& u_+(z) + u_-(z) = 0, \qquad &&z \in \gamma_{m,1}, \\
						& u_+(z) + u_-(z) = \tau, \qquad &&z \in \gamma_{m,0}, \\
						&u_+(z) - u_-(z)= 1, \qquad &&z \in \gamma_{c,1}.
					\end{alignat} 
				\end{subequations}
				\begin{remark}\label{g-tilde-remark}
				Observe that $u(z)$ defined in this way satisfies $\tilde{g}(z) = 2\pi i \eta_1 u(z)$. To see this, consider the function $f(z) := \tilde{g}(z)  - 2\pi i \eta_1 u(z)$. From the behavior of $\tilde{g}(z),  u(z)$, the function $f(z)$ is bounded $z \to z_0, z_0 \in \Lambda \cup \{ \infty\}$. From \eqref{eq : g-tide-jump}, \eqref{eq: u jumps}, we see that $f_+(z) = -f_-(z)$ for $z \in \mathfrak{M}$ and is otherwise analytic. Applying Liouville's theorem to $f(z)/ \varXi(z)$ yields the claim.
				\end{remark}
				Next we set
				\begin{equation}
					\mathcal{M}_1(z, d) := \frac{\Theta(u(z)- W + d)}{\Theta(u(z)+d)}, \qquad \mathcal{M}_2(z, d) := \frac{\Theta(-u(z)- W + d)}{\Theta(-u(z)+d)},
				\end{equation}
				where we recall that $W = n(\omega_0+\Delta_0)$ and $d$ is yet to be determined. Then, both $\mathcal{M}_1$ and $\mathcal{M}_2$ are single valued on $\C\setminus\mathfrak{M}$. Equations \eqref{eq: theta properties} and \eqref{eq: u jumps} immediately show that the functions $\mathcal{M}_1$ and $\mathcal{M}_2$ satisfy \eqref{eq: Mcal jumps}, as desired. In the remainder of this section, we will slightly abuse notation and think of the functions $\phi^2(z)$ and $\mathcal{M}_{1, 2}(z)$ as functions on $\mathfrak{R}$. The latter are multiplicatively multi-valued on $\mathfrak{R}$, but one may still consider the order of zeros and poles in the usual fashion.
				
			\subsubsection{Step 4 - Choose $d$ and normalize $\mathcal{L}$}
				We have now constructed $\mathcal{M}_1$ and $\mathcal{M}_2$ so that $\mathcal{L}$ defined in \eqref{eq: L cal defn} satisfies \eqref{eq: model problem jump condition 1} and \eqref{eq: model problem jump condition 2}. We must now choose $d$ so that $\mathcal{L}$ is analytic in $\C\setminus\mathfrak{M}$ and normalize $\mathcal{L}$ so that it tends to the identity as $z\to\infty$. By standard theory \cite{farkas1992riemann}, for arbitrary $d\in\C$ the function $\Theta(u(z)-d)$ on $\mathfrak{R}$ either vanishes identically or vanishes at a single point $p_1$, counted with multiplicity. Recall that we have defined $z_1$ to be the unique finite solution to $\phi(z)^2 -1 =0$ and $\hat{z}_1$, its projection onto the opposite sheet of $\mathfrak{R}$, to be the unique finite solution to $\phi(z)^2+1=0$ on $\mathfrak{R}$. 
				
				We now choose $d$ so that the simple zeros of the denominators of $\mathcal{L}$ cancel the zeros of $\phi\pm\phi^{-1}$. From the remarks immediately following \eqref{eq: theta properties}, this is satisfied if we set
				\begin{equation}\label{eq: d def}
					d = -u\left(\hat{z}_1\right)+ \frac{1}{2} + \frac{\tau}{2} \mod \Lambda_\tau,
				\end{equation}
				as $\Theta(\zeta)= 0$ when $\zeta = \frac{1}{2} + \frac{\tau}{2} \mod \Lambda_\tau$. For definiteness, we choose $d = -u\left(\hat{z}_1\right)+ \frac{1}{2} + \frac{\tau}{2}$. As the Theta function is even, we have that
				\begin{equation}
					\Theta(u(\hat{z}_1)+ d)
					 = \Theta(-u(z_1)+ d) = \Theta(u(z_1)-d) = 0, 
				\end{equation}
				which verifies that each entry of $\mathcal{L}$ is analytic in $\C\setminus \mathfrak{M}$. 
				
				Now we must normalize $\mathcal{L}$ so that it decays to the identity as $z\to\infty$. We first note that we have alternative formula for $d$, 
				\begin{equation}\label{eq: d at infinity}
					d = -u(\infty_1) \mod \Lambda_\tau.
				\end{equation}
				To see this, we note that $\phi^2(z)-1$ is meromorphic on $\mathfrak{R}$ with a zero at $\infty_1$, a simple zero at $z_1$, and poles at $\lambda_0$ and $1$. By Abel's Theorem \cite[Theorem~III.6.3]{farkas1992riemann}, we have that
				\begin{equation*}
					u(\infty_1) + u(z_1) - u(1) - u(\lambda_0)=0 \mod \Lambda_\tau.
				\end{equation*}
				Using \eqref{eq: Abel map}, along with \eqref{eq: normalization of omega} and \eqref{eq: tau def}, we see that
				\begin{equation}
					u(1)= 0, \qquad u(\lambda_0) = -\frac{1}{2} -\frac{\tau}{2}, 
				\end{equation}
				so that \eqref{eq: d at infinity} follows by \eqref{eq: d def}. As $\phi(z)-\phi(z)^{-1}\to 0$ as $z\to\infty$, 
				\begin{equation}
					\det\mathcal{L}(\infty) = \mathcal{M}_1(\infty, d)\mathcal{M}_2(\infty,-d) = \frac{\Theta^2(W)}{\Theta^2(0)}. 
				\end{equation}
				As $\mathcal{L}$ has the same jumps as $M_0$ in \eqref{eq: model problem jump condition 1} and \eqref{eq: model problem jump condition 2}, we can conclude that $\det \mathcal{L}$ is entire, and as $\mathcal{L}$ is bounded at infinity, we have that
				\begin{equation}
					\det \mathcal{L}(z) = \frac{\Theta^2(W)}{\Theta^2(0)}. 
				\end{equation}
				If $\Theta(W)\not =0$, then
				\begin{equation}
					M_0(z) = \mathcal{L}^{-1}(\infty)\mathcal{L}(z)
				\end{equation}
				solves \eqref{eq: M0 RHP genus 1}. The condition $\Theta(W)\not=0$ can be rewritten as
				\begin{equation}\label{eq:n0w0D0}
					n(\omega_0+\Delta_0)\not= \frac{1}{2} + \frac{\tau}{2} \mod \Lambda_\tau. 
				\end{equation}
    				In the genus 1 case, the fact that $\mathcal{L}$ in \eqref{eq: L cal defn} is well defined implies that the previous condition is in fact necessary and sufficient; to see this, we note that the solution of the RHP \eqref{eq: M0 RHP genus 1} is unique, but when condition \eqref{eq:n0w0D0} is not satisfied, given a solution $\widetilde{M}_0(z)$, the matrix $\widetilde{M}_0(z)+k\mathcal{L}(z)$ is a solution for any $k\in\mathbb{Z}$. Therefore, we have proven the following Lemma (see Theorem~2.17 of \cite{bertola2016asymptotics}).
				\begin{lemma}\label{lem: model solution genus 1}
					The model Riemann-Hilbert problem \eqref{eq: M0 RHP genus 1} has a solution if and only if 
					\begin{equation}
							n(\omega_0+\Delta_0)\not= \frac{1}{2} + \frac{\tau}{2} \mod \Lambda_\tau.
					\end{equation}
					Moreover, the solution is given by $M_0(z) = \mathcal{L}^{-1}(\infty)\mathcal{L}(z)$,	where $\mathcal{L}$ is defined in \eqref{eq: L cal defn}. 
				\end{lemma}
				Next, we will define the sequence of indices $\N(s, \epsilon)$. To do so, note that zeros of $\mathcal{M}_{1, n}(z, d), \mathcal{M}_{2, n}(z, -d)$, denoted $z_{n, 1}, z_{n, 2}$ respectively, are defined via the Jacobi inversion problem 
				\begin{equation}\label{Jacobi Inversion Problem}
				u(z_{n, i}) - (-1)^{i+1} n(\Delta_0 + \omega_0) - u(\infty_1)  = \dfrac{1}{2} + \dfrac{\tau}{2} \mod \Lambda_{\tau}.
				\end{equation}
				In particular, $z_{n, 1} = \infty_1 = z_{n, 2}$ exactly when $n(\Delta_0 + \omega_0) = \frac{1}{2} + \frac{\tau}{2} \mod \Lambda_{\tau}$. As such, we let \[\N(s, \epsilon) = \left\{n \in \N \ | \ z_{n, 1} \not \in \pi^{-1}(\{ z \ | \ |z| > 1/\epsilon \}) \cap \mathfrak{R}^{(1)} \right\} \] where $\pi: \mathfrak{R} \to \overline{\C}$ is the natural projection and $\mathfrak{R}^{(1)}$ is the first sheet. With this definition, we have the following lemma.
				\begin{lemma} \label{seq-prop} 
				For all $n \geq 1$ and $\epsilon>0$ small enough, if $n \not\in \N(s, \epsilon)$, then $n+1 \in \N(s, \epsilon)$.
				\end{lemma}
				\begin{proof}
				To begin with, observe that \eqref{Jacobi Inversion Problem} yields
				\[
				u(z_{n + 1, i}) - u(z_{n, i}) = (-1)^{i+1}(\Delta_0 + \omega_0) \mod \Lambda_{\tau}.
				\] 
				Let $\epsilon_0 >0$ be such that for all $\epsilon < \epsilon_0$, $n\not\in \N(s, \epsilon)$.  For the sake of a contradiction, $n + 1 \not \in \mathbb{N}(s, \epsilon)$. Then, taking $\epsilon \to 0$, the above equation immediately yields that $0 = \Delta_0 + \omega_0 \mod \Lambda_\tau$. However, by deforming the contour and using expansion \eqref{eq: h prime asymptotics genus 1}, one can check that 
				\[
				\dfrac{1}{2\pi i} h_+'(z; s) \ dz, \ z \in \mathfrak{M}
				\]
				is a positive probability measure (cf. \cite[Theorem 2.3]{celsus2020supercritical}) and $\Delta_0 = -\tau \eta_1$ where $\eta_1$ is the measure of $\gamma_{m, 1}$. Hence, $\eta_1 \in (0, 1)$ and since $\omega_0 \in \R$, we have $\Delta_0 + \omega_0 \neq 0 \mod \Lambda_\tau$ and thus have reached a contradiction. 
				\end{proof}
				Let us pause here to note that the matrix $M(z)$ depends on $n$, and we now show that for large $n \in \N(s, \epsilon)$, $M(z)$ remains bounded. Write $n \omega_0 = \{ n\omega_0 \} + [n\omega_0], n \eta_1 = \{ n\eta_1 \} + [n\eta_1] $, where $\{x\}, [x]$ are the integer and fractional parts of $x \in \R$, respectively. Applying \eqref{eq: theta properties} and using the fact that $g(z) = 2\pi i \eta_1 u(z)$ (see Remark \ref{g-tilde-remark}) shows that the expressions dependent on $n$ in $M(z)$ are of the form
				\[e^{ \pm 2\pi i \{n \eta_1\} (u(z) \pm  u(\infty))} \dfrac{\Theta \left( \pm u(z) - \{n\omega_0\} - \{n \eta_1\}\tau \pm d \right)}{\Theta \left( \pm u(\infty) - \{n\omega_0\} - \{n \eta_1\}\tau \pm d \right)} \dfrac{\Theta \left( \pm u(\infty) \pm d \right)}{\Theta \left( \pm u(z) \pm d \right)},\] 
				where the choice of sign in each instance depends on the entry of $M(z)$ being considered. Since quantities $\{ n\omega_0 \}, \{ n\eta_1 \}$ remain bounded, we conclude that along any convergent subsequence, the sequence of functions $\{M(z) \}_{n \in \N(s, \epsilon)}$ is uniformly bounded as $n \to \infty$.
		\subsection{The Local Parametrices}\label{sec: Local Parametrices}
			Recalling the discussion preceding \eqref{eq: P RHP}, we will need a more detailed local analysis about the endpoints $\lambda\in\Lambda$. Although these constructions are now standard, we state them below for completeness. For details we refer the reader to \cite{bleher2011lectures,deift1999orthogonal,deift1999strong,arnojacobi}. 
			\subsubsection{Soft Edge}
				In light of \eqref{eq: RHP for h}, let $\lambda \in \Lambda$ be such that  $\Re h(z) = c\left(z-\lambda\right)^{3/2} + \mathcal{O}\left(\left(z-\lambda\right)^{5/2} \right)$ as $z\to\lambda$ for some $c\not=0$. We will also make use of the following function
				\begin{equation}
					h^{(\lambda)}(z) = \int_{\lambda}^z h'(z;s)\, ds,
				\end{equation}
				where the path of integration emanates upwards in the complex plane from $\lambda$ and does not cross $\Omega(s)$. Then, 
				\begin{equation}
					h^{(\lambda)}_\pm(z) = c\left(z-\lambda\right)^{3/2}+\mathcal{O}\left(\left(z-\lambda\right)^{5/2}\right), \qquad z\to\lambda,
				\end{equation}
				where $c\not=0$. There exist real constants $K^{\lambda}_\pm$ such that
				\begin{equation}
					h^{(\lambda)}_\pm(z) = h(z) + i K^{\lambda}_\pm
				\end{equation} 
				where in light of \eqref{eq: RHP for h}, $K^\lambda_+-K^\lambda_-=4\pi i \eta_1$.
				
				We assume $\lambda=\lambda_0$ so that the main arc $\gamma_{m,0}$ lies to the left of $\lambda$ and the complementary arc $\gamma_{c,1}$ lies to the right of $\lambda$, where left and right are in reference to the orientation of $\hat{\Sigma}$. The case where the complementary arc leads into $\lambda$ and the main arc exits $\lambda$ can be handled similarly with minor alterations. 
				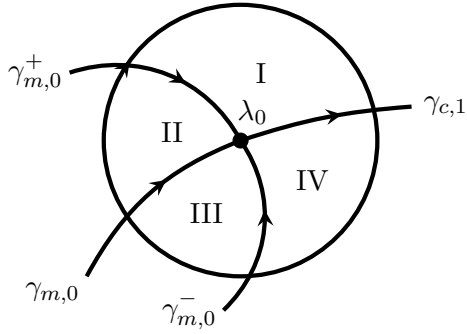
\begin{figure}[h]
					\centering
					\begin{tikzpicture}[scale=4.5]
					\draw[ultra thick, postaction={mid arrow=black}] (0,0) arc (0:-360:0.4 and 0.4);
					\draw [fill] (-0.4,0) circle [radius=0.022];
					\node at (-0.365,0.085) {$\lambda_0$};
					
					\draw [ultra thick,postaction={mid arrow=black}] (-0.45,-0.5) to [bend right=40] (-0.4, 0);
					\draw [ultra thick,postaction={mid arrow=black}] (-0.9,0.2) to [bend left=40] (-0.4, 0);
					\draw [ultra thick,postaction={mid arrow=black}] (-0.85,-0.4) to [bend left=20] (-0.4, 0);
					\draw [ultra thick,postaction={mid arrow=black}] (-0.4,0) to [bend left=7] (0.1, 0.1);
					
					\node at (0.2,.10) {$\gamma_{c,1}$};
					\node at (-0.55,-.5) {$\gamma_{m,0}^-$};
					\node at (-0.95,-.45) {$\gamma_{m,0}$};
					\node at (-1,.2) {$\gamma_{m,0}^+$};
					
					\node at (-0.34,.20) {I};
					\node at (-0.6,.025) {II};
					\node at (-0.5,-.215) {III};
					\node at (-0.19,-.115) {IV};
					\end{tikzpicture}
					\caption{Definition of Sectors I, II, III, and IV within $D_{\lambda_0}$.}
					\label{fig: defn of I II III IIV}
				\end{figure}
				We want to solve the following Riemann-Hilbert problem in a neighborhood of $\lambda_0$, $D_{\lambda_0}$, 
				\begin{subequations}\label{eq: P RHP soft}
					\begin{alignat}{2}
						&P^{(\lambda_0)}(z) \text{ is analytic for } z\in D_{\lambda_0}\setminus \hat{\Sigma}, \qquad && \\
						&P^{(\lambda_0)}_+(z) = P^{(\lambda_0)}_-(z) j_S(z),  \qquad &&z \in D_{\lambda_0}\cap \hat{\Sigma}, \\
						&P^{(\lambda_0)}(z)= \left(I + \mathcal{O}\left(\frac{1}{n}\right)\right)M(z), \qquad && n \to \infty, \quad z\in \partial D_{\lambda_0},\label{eq: matching condition soft edge}
					\end{alignat}
				\end{subequations}
				where $j_S(z)$ is as in \eqref{eq: S RHP}.
				We also require that $P^{(\lambda_0)}$ has a continuous extension to $\overline{D}_{\lambda_0}\setminus\hat{\Sigma}$ and remains bounded as $z\to\lambda_0$. $P^{(\lambda_0)}(z)$ is given by
				\begin{equation}
					P^{(\lambda_0)}(z) = E_n^{(\lambda_0)}(z) A\left(f_{n,A}(z)\right)e^{-\frac{n}{2}h(z)\sigma_3}
				\end{equation}
				where $A(\zeta)$ is built out of Airy functions as in \cite{bleher2011lectures,deift1999strong}. Above, 
				\begin{equation}
					f_{n,A}(z) = n^{2/3}f_A(z), \qquad f_A(z) = \left[-\frac{3}{4} h^{(\lambda)}(z)\right]^{2/3}, 
				\end{equation}
				so that $f_A(z)$ conformally maps a neighborhood of $\lambda_0$ to a neighborhood of $0$. Recall that we still have the freedom to choose the precise description of $\gamma_{m,0}^\pm$, so we choose them in $D_{\lambda_0}$ so they are mapped to the rays $\{z: \arg z = \pm \frac{2\pi}{3}\}$, respectively, under the map $f_A$. $E_{n}^{(\lambda_0)}(z)$ is the analytic prefactor chosen to satisfy the matching condition \eqref{eq: matching condition soft edge} and is given by 
				\begin{align}\label{eq: En eqn}
					E_n^{(\lambda_0)}(z) =\begin{cases}
					M(z) e^{-\frac{1}{2}niK^\lambda_+\sigma_3}L^{(\lambda_0)}_n(z)^{-1}, \qquad & z\in \text{I}, \text{II}, \\
					M(z) e^{-\frac{1}{2}niK^\lambda_-\sigma_3}L^{(\lambda_0)}_n(z)^{-1}, \qquad & z\in \text{III}, \text{IV},
					\end{cases}
				\end{align}
				where Sectors I, II, III, and IV are defined in Figure~\ref{fig: defn of I II III IIV}. Here, 
				\begin{equation*}
					L^{(\lambda_0)}_n(z) = \frac{1}{2\sqrt{\pi}}n^{-\sigma_3/6}f_A(z)^{-\sigma_3/4}\begin{pmatrix}
					1 & i \\
					-1 & i
					\end{pmatrix}.
				\end{equation*}
				In the formulas above, the branch cut for $f_A^{1/4}$ is taken on $\gamma_{m,0}$ and is the principal branch.

			\subsubsection{Hard Edge}
				Now we assume that we are looking at the analysis near $z=1$, and we recall that $\Re h(z) = \mathcal{O}\left(\left(z-1\right)^{1/2}\right)$ as $z\to 1$. We will show in the construction of $h$ in Section~\ref{sec: Existence of h} that
				\begin{equation}
					h(z) = c\left(z-1\right)^{1/2}+\mathcal{O}\left(\left(z-1\right)^{3/2}\right), \qquad z\to 1,
				\end{equation}
				for some $c\not=0$. Recall that we wish to solve
				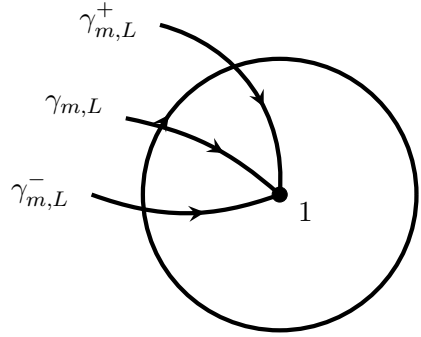
\begin{figure}[h]
					\centering
					\begin{tikzpicture}[scale=4.5]
					\draw[ultra thick, postaction={mid arrow=black}] (0,0) arc (0:-360:0.4 and 0.4);
					\draw [fill] (-0.4,0) circle [radius=0.022];
					
					\draw [ultra thick,postaction={mid arrow=black}] (-0.95,0) to [bend right=20] (-0.4, 0);
					\draw [ultra thick,postaction={mid arrow=black}] (-0.75,.5) to [bend left=40] (-0.4, 0);
					\draw [ultra thick,postaction={mid arrow=black}] (-0.85,0.225) to [bend left=15] (-0.4, 0);
					
					\node at (-1.1,0.01) {$\gamma_{m,L}^-$};
					\node at (-1,.26) {$\gamma_{m,L}$};
					\node at (-0.9,.51) {$\gamma_{m,L}^+$};
					\node at (-0.325,-0.05) {$1$};
					\end{tikzpicture}
					\caption{Structure of $\hat{\Sigma}$ in $D_1$.}
					\label{fig: hard edge}
				\end{figure}
				
				\begin{subequations}\label{eq: P RHP hard}
					\begin{alignat}{2}
					&P^{(1)}(z) \text{ is analytic for } z\in D_{1}\setminus \hat{\Sigma}, \qquad && \\
					&P^{(1)}_+(z) = P^{(1)}_-(z) j_S(z),  \qquad &&z \in D_{1}\cap \hat{\Sigma}, \\
					&P^{(1)}(z)= \left(I + \mathcal{O}\left(\frac{1}{n}\right)\right)M(z), \qquad && n \to \infty, \quad z\in \partial D_{1},\label{eq: matching condition hard edge}
					\end{alignat}
				\end{subequations}
				where $P^{(1)}$ has a continuous extension to $\overline{D}_1\setminus\hat{\Sigma}$ and remains bounded as $z\to 1$, and where the jump matrix $j_S$ in $D_1$ is given by
				\begin{align}
					j_S(z) = \begin{cases}
					\begin{pmatrix}
					1 & 0 \\
					e^{-nh(z)} & 1
					\end{pmatrix}, \qquad & z \in \gamma_{m,L}^{\pm}, \\
					\begin{pmatrix}
					0 & 1 \\
					-1 & 0
					\end{pmatrix}, \qquad & z \in \gamma_{m,L}. \\
					\end{cases}
				\end{align}
				Analogously to the analysis in the soft edge, we define $P^{(1)}(z) = U^{(1)}(z) e^{-\frac{n}{2}h(z)\sigma_3}$, so that $U^{(1)}$ solves a new Riemann-Hilbert problem in $D_1$, with jump matrix given by
				\begin{align}\label{eq: U jumps hard edge}
					j_{U^{(1)}}(z) = \begin{cases}
					\begin{pmatrix}
					1 & 0 \\
					1 & 1
					\end{pmatrix}, \qquad & z \in \gamma_{m,L}^{\pm}, \\
					\begin{pmatrix}
					0 & 1 \\
					-1 & 0
					\end{pmatrix}, \qquad & z \in \gamma_{m,L}.
					\end{cases}
				\end{align}
				Now, $U^{(1)}$ can be written explicitly in terms of Bessel functions, as in \cite{arnojacobi}, and we state this construction below. First set
				\begin{subequations}
					\begin{equation}
					b_1(\zeta) = H_0^{(1)}\left(2\left(-\zeta\right)^{1/2}\right), \qquad b_2(\zeta) = H_0^{(2)}\left(2\left(-\zeta\right)^{1/2}\right),
					\end{equation}
					\begin{equation}
					b_3(\zeta) = I_0\left(2\zeta^{1/2}\right), \qquad b_4(\zeta) = K_0\left(2\zeta^{1/2}\right),
					\end{equation}
				\end{subequations}
				where $I_0$ is the modified Bessel function of the first kind, $K_0$ is the modified Bessel function of the second kind, and $H_0^{(1)}$ and $H_0^{(2)}$ are Hankel functions of the first and second kind, respectively. With this in hand, we may define the Bessel parametrix as
				\begin{align}\label{eq: Bessel parametrix}
					B(\zeta) = \begin{cases}
					\begin{pmatrix}
					\frac{1}{2} b_2\left(\zeta\right) & -\frac{1}{2}b_1\left(\zeta\right) \\
					-\pi z^{1/2} b_2'\left(\zeta\right) & \pi z^{1/2} b_1'\left(\zeta\right)
					\end{pmatrix}, \qquad & -\pi < \arg \zeta < -\frac{2\pi}{3}, \\
					\begin{pmatrix}
					b_3(\zeta) & \frac{i}{\pi}b_4(\zeta) \\
					2\pi i z^{1/2} b_3'(\zeta) & -2z^{1/2} b_4'(\zeta)
					\end{pmatrix}, \qquad & \left|\arg \zeta\right| < \frac{2\pi}{3}, \\
					\begin{pmatrix}
					\frac{1}{2} b_1\left(\zeta\right) & \frac{1}{2}b_2\left(\zeta\right) \\
					\pi z^{1/2} b_1'\left(\zeta\right) & \pi \zeta^{1/2} b_2'\left(\zeta\right)
					\end{pmatrix}, \qquad & \frac{2\pi}{3} < \arg \zeta < \pi.
					\end{cases}
				\end{align}
				Using the conformal map, $f_{n,B}$, where
				\begin{equation}\label{eq: Bessel conformal map}
				f_{n,B}(z) = n^2 f_B(z), \qquad \text{where} \qquad f_B(z) = \frac{h(z)^2}{16},
				\end{equation}
				the matrix $U^{(1)}$ is given by
				\begin{equation}\label{eq: U4 eqn}
					U^{(1)}(z) = E_n^{(1)}(z) B(f_{n,B}(z)),
				\end{equation}
				where $E_n^{(1)}$ is analytic prefactor chosen to ensure the matching condition \eqref{eq: matching condition hard edge}. Therefore, we have that 
				\begin{equation}
					E_n^{(1)}(z) = M(z) L_n^{(1)}(z)^{-1}, 
				\qquad			
				L_n^{(1)}(z) := \frac{1}{\sqrt{2}} \left(2\pi n\right)^{-\sigma_3/2} f_B(z)^{-\sigma_3/4}
				\begin{pmatrix}
				1 & i \\
				i & 1
				\end{pmatrix},
				\end{equation}
				where all branch cuts above are again taken to be principal branches. 
				
				A similar analysis may be conducted around $z=-1$, and we state the solution to the local parametrix here is given by
				\begin{equation}\label{eq: p-1 Bessel}
					P^{(-1)}(z) = E_n^{(-1)}(z) \tilde{B}\left(\tilde{f}_{n,B}(z)\right)e^{-\frac{n}{2}h(z)}
				\end{equation}
				where $\tilde{B}(z)=\sigma_3 B(z) \sigma_3$,
				\begin{equation}
					\tilde{f}_{n,B}(z) = n^2 \tilde{f}_B(z), \qquad \tilde{f}_B(z) = \frac{\tilde{h}(z)^2}{16},
				\end{equation}
				and $\tilde{h}(z)=h(z)-2\pi i$. Similarly, we have
				\begin{equation}
				E_n^{(-1)}(z) = M(z) L_n^{(-1)}(z)^{-1}, 
				\qquad			
				L_n^{(-1)}(z) := \frac{1}{\sqrt{2}} \left(2\pi n\right)^{-\sigma_3/2} \tilde{f}_B(z)^{-\sigma_3/4}
				\begin{pmatrix}
				-1 & i \\
				i & -1
				\end{pmatrix}.
				\end{equation}
			
\section{The Global Phase Portrait - Continuation in Parameter Space}\label{sec: Existence of h}
	As seen above, one of the keys to implementing the Deift-Zhou method of nonlinear steepest descent is the existence of the $h$-function. Fortunately, genus $0$ and $1$ solutions for $s \in i \R$ have already been established in \cite{deano2014large, celsus2020supercritical}, so we can implement the continuation in parameter space technique developed in \cite{bertola2016asymptotics, tovbis2010nonlinear,tovbis2004semiclassical}. By following this procedure, we will show that by starting with some genus $L$ $h$-function for $s\in i\R \cap \mathfrak{G}_L$, we will  be able to continue this genus $L$ solution to all $s \in \mathfrak{G}_L$.
	
	Below, we will first define breaking points and breaking curves. The set of breaking curves along with their endpoints will be denoted as $\mathfrak{B}$ and we will show that the inequalities \eqref{eq: both inequalities} can only break down as we cross a breaking curve. Next, we provide the basic background on quadratic differentials needed for our analysis. Finally, we recap the previous work on orthogonal polynomials of the form \eqref{eq: ortho p defn} where $s\in i \R$ and show how we may deform these solutions to all $s \in \C\setminus\mathfrak{B}$.

	\subsection{Breaking Curves}
		We define a breaking point as follows: $s_b \in \mathbb{C}$ is a breaking point if there exists a saddle point $z_0 \in \Omega(s)$ such that 
		\begin{equation}\label{eq: def of breaking point}
			h'(z_0; s_b) = 0, \qquad \text{ and } \qquad \Re h(z_0; s_b) = 0.
		\end{equation} 
		Above, we also impose that the zero of $h'$ is of at least order 1. We call a breaking point \textit{critical} if either:
		\begin{enumerate}[(i)]
			\item The saddle point in \eqref{eq: def of breaking point} coincides with a branchpoint in $\Lambda(s)$, or
			\item The order of the zero at the saddle point is greater than one or there are at least two saddle points of $h$ on $\Omega$ counted with multiplicity. 
		\end{enumerate}
		If a breaking point $s$ is not a critical breaking point, it is a \textit{regular breaking point}. 
		\begin{remark}
			Note that $h'$ is analytic in $\mathbb{C}\backslash\mathfrak{M}(s)$. In the above definition of breaking point, if $z_0\in\mathfrak{M}(s)$, we mean $h'(z_0) = 0$ in the following sense. Note that $h'_+(z)$ and $h'_-(z)$ have analytic extensions to a neighborhood of $z_0\in\mathfrak{M}(s)$. Moreover, in this neighborhood, the two extensions are related via $h'_+(z) = -h'_-(z)$. Therefore, if $z_0$ is such that $h'_+(z_0) = 0$ (where here we are referring to the extension so this is well defined), then $h'_-(z_0)=0$, so we say $h'(z_0) = 0$.
		\end{remark}
		We have the following lemma from \cite[Lemma~4.3]{bertola2016asymptotics}, and we include the proof for convenience. 
		\begin{lemma}\label{lem: def of breaking curves}
			Let $s=s_1+is_2$ where $s_1, s_2 \in \R$ and let $s_b$ be a regular breaking point. If both $\partial_{s_k} h\left(z_0;s_b\right)$, for $k=1,2$, exist and at least one of them is $ \not=0$, then there exists a smooth curve passing through $s_b$ consisting of breaking points. 
		\end{lemma}
		\begin{proof}
			Writing $z = u+iv$ and $s=s_1+is_2$, we may consider \eqref{eq: def of breaking point} to be a system of 3 real equations in 4 real unknowns in the form $G(u,v,s_1,s_2)=0$. We may choose either $j=1$ or $j=2$ so that  $\Re \partial_{s_j} h(z_0;s_b)\not=0$. Then, as $h'(z_0;s_b)=0$, we may calculate the Jacobian as
			\begin{align*}
				\det\left(\frac{\partial G}{\partial (u,v,s_j)}\right) &= i^{j-1} \Re h_{s_j}\left(z_0;s_b\right)\begin{vmatrix}
					\frac{\partial}{\partial u} \Re h'(z_0;s_b) & \frac{\partial}{\partial v} \Re h'(z_0;s_b) \\
					\frac{\partial}{\partial u} \Im h'(z_0;s_b) & \frac{\partial}{\partial v} \Im h'(z_0;s_b)
				\end{vmatrix} \\
				&= i^{j-1} \Re h_{s_j}\left(z_0;s_b\right) \left|h''(z_0;s_b)\right|^2, 
			\end{align*}
			where we have used the Cauchy-Riemann equations for the second equality above. As $h''\not=0$, as $s_b$ is a regular breaking point, the Implicit Function Theorem completes the proof.
		\end{proof}
		The curves in Lemma~\ref{lem: def of breaking curves} are defined to be \textit{breaking curves}. We will see that the breaking curves partition the parameter space so as to separate regions of different genus of $h$ function, as they are precisely where the inequalities on $h$ break down. Assume that $h(z;s)$ satisfies the scalar Riemann--Hilbert problem \eqref{eq: RHP for h}.
		\begin{lemma}\label{lem: breaking point lemma}
			Let $s(t)$ for $t\in [0,1]$ be a smooth curve in the parameter space starting from $s_0 = s(0)$ and ending at $s_1 = s(1)$. Assume further that $s(t)$ is a regular point for all $0\leq t<1$, that is, the inequalities \eqref{eq: both inequalities} are satisfied for $0\leq t<1$, and that $\Re h(z;s)$ is a continuous function of $s$. Then, the inequalities \eqref{eq: both inequalities} do not hold at $s_1$ if and only if $s_1$ is a breaking point. 
		\end{lemma}
		\begin{proof}
			To see this, first consider the case that the Inequality~\eqref{eq: main arc inequality} breaks down in a vicinity of $z_0$, where $z_0$ is an interior point of a main arc. By definition, $\Re h(z;s)=0$ for $s=s(t)$, $0\leq t<1$ and for all interior points $z$ of a main arc, so by continuity we must have that $\Re h(z_0;s_1)=0$. To show that $s_1$ is a breaking point, we must just show that $h'(z_0; s_1)=0$. To get a contradiction, assume that $h'(z_0;s_1)\not=0$. As $h_+$ is analytic at $z_0$ and its derivative doesn't vanish, we may write that $h_+'(z) = c + (z-z_0)a(z)$, where $a$ is analytic in a neighborhood of $z_0$ and does not vanish in this neighborhood and $c\neq 0$, which implies that the map is conformal. Therefore, $\Re h_+(z)$ does not change sign in close proximity to $z_0$ on the $+$-side of the cut, and as $h=h_+$ here, the real part of $h$ does not change on the $+$ side of the cut in close proximity of $z_0$. A similar argument applied to $h_-$ shows that the real part of $h$ does not change on the $-$-side of the cut in close proximity of $z_0$, either. As $\Re h(z;s(t))>0$ for all $z$ in close proximity of a main arc for $t<1$, we have that by continuity in $s$ and by the constant sign of $\Re h(z;s_1)$ in close proximity to $z_0$ that $\Re  h(z;s_1)>0$ for all $z$ in close proximity to $z_0$. This is precisely the inequality which we have assumed to have broken down, so we have reached the desired contradiction. As such $h'(z_0;s_1)=0$, and $s_1$ is a breaking point. Going the other way, we have that the real part of $h_+$ must change sign above/below the cut if $h'_\pm(z_0)=0$, which clearly violates Inequality~\eqref{eq: main arc inequality}. 
			
			Next, assume that Inequality~\eqref{eq: complementary arc inequality} breaks down at $z_0$, where $z_0$ is an interior point of a complementary arc, $\gamma_c$. Given that $\Re h(z; s(t))<0$ for all interior points of a complementary arc, we have by continuity that if the inequality breaks down for $s_1$ at some point $z_0$, we must have that $\Re h(z_0;s_1) = 0$. We are now left to show that $h'(z_0)=0$. To get a contradiction, assume that $h'(z_0)\not=0$. Then there is a zero level curve of $\Re h(z)$ passing through $z_0$ which looks locally like an analytic arc (that is, no intersections). Furthermore, the sign of $\Re h(z)$ is constant on either side of $\gamma_c$ in close proximity to $z_0$. By continuity, we have that $\Re h(z;s_1)<0$ for all interior points $z \in \gamma_c\backslash\{z_0\}$. Therefore, we are able to deform the complementary arc back into the region where $\Re h(z) <0$ for all $z \in \gamma_c$, contradicting the assumption the inequality was violated. Therefore, we must have that $h'(z_0;s_1)=0$, and as such $s_1$ is a breaking point. On the other hand, assume that $s_1$ is a breaking point. Then as $\Re h(z_0;s_1) = 0$, we clearly have that the strict inequality \eqref{eq: complementary arc inequality} is violated at $z_0$. Moreover, the condition that $h'(z_0) =0$ enforces that we can not deform the complementary arc so as to fix the inequality.
		\end{proof}

	\subsection{Quadratic Differentials}\label{sec: Quadratic Differentials}
		In this subsection, we review the basic theory of quadratic differentials needed for the subsequent analysis. The theory presented below follows \cite{guilhermeThesis, strebel1984quadratic}, and we refer the reader to these works for complete details.  
		
		A meromorphic differential $\varpi$ on a Riemann surface $\mathfrak{R}$ is a second order form on the Riemann surface, given locally by the expression $-f(z)\, dz^2$, where $f$ is a meromorphic function of the local coordinate $z$. In particular, if $z=z(\zeta)$ is a conformal change of variables, 
		\begin{equation}
			-\tilde{f}(\zeta) \, d\zeta^2 = -f(z(\zeta))z'(\zeta)^2 \, d\zeta^2
		\end{equation}
		represents $\varpi$ in the local coordinate $\zeta$. In the present context, we may always take the underlying Riemann Surface to be the Riemann sphere. Of particular interest to us is the \textit{critical graph} of a quadratic differential $\varpi$, which we explain below.
		
		First, we define the \textit{critical points} of $\varpi = -f \, dz^2$ to be the zeros and poles of $-f$. The order of the critical point, $p$, is the order of the zero or pole, and is denoted by $\eta(p)$. Zeros and simple poles are called \textit{finite critical points}; all other critical points are \textit{infinite}. Any point which is not a critical point, is a \textit{regular point}. 
		
		In a neighborhood of any regular point $p$, the primitive
		\begin{equation}
			\Upsilon(z) = \int_p^z \sqrt{-\varpi} = \int_{p}^{z} \sqrt{f(s)}\, ds
		\end{equation}
		is well defined by specifying the branch of the root at $p$ and analytically continuing this along the path of integration. Then, we define an arc $\gamma\subset\mathfrak{R}$ to be an \textit{arc of trajectory} of $\varpi$ if it is locally mapped by $\Upsilon$ to a vertical line. Equivalently, for any point $p\in\gamma$, there exists a neighborhood $U$ where $\Upsilon$ is well defined and moreover, $\Re \Upsilon(z)$ is constant for $z \in \gamma \cap U$. A maximal arc of trajectory is called a \textit{trajectory} of $\varpi$. Moreover, any trajectory which extends to a finite critical point along one of its directions is called a \textit{critical trajectory} of $\varpi$ and the set of critical trajectories of $\varpi$, along with their limit points, is defined to be the \textit{critical graph} of $\varpi$. 
		
		To understand the topology of the critical graph of a quadratic differential $\varpi$, we must necessarily study both the local structure of trajectories near finite critical points, along with the global structure of the critical trajectories. Fortunately, the local behavior near a finite critical point is quite regular. Indeed, from a point $p$ of order $\eta(p)=m\geq -1$ emanate $m+2$ trajectories, from equal angles of $2\pi/(m+2)$ at $p$. This also includes regular points, which implies that through any regular point passes exactly one trajectory, which is locally an analytic arc. In particular, this implies that trajectories may only intersect at critical points. 
		
		The global structure of trajectories is more involved, and requires more detailed analysis. In general, a trajectory $\gamma$ is either
		\begin{enumerate}[(i)]
			\item a closed curve containing no critical points,
			\item an arc connecting two critical points (which may coincide), or
			\item an arc that has no limit along at least one of its directions.
		\end{enumerate}
		Trajectories satisfying $(iii)$ are called \textit{recurrent trajectories}, and their absence in the present work is assured by Jenkins' Three Poles Theorem \cite[Theorem~8.5]{pommerenke1975univalent}.
		
		With the necessary background on quadratic differentials now complete, we will see how their trajectories play a crucial role in the construction of the $h$-function. 
		
	\subsection{The Genus $0$ and $1$ $h$-functions}
		In this section, we review the previous work in the literature for polynomials of the form \eqref{eq: ortho p defn} where $s\in i \R$ and show how they can be extended to all $s\in \mathfrak{G}_0\cup\mathfrak{G}^{\pm}_1$, where these domains have been defined in Figure \ref{fig: breaking curve}.

			\subsubsection{Genus 0}
				The case where $s = -it$ and $0 < t < t_0$ was studied in \cite{deano2014large}. We recall that $t_0$ was defined as the unique positive solution to 
				\begin{equation}
					2 \log\left(\frac{2+\sqrt{t^2+4}}{t}\right) - \sqrt{t^2+4} = 0.
				\end{equation}
				We want to show that we may extend the results of \cite{deano2014large}, by using the technique of continuation in parameter space discussed above, to construct a genus $0$ $h$-function which satisfies both \eqref{eq: RHP for h} and \eqref{eq: both inequalities}. In order to state some of the results from \cite{deano2014large}, we first define
				\begin{equation}\label{eq: h prime genus 0}
					h'(z;s) = \frac{2-sz}{\left(z^2-1\right)^{1/2}}.
				\end{equation}
				Next, we consider the quadratic differential $\varpi_s := -h'(z;s)^2\, dz^2$. The following is a restatement of \cite[Theorem~2.1]{deano2014large}.
				\begin{lemma}
					Let $s=-it$ where $0<t<t_0$. There exists a smooth curve $\gamma_{m,0}(s)$ connecting $-1$ and $1$ which is a trajectory of the quadratic differential $\varpi_s$. 
				\end{lemma}
				With this lemma in hand, we take the branch cut of \eqref{eq: h prime genus 0} on $\gamma_{m,0}(s)$, with the branch chosen so that
				\begin{equation}\label{eq: h prime genus 0 asymptotics}
					h'(z;s) = -s +\frac{2}{z}+ \mathcal{O}\left(\frac{1}{z^2}\right), \qquad z \to\infty. 
				\end{equation}
				
				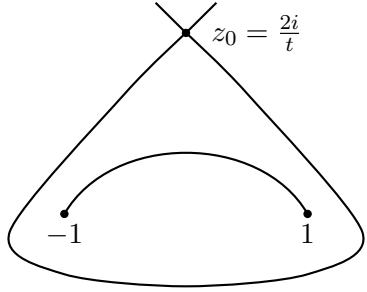
\begin{figure}[t]
					\centering
					\begin{tikzpicture}[scale=.8]
					\draw [thick] (-2,0) to[bend left=60] (2,0);

					\draw [thick] plot [smooth] coordinates {(0,3)(1,2.) (2.9,-.3) (2,-1) (0,-1.2) (-2,-1) (-2.9,-.3) (-1,2.) (0,3)};
					\draw [thick] (0,3) to (0.5,3.5);
					\draw [thick] (0,3) to (-0.5,3.5);
					\draw [fill] (0,3) circle [radius=0.06];
					\draw [fill] (2,0) circle [radius=0.06];
					\draw [fill] (-2,-0) circle [radius=0.06];
					\node [below] at (-2,-0) {$-1$};
					\node [below] at (2,-0) {$1$};
					\node [right] at (0.25,3) {$z_0 = \frac{2i}{t}$};				
					\end{tikzpicture}
					\caption{Critical Graph of $-h'^2 \, dz^2$ for $h'$ defined in \eqref{eq: h prime genus 0} and $s =-i t$ with $0 < t < t_0$. }
					\label{fig: subcritical kissing trajectories}
				\end{figure}
				The critical graph of $\varpi_s$ is depicted in Figure~\ref{fig: subcritical kissing trajectories}. We see that there are four trajectories emanating from the double zero at $z= 2i/t = 2/s$, two of which form a loop surrounding the endpoints $-1$ and $1$. We may easily extend this critical graph from the subset of the imaginary axis to all $s \in \mathfrak{G}_0$. 
				\begin{lemma}\label{lem: QD in genus 0}
					For all $s \in \mathfrak{G}_0$, there exists a smooth curve $\gamma_{m,0}(s)$ connecting $-1$ and $1$ which is a trajectory of the quadratic differential $\varpi_s$. 
				\end{lemma}
				\begin{proof}
					Fix some $s_0 = -it$ with $0<t<t_0$ and some $s_1 \in \mathfrak{G}_0$. The goal is to show that there exists a trajectory of $\varpi_{s_1}$ which connects $-1$ to $1$. As $\mathfrak{G}_0$ is the region bounded by the curves $\mathfrak{b}_{\pm}$, we may connect $s_0$ to $s_1$ with a curve that lies completely within $\mathfrak{G}_0$, which we call $\rho$. As we deform $s$ along $\rho$ towards $s_1$, we note that the topology of the critical graph of $\varpi_s$ will only change if a trajectory emanating from $2/s$ ever meets $\gamma_{m,0}(s)$. Assume for sake of contradiction, there existed some $s_*\in \rho$ for which this occurred. We would then have $\Re h(z;s_*)=0$ for $z \in \gamma_{m,0}(s)$, as it is a trajectory of the quadratic differential $\varpi_s$. Moreover, we would also have that $h'(2/s_*;s_*) = 0$ as $2/s_*$ is a zero of $h'(z;s_*)$. In other words, $s_*$ is a breaking point. However, this contradicts the fact that $\rho$ lies completely within $\mathfrak{G}_0$, which by definition contains no breaking points in its interior. As such, the topology of the critical graph at $s_1$ is the same as it was at $s_0$, and we conclude that there exists a trajectory of $\varpi_{s_1}$ connecting $-1$ and $1$. 
				\end{proof}
				
				In light of the lemma above, we keep the notation of $\gamma_{m,0}(s)$ to be the trajectory of $\varpi_s$ which connects $-1$ and $1$. We then have $\Omega(s) :=\gamma_{c,0} \cup\gamma_{m,0}(s)$, where we recall $\gamma_{c,0} =(-\infty, -1]$. Now, consider the function
				\begin{equation}\label{eq: genus 0 h function}
					h(z;s) = \int_1^z h'(u;s)\, du, 
				\end{equation}
				where the path of integration is taken in $\C\setminus\Omega(s)$. 
				\begin{lemma}
					Let $s \in \mathfrak{G}_0$. Then, $h(z;s)$ defined in \eqref{eq: genus 0 h function} solves the Riemann-Hilbert problem \eqref{eq: RHP for h} and satisfies the inequalities \eqref{eq: both inequalities}. 
				\end{lemma}
				\begin{proof}
					It is clear that $h$ is analytic in $\C\setminus\Omega(s)$. Next, note that $\Re h(z;s)\to 0$ as $z\to 1$ and $\Re h(z;s)$ is constant along $\gamma_{m,0}(s)$, as it is a trajectory of $\varpi_s$. Therefore, we have that $\Re h(z;s)=0$ for $z\in\gamma_{m,0}(s)$. As $h'_+ = - h'_-$ on $\gamma_{m,0}$, we we have that $h_+(z) + h_-(z) = 0$ for $z \in \gamma_{m,0}$, so that $h$ satisfies the appropriate jump over $\gamma_{m,0}$. Next, a residue calculation gives us that $h_+(z)-h_-(z) = 4\pi i$ for $z \in \gamma_{c,0}$.
					
					We can integrate \eqref{eq: genus 0 h function} directly to yield, 
					\begin{equation}\label{eq: genus 0 h function explicit}
						h(z;s) = 2 \log\left(z+\left(z^2-1\right)^{1/2}\right) - s\left(z^2-1\right)^{1/2}.
					\end{equation}
					From this, we can compute that
					\begin{equation}\label{eq: genus 0 h at infinity}
						h(z;s) = - sz + 2\log 2 + 2\log z + \mathcal{O}\left(\frac{1}{z}\right), \qquad z\to\infty, 
					\end{equation}
					so that $h$ satisfies \eqref{eq: h asymptotics}. Finally, it is clear from \eqref{eq: h prime genus 0} that $h(z) = \mathcal{O}\left(\sqrt{z\mp 1}\right)$ as $z\to\pm 1$, so that the $h$ constructed above satisfies all of the requirements of \eqref{eq: RHP for h}. 
					
					To see that $h(z;s)$ satisfies \eqref{eq: both inequalities}, we note that the inequalities were proven directly in \cite{deano2014large} for $s = -it$ with $0<t<t_0$. By using Lemma~\ref{lem: breaking point lemma}, we see that the inequalities will hold for all $s \in \mathfrak{G}_0$, completing the proof. 
				\end{proof}
				
				With the genus $0$ $h$-function now constructed explicitly for all $s\in\mathfrak{G}_0$, we now turn to the genus $1$ case. 				
				
			\subsubsection{Genus 1}
				The genus $1$ case is slightly more involved, but as before, we will deform the existing solution on the imaginary axis to all other values of $s$. Therefore, we start with defining
				\begin{equation}\label{eq: h prime genus 1}
					h'(z;s) = -s\left(\frac{(z-\lambda_0(s))(z-\lambda_1(s))}{z^2-1}\right)^{1/2}, 
				\end{equation}
				and we now set $\varpi_s := -h'(z;s)^2\, dz^2$, where $h'$ is defined in \eqref{eq: h prime genus 1}. It was shown in \cite{celsus2020supercritical} that for $s= -it$ where $t>t_0$, there exist trajectories of the quadratic differential $\varpi_s$ connecting $-1$ to $\lambda_0$ and $\lambda_1$ to $1$. Here, $\lambda_0$ and $\lambda_1$ satisfy 
				\begin{equation}\label{eq: endpoint boutroux}
					\lambda_0 + \lambda_1 = \frac{4}{s} ,\qquad \Re \oint_C h'(z)\, dz = 0, 
				\end{equation}
				and where $C$ is any loop on the Riemann surface $\mathfrak{R}$ associated to the algebraic equation $y^2 = (h')^2$, defined in Remark~\ref{rem: Riemann surface} and Subsection~\ref{sub: Genus 1 Global Parametrix}.
				Note that the first condition in \eqref{eq: endpoint boutroux} ensures that 
				\begin{equation}\label{eq: h prime asymptotics genus 1}
					h'(z) = -f'(z) + \frac{2}{z} + \mathcal{O}\left(\frac{1}{z^2}\right), \qquad z \to \infty. 
				\end{equation}
				The second condition of \eqref{eq: endpoint boutroux} is known as the Boutroux Condition, and its importance will become clear shortly. The critical graph of $\varpi_s$ for $s \in i\R \cap \mathfrak{G}_1^-$ as proven in \cite{celsus2020supercritical} is displayed in Figure~\ref{fig: supercritical kissing trajectories}. In this case, the critical graph is symmetric with respect to the imaginary axis, and there exists a trajectory connecting $-1$ to $\lambda_0$ and one connecting $\lambda_1 = -\overline{\lambda_0}$ to $1$. 

				\begin{figure}[t]
						\centering
						\begin{tikzpicture}[scale=.8]
						\draw [thick] (-1.5,1) to [out=10,in=170] (1.5,1);
						\draw [thick] (1.5,1) to [out=50,in=180] (4,1.7);
						\draw [thick] (-1.5,1) to [out=180-50,in=0] (-4,1.7);								   
						\draw [thick] (1.5,1) to [out=290,in=100] (2,-0.5);		 				
						\draw [thick] (-1.5,1) to [out=180-290,in=180-100] (-2,-0.5);		
						\draw [fill] (1.5,1) circle [radius=0.06];
						\draw [fill] (-1.5,1) circle [radius=0.06];
						\draw [fill] (2,-0.5) circle [radius=0.06];
						\draw [fill] (-2,-0.5) circle [radius=0.06];
						\node [below] at (-2,-0.5) {$-1$};
						\node [below] at (2,-0.5) {$1$};
						\node [left] at (-1.7,1) {$\lambda_0$};
						\node [right] at (1.7,1) {$\lambda_1$};

						\end{tikzpicture}
					\caption{Critical Graph of $-h'^2 \, dz^2$ for $h'$ defined in \eqref{eq: h prime genus 1} and $s \in i\R$ with $\Im s < -t_0$. }
					\label{fig: supercritical kissing trajectories}
				\end{figure}
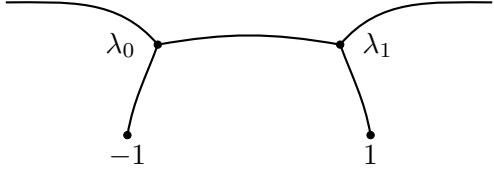

				We consider the case $s \in \mathfrak{G}_1^-$. In particular this means that $s$ is a regular point in the genus 1 region. As in the proof of Lemma~\ref{lem: QD in genus 0}, we note that for any $s \in \mathfrak{G}_1^-$, there will exist trajectories connecting $-1$ to $\lambda_0$ and $\lambda_1$ to $1$, which we define to be $\gamma_{m,0}(s)$ and $\gamma_{m,1}(s)$. Further, we define $\gamma_{c,1}$ to be the curve connecting $\lambda_0$ to $\lambda_1$ along which $\Re h(z) < 0$, whose existence is guaranteed by the definition of a regular point. 
				
				We now show that \eqref{eq: endpoint boutroux} holds for any $s\in\mathfrak{G}_1^-$. Denoting $\lambda_0 = u_0 + i v_0$ and $\lambda_1=u_1+iv_1$, we may write the conditions \eqref{eq: endpoint boutroux} as $F(s; u_0, v_0, u_1, v_1) = 0$, where $F = (f_1, f_2, f_3, f_4)$ and 
					\begin{equation*}
						f_1 = u_0 + u_1 - \Re\frac{4}{s}, \qquad f_2 = v_0 + v_1 - \Im\frac{4}{s}, \qquad f_3 = \Re \oint_A h'(z)\, dz, \qquad f_4 = \Re \oint_B h'(z)\, dz.
					\end{equation*}
				Note that $f_3=0$ and $f_4=0$ are equivalent to the Boutroux condition, as any loop on $\mathfrak{R}$ may be written as a combination of the $A$ and $B$ cycle on $\mathfrak{R}$. Taking the Jacobian of the above conditions with respect to the endpoints yields, 
				\begin{equation}
					\nabla F = \begin{pmatrix}
						1 & 0 & 1 & 0 \\
						0 & 1 & 0 & 1 \\
						\Re \oint_A h'_{\lambda_0} \, dz & \Im \oint_A h'_{\lambda_0} \, dz &\Re \oint_A h'_{\lambda_1} \, dz & \Im \oint_A h'_{\lambda_1} \, dz \\
						\Re \oint_B h'_{\lambda_0}\, dz & \Im \oint_B h'_{\lambda_0}\, dz & \Re \oint_B h'_{\lambda_1} \, dz & \Im \oint_B h'_{\lambda_1} \, dz 
					\end{pmatrix},
				\end{equation}
				where 
				\begin{equation}
					h_{\lambda_j}'(z) = \frac{-1}{2(z-\lambda_j)}h'(z), \qquad j=1,2.
				\end{equation}
				As $\lambda_0 \not=\lambda_1$ since we are at a regular point, note that
				\begin{equation}
					\left(h_{\lambda_1}'(z)-h_{\lambda_0}'(z)\right) \, dz
				\end{equation}
				is the unique (up to multiplicative constant) holomorphic differential on $\mathfrak{R}$. Subtracting the first and second columns from the third and fourth columns, we get that 
				\begin{equation}
					\det \nabla F = \det\begin{pmatrix}
						1 & 0 & 0 & 0 \\
						0 & 1 & 0 & 0 \\
						\Re \oint_A h'_{\lambda_0} \, dz & \Im \oint_A h'_{\lambda_0} \, dz &\Re \mathcal{A} & \Im \mathcal{A} \\
						\Re \oint_B h'_{\lambda_0}\, dz & \Im \oint_B h'_{\lambda_0}\, dz & \Re \mathcal{B}& \Im \mathcal{B} 
					\end{pmatrix},
				\end{equation}
				where 
				\begin{equation}
					\mathcal{A} = \oint_A 	\left(h_{\lambda_1}'(z)-h_{\lambda_0}'(z)\right) \, dz, \qquad \mathcal{B} = \oint_B	\left(h_{\lambda_1}'(z)-h_{\lambda_0}'(z)\right) \, dz.
				\end{equation}
				That is, $\mathcal{A}$ and $\mathcal{B}$ are the $A$ and $B$ periods of a holomorphic differential on $\mathfrak{R}$, and the determinant is given by
				\begin{equation}
					\det\nabla F = \Im (\overline{A}B) > 0, 
				\end{equation}
				which follows from Riemann's Bilinear inequality. As this determinant is non-zero, we can deform the endpoints continuously in $s$ so as to preserve \eqref{eq: endpoint boutroux}, verifying that for all $s\in\mathfrak{G}_1^-$, we may construct a genus $1$ $h$-function.

				For $s\in \mathfrak{G}_1^-$ we have $\Omega(s) = \gamma_{c,0}\cup\gamma_{m,0}\cup\gamma_{c,1}\cup\gamma_{m,1}$, and we define
				\begin{equation}\label{eq: genus 1 h function}
					h(z;s) = \int_1^z h'(u;s)\,du,
				\end{equation}
				where the path of integration is taken in $\C\setminus\Omega(s)$ and $h'$ is given in \eqref{eq: h prime genus 1}. We now have the following Lemma, which shows that the so-constructed $h$ function is the correct one needed for genus $1$ asymptotics.
				\begin{lemma}
					Let $s \in \mathfrak{G}_1^-$. Then, $h(z;s)$ defined in \eqref{eq: genus 1 h function} solves the Riemann-Hilbert problem \eqref{eq: RHP for h} and satisfies the inequalities \eqref{eq: both inequalities}. 
				\end{lemma}
				\begin{proof}
					Again, it is immediate that $h$ is analytic in $\C\setminus\Omega(s)$ and has the appropriate endpoint behavior near all endpoints in $\Lambda$. Moreover, from the first condition of \eqref{eq: endpoint boutroux}, we ensure that $h$ has the correct asymptotics at infinity. The Boutroux condition ensures that we have a purely imaginary jump over $\gamma_{c,1}$ and the same residue calculation as in the genus $0$ case yields that $h_+(z) - h_-(z) = 4\pi i$ for $z\in\gamma_{c,0}$. Finally, as $\Re h(z)=0$ for $z \in\mathfrak{M}$, along with $h'_+(z)+h'_-(z)=0$ for $z\in\mathfrak{M}$ and the Boutroux condition, we have that $h_++h'_-$ is purely imaginary on the main arcs $\gamma_{m,0}$ and $\gamma_{m,1}$.
					
					As before, the inequalities \eqref{eq: both inequalities} were established in \cite{celsus2020supercritical} directly for $s \in i\mathbb{R}$ with $\Im s < -t_0$, so we may again use Lemma~\ref{lem: breaking point lemma} to show that the inequalities continue to hold for all $s \in \mathfrak{G}_1^{-}$. 
				\end{proof}
			
				The case $s \in \mathfrak{G}_1^+$ may be easily obtained via reflection. To see this, note that if $s\in\mathfrak{G}_1^+$, then $-s \in \mathfrak{G}_1^-$. Take $\lambda_0(s) = -\lambda_0(-s)$ and $\lambda_1(s) = -\lambda_1(-s)$, so that $h'(z;s) = - h'(-z;-s)$, and we may use the results for $-s\in\mathfrak{G}_1^-$ to construct the appropriate genus $1$ $h$-function. 
				
	\subsection{Proof of Theorem~\ref{thm: global phase portrait}}\label{sec: construction of breaking curves}
		We recall that the aim of Theorem~\ref{thm: global phase portrait} is to verify that Figure~\ref{fig: breaking curve} is the accurate picture of the set of breaking curves in the parameter space. 
		
		As the genus of $\mathfrak{R}(s)$ is either $0$ or $1$, we have that the genus must be $0$ along a breaking curve. That is, $\Omega(s) = \gamma_{c,0} \cup \gamma_{m,0}$. We have seen in \eqref{eq: genus 0 h function explicit} that the regular genus $0$ $h$-function is given by:
		\begin{equation}\label{eq: h function for bc}
			h(z;s)=2\log\left(z+\left(z^2-1\right)^{1/2}\right) - s\left(z^2-1\right)^{1/2}. 
		\end{equation}
		\begin{remark}
			Note that there is one other genus zero $h$ function which occurs when $s\in \mathbb{R}$ and $\left|s\right|>2$. Here, we have that
			\begin{equation*}
			h'(z) = \sqrt{\frac{z-\lambda_1(s)}{z-1}}, \qquad \text{ or } \qquad h'(z) = \sqrt{\frac{z-\lambda_2(s)}{z+1}},
			\end{equation*}
			with a cut taken on the real line connecting $\lambda_1$ and $1$ or $\lambda_2$ and $-1$, depending on the situation. However, neither of these $h$-functions admit saddle points, so they do not need to be considered when looking for breaking points. 
		\end{remark}
		It is clear by looking at \eqref{eq: h function for bc} that the only saddle point is at $z_0 = 2/s$. As this is a simple zero of $h'$, we see that the only critical breaking points occur when the saddle point coincides with the branchpoints in $\Lambda(s)$. That is, the only critical breaking points are $s=\pm 2$. To study the structure of breaking curves, we will need the following calculation. 
		\begin{prop}\label{prop: partial h with respect to s}
			If $s_b$ is a regular breaking point, then
			\begin{equation}
				\frac{d}{d s} h\left(\frac{2}{s_b}, s_b\right)\not = 0. 
			\end{equation}
		\end{prop}
		\begin{proof}
			We write 
			\begin{equation}
				h\left(\frac{2}{s}, s\right) =2\log\left(\frac{2}{s}+\left(\frac{4}{s^2}-1\right)^{1/2}\right) - s\left(\frac{4}{s^2}-1\right)^{1/2}, 
			\end{equation}
			so that
			\begin{equation}
				h'\left(\frac{2}{s}, s\right)= -\left(-1+\frac{4}{s^2}\right)^{1/2}.
			\end{equation}
			Note that this vanishes only for $s = \pm 2$, which are critical breaking points, so that the proposition above is true for all regular breaking points. 
		\end{proof}
		By Lemma~\ref{lem: def of breaking curves}, the above proposition immediately implies the following, just as in \cite[Corollary~6.2]{bertola2016asymptotics}. 
		\begin{corollary}\label{prop: COR 6.2}
			Breaking curves are smooth, simple curves consisting of regular breaking points (except possibly the endpoints). They do not intersect each other except perhaps at critical breaking points $s=\pm 2$ or at infinity. They can originate and end only at critical breaking points and at infinity.
		\end{corollary} 
		Now, we can indeed verify the global phase portrait depicted in Figure~\ref{fig: breaking curve} is the correct picture, proving Theorem~\ref{thm: global phase portrait}.
	
		To find the breaking curves, we recall that the only saddle point occurs at 
		\begin{equation}
			z_0(s) = \frac{2}{s},
		\end{equation}
		so that the breaking curves are part of the zero level set
		\begin{equation}\label{eq: breaking curve linear weight}
			\Re\left(2\log\left(\frac{2}{s}+\left(\frac{4}{s^2}-1\right)^{1/2}\right) - s\left(\frac{4}{s^2}-1\right)^{1/2}\right) = 0. 
		\end{equation}
		Recall also, that the only critical breaking points are $s=\pm 2$, at which the saddle point collides with the hard edge at $\pm 1$, respectively. As $h(2/s, s)=\mathcal{O}\left(\left(s-2\right)^{3/2}\right)$ as $s\to\pm 2$, we note that 3 breaking curves emanate from each of $\pm 2$. 
		
		Now, if $s\in \mathbb{R}$ and $\left|s\right| > 2$, then
		\begin{equation*}
			-s\left(\frac{4}{s^2}-1\right)^{1/2} \in i \mathbb{R},
		\end{equation*}
		where we have taken the branch cut to be the interval $[-1,1]$. Furthermore, recall that the map $z \to z+\left(z^2-1\right)^{1/2}$ sends the interval $(-1,1)$ to the unit circle. As such, we also have that 
		\begin{equation*}
			2\log\left(\frac{2}{s}+\left(\frac{4}{s^2}-1\right)^{1/2}\right) \in i\mathbb{R}
		\end{equation*}
		when $s \in \mathbb{R}$ and $\left|s\right|>2$. Therefore, the rays $(2, \infty)$ and $(-\infty,-2)$ are both breaking curves. Finally, note that 
		\begin{equation}
			h\left(\frac{2}{s},s\right) = -is+i\pi+ \mathcal{O}\left(\frac{1}{s}\right), \qquad s\to\infty,
		\end{equation}
		so that the two rays emanating from $\pm 2$ towards infinity along the real axis are the only two portions of the breaking curve which intersect at infinity.
		
		According to Corollary~\ref{prop: COR 6.2}, the remaining breaking curves either emanate from $\pm 2$ or form closed loops in the $s$-plane consisting of only regular breaking points. As $h(2/s;s)$ has non-zero real part for $s\in(-2,2)$, we conclude that the remaining breaking curves do not intersect the real axis. Next, note that $\Re h(2/s;s)$ is harmonic for $s$ off the real axis, so that off the real axis, there are no closed loops along which $\Re h(2/s,s)=0$. Therefore, the remaining breaking curves begin and end at $\pm 2$. Finally, as 
		\begin{equation}
			h\left(\frac{2}{\overline{s}}, \overline{s}\right) = \overline{h\left(\frac{2}{s},s\right)}, 
		\end{equation}
		we see that the breaking curves which connect $-2$ and $2$ are symmetric about the real axis. 
	
	\subsection{Proof of Theorem~\ref{thm: recurrence coeffs genus 0}}\label{sec: asymptotics I}
		Having successfully verified the global phase portrait is as depicted in Figure~\ref{fig: breaking curve}, with $\mathfrak{G}_0$ corresponding to the genus $0$ region and $\mathfrak{G}_1^\pm$ corresponding to the genus $1$ regions, we may now use the techniques illustrated in Section~\ref{sec: unwinding transormations} to obtain asymptotics of the recurrence coefficients for $s \in \C\setminus\mathfrak{B}$. 
		
		For $s \in \mathfrak{G}_0$, we are in the genus $0$ region and as such we will use the global  parametrix defined in Subsection~\ref{sub: Global parametrix genus 0}. We recall that the global parametrix given in \eqref{eq: global parametrix eq genus 0} satisfies as $z \to \infty$
		\begin{equation}
			\label{eq: M1 and M2}
			M(z) = I + \frac{M^{(1)}}{z} + \frac{M^{(2)}}{z^2} + \mathcal{O}\left(\frac{1}{z^3}\right), \qquad M^{(1)} = \begin{pmatrix}
				0 & \frac{i}{2} \\
				-\frac{i}{2} & 0
			\end{pmatrix}, \qquad M^{(2)} = \begin{pmatrix}
				\frac{1}{8} & 0\\
				0& \frac{1}{8} 
			\end{pmatrix}.
		\end{equation}
		
		Recall from \eqref{eq: recurrence coeffs in terms of T} that $\alpha_n, \beta_n$ may be written in terms of the matrices $T^{(1)}, T^{(2)}$ appearing in the asymptotic expansion of $T(z)$ as $z \to \infty$.
		In Section~\ref{sec: Small Norm} we stated that $R$ has an asymptotic expansion of the form
		\begin{equation}
			R(z) = I +\sum_{k=1}^\infty \frac{R_k(z)}{n^k}, \qquad n\to\infty, 
		\end{equation}
		which is valid uniformly in the variable $z$ near infinity, and each $R_k(z)$ for $k\geq 1$ satisfies
		\begin{equation}
			R_k(z) = \frac{R_k^{(1)}}{z} + \frac{R_k^{(2)}}{z^2} +\mathcal{O}\left(\frac{1}{z^3}\right), \qquad z\to\infty. 
		\end{equation}
		Recalling that $T(z)=R(z)M(z)$ outside of the lens, we may write
		\begin{subequations}
			\label{eq: T in terms of M, R}
			\begin{equation}
			T^{(1)} = M^{(1)} + \frac{R_1^{(1)}}{n} + \frac{R_2^{(1)}}{n^2} + \mathcal{O}\left(\frac{1}{n^3}\right), \qquad n\to\infty,
			\end{equation}
			and 
			\begin{equation}
			T^{(2)} = M^{(2)} + \frac{R_1^{(1)}M^{(1)} + R_1^{(2)}}{n} + \frac{R_2^{(1)}M^{(1)} + R_2^{(2)}}{n^2} + \mathcal{O}\left(\frac{1}{n^3}\right), \qquad n\to\infty,
			\end{equation}
		\end{subequations}
		and as such we turn our attention to determining $R_1$ and $R_2$.
		
		We recall the discussion in Section~\ref{sec: Small Norm}, where we wrote $j_R(z) = I + \Delta(z)$, where $\Delta$ admits an asymptotic expansion in inverse powers of $n$ as
		\begin{equation}
		\Delta(z) \sim \sum_{k=1}^\infty \frac{\Delta_k(z)}{n^k}, \qquad n\to\infty. 
		\end{equation}
		As $\Delta(z)$ decays exponentially quickly for $z \in\Sigma_R\setminus\cup_{\lambda\in\Lambda}\partial D_\lambda$, we have that
		\begin{equation}
		\Delta_k(z) = 0, \qquad z\in\Sigma_R\setminus\bigcup_{\lambda\in\Lambda}\partial 	D_\lambda.
		\end{equation}
		The behavior of $\Delta_k(z)$ for $z\in \partial D_\lambda$ can be determined in terms of the appropriate local parametrix used at the particular $\lambda \in \Lambda$. 
		
		We give an explicit formula for $\Delta_k(z)$ for $z\in\partial D_1$ following \cite[Section~8]{arnojacobi}. We compute that the Bessel parametrix defined in \eqref{eq: Bessel parametrix} satisfies
		\begin{equation}\label{eq: Bessel parametrix expansion}
		B\left(\zeta\right) = \frac{1}{\sqrt{2}}(2\pi)^{-\sigma_3/2} \zeta^{-\sigma_3/4}\begin{pmatrix}
		1 & i \\
		i & 1
		\end{pmatrix}\left(I + \sum_{k=1}^\infty \frac{B_k}{\zeta^{k/2}}\right)e^{2\zeta^{1/2}\sigma_3}
		\end{equation}
		uniformly as $\zeta \to\infty$, where the matrices $B_k$ are defined as
		\begin{equation}
			B_k := \frac{(-1)^{k-1}\prod_{j=1}^{k-1}(2j-1)^2}{4^{2k-1}(k-1)!}\begin{pmatrix}
			\frac{(-1)^k}{k} \left(\frac{k}{2}-\frac{1}{4}\right) & -i \left(k-\frac{1}{2}\right) \\
			(-1)^k i \left(k-\frac{1}{2}\right) &\frac{1}{k}\left(\frac{k}{2}-\frac{1}{4}\right)
			\end{pmatrix}
		\end{equation}
		As $\Delta(z) = P^{(1)}(z)M^{-1}(z)-I$ for $z\in\partial D_1$, we may use\eqref{eq: matching condition hard edge}-\eqref{eq: Bessel conformal map} to see that
		\begin{align}
			\Delta(z) &= P^{(1)}(z)M^{-1}(z) -I=M(z)\left[\sum_{k=1}^\infty \frac{4^k B_k}{n^k h(z)^k}\right]M^{-1}(z), \qquad n\to\infty,
		\end{align}
		so that we have by direct inspection, 
		\begin{equation}\label{eq: Delta k 1}
			\Delta_k(z)= \frac{(-1)^{k-1}\prod_{j=1}^{k-1}(2j-1)^2}{4^{k-1}(k-1)!h(z)^k} M(z)\begin{pmatrix}
			\frac{(-1)^k}{k} \left(\frac{k}{2}-\frac{1}{4}\right) & -i \left(k-\frac{1}{2}\right) \\
			(-1)^k i \left(k-\frac{1}{2}\right) &\frac{1}{k}\left(\frac{k}{2}-\frac{1}{4}\right)
			\end{pmatrix}M^{-1}(z),
		\end{equation}
		for $z\in\partial D_1$. Defining $\tilde{h}(z) = h(z) - 2\pi i$, we are able to similarly compute that
		\begin{equation} \label{eq: Delta k -1}
			\Delta_k(z)= \frac{(-1)^{k-1}\prod_{j=1}^{k-1}(2j-1)^2}{4^{k-1}(k-1)!\tilde{h}(z)^k} M(z)\begin{pmatrix}
			\frac{(-1)^k}{k} \left(\frac{k}{2}-\frac{1}{4}\right) & i \left(k-\frac{1}{2}\right) \\
			(-1)^{k+1} i \left(k-\frac{1}{2}\right) &\frac{1}{k}\left(\frac{k}{2}-\frac{1}{4}\right)
			\end{pmatrix}M^{-1}(z),
		\end{equation}
		when $z \in \partial D_{-1}$. It was also shown in \cite[Section~8]{arnojacobi} that we may write that
		\begin{align}\label{eq: A1 and B1 def}
			\Delta_1(z) = 
			\begin{cases}
			\displaystyle
			\frac{A^{(1)}}{z-1} + \mathcal{O}\left(1\right), \qquad & z\to 1, \\[2mm]
			\displaystyle
			\frac{B^{(1)}}{z+1}  + \mathcal{O}\left(1\right), \qquad & z\to -1,
			\end{cases}
		\end{align}
		for some constant matrices $A^{(1)}$ and $B^{(1)}$. Using the behavior of $h$ defined in \eqref{eq: genus 0 h function explicit} and $\varphi$ near $\pm 1$, we find that
		\begin{equation}\label{eq: A1 and B1}
			A^{(1)} = \frac{1}{8(s-2)}\begin{pmatrix}
			-1 & i \\
			i & 1
			\end{pmatrix}, \qquad B^{(1)} = \frac{1}{8(s+2)}\begin{pmatrix}
			-1 &- i \\
			-i & 1
			\end{pmatrix}.
		\end{equation}
		We recall from Section~\ref{sec: Small Norm} that the $\Delta_k$ may be used to solve for the $R_k$ via the following Riemann-Hilbert problem. 
		\begin{subequations}
			\label{eq: R_k RHP genus 0}
			\begin{alignat}{2}
			&R_k(z) \text{ is analytic for } z \in \C\setminus\left(\partial D_1 \cup \partial D_{-1}\right) \qquad &&\\
			& R_{k,+}(z) = R_{k,-}(z) + \sum_{j=1}^{k-1}R_{k-j,-}\Delta_j(z), \qquad &&z\in \partial D_1 \cup \partial D_{-1}\\
			& R_k(z) = \mathcal{O}\left(\frac{1}{z}\right), \qquad &&z \to\infty.
			\end{alignat}
		\end{subequations}
		Having determined the $\Delta_k(z)$ for $z \in \partial D_{\pm 1}$, we may solve for the $R_k$ directly. By inspection, we see that
		\begin{align}\label{eq: R_1 explicit genus 0}
		R_1(z) =\begin{cases}
		\displaystyle
		\frac{A^{(1)}}{z-1}+ \frac{B^{(1)}}{z+1}, \qquad & z\in \C\setminus\left(D_1 \cup D_{-1}\right), \\[2mm]
		\displaystyle
		\frac{A^{(1)}}{z-1}+ \frac{B^{(1)}}{z+1} - \Delta_1(z), \qquad &z\in D_{1}\cup D_{-1}, 
		\end{cases}
		\end{align}
		solves the Riemann-Hilbert problem \eqref{eq: R_k RHP genus 0} for $R_1$. 
		
		To determine $R_2$, we again follow \cite{arnojacobi} where it was shown 
		\begin{align}
			R_1(z)\Delta_1(z)+\Delta_2(z) = \begin{cases}
			\displaystyle
			\frac{A^{(2)}}{z-1} + \mathcal{O}\left(1\right), \qquad & z\to 1, \\[2mm]
			\displaystyle
			\frac{B^{(2)}}{z+1}  + \mathcal{O}\left(1\right), \qquad & z\to -1,
			\end{cases}
		\end{align}
		for some constant matrices $A^{(2)}$ and $B^{(2)}$. As we now have explicit formula for $R_1$, $\Delta_1$, and $\Delta_2$, we may use the properties of $h$ and $\varphi$ to determine that 		
		\begin{subequations}
			\begin{equation}
				A^{(2)} = 	\frac{1}{16(s-2)^2(s+2)} \begin{pmatrix}
				\frac{s-2}{4} & i(2s+5) \\
				-i(2s+5) & \frac{s-2}{4} 
				\end{pmatrix} 
			\end{equation}
			and
			\begin{equation}
				B^{(2)} = \frac{1}{16(s-2)(s+2)^2}\begin{pmatrix}
				-\frac{s+2}{4} &  i(2s-5)\\
				-i(2s-5) & -\frac{s+2}{4}
				\end{pmatrix}.
			\end{equation}
		\end{subequations}
		Having determined the $A^{(2)}$ and $B^{(2)}$, we may again solve the Riemann-Hilbert problem for $R_2$ by inspection as 
		\begin{align}\label{eq: R_2 explicit genus 0}
			R_2(z) =\begin{cases}
			\displaystyle
				\frac{A^{(2)}}{z-1}+ \frac{B^{(2)}}{z+1}, \qquad & z\in \C\setminus\left(D_1 \cup D_{-1}\right), \\[2mm]
			\displaystyle
				\frac{A^{(2)}}{z-1}+ \frac{B^{(2)}}{z+1} - R_1(z)\Delta_1(z)-\Delta_2(z), \qquad &z\in D_{1}\cup D_{-1}.
			\end{cases}
		\end{align}
		
		Now, we may expand the $R_k$ at infinity to determine the appropriate terms in \eqref{eq: T in terms of M, R}. As $R_k(z) = A^{(k)}/(z-1)+B^{(k)}/(z+1)$ for $k =1, 2$ and $z\in \C\setminus\left(D_1\cup D_{-1}\right)$, we have that
		\begin{equation}
			R_k(z) = \frac{A^{(k)}+B^{(k)}}{z} + \frac{A^{(k)}-B^{(k)}}{z^2} + \mathcal{O}\left(\frac{1}{z^3}\right), \qquad z\to\infty. 
		\end{equation}
		Using the explicit formula for the $A^{(k)}$ and $B^{(k)}$, we determine that
		\begin{subequations}
			\label{eq: Rij eqn genus 0}
			\begin{equation}
				R_1^{(1)} = \frac{1}{4\left(4-s^2\right)}\begin{pmatrix}
				s & -2i \\
				-2i& -s
				\end{pmatrix}, \qquad
				R_1^{(2)} = \frac{1}{4\left(4-s^2\right)}\begin{pmatrix}
				2 & -is \\
				-is& -2
				\end{pmatrix}
			\end{equation}
			\begin{equation}
				R_2^{(1)} = \frac{i(s^2+5)}{4\left(s^2-4\right)^2}\begin{pmatrix}
				0& 1\\
				-1 &0
				\end{pmatrix}, \qquad
				R_2^{(2)} = \frac{1}{32\left(s^2-4\right)^2}\begin{pmatrix}
				s^2-4 & 36is\\
				-36is & s^2-4
				\end{pmatrix}
			\end{equation}
		\end{subequations}
		Finally, using \eqref{eq: M1 and M2} and \eqref{eq: Rij eqn genus 0} in \eqref{eq: recurrence coeffs in terms of T} and \eqref{eq: T in terms of M, R}, we see that as $n \to \infty$
			\begin{equation}
				\alpha_n(s) = \frac{2s}{(s^2-4)^2}\frac{1}{n^2} + \mathcal{O}\left(\frac{1}{n^3}\right), \qquad \beta_n(s) = \frac{1}{4} + \frac{s^2+4}{4(s^2-4)^2}\frac{1}{n^2}  + \mathcal{O}\left(\frac{1}{n^4}\right) 
			\end{equation}
		completing the proof of Theorem~\ref{thm: recurrence coeffs genus 0}.

	\subsection{Proof of Theorem~\ref{thm: recurrence coeffs genus 1}}\label{sec: asymptotics II}
		For $s\in \mathfrak{G}_1^\pm$, the $h$-function is of genus $1$, and we must use the global parametrix constructed in Section~\ref{sub: Genus 1 Global Parametrix}. Throughout this proof, we recall that we are working with the assumption that $n\in\N(s,\epsilon)$, so that the global parametrix exists by Lemma~\ref{lem: model solution genus 1}. 
		We see with a similar calculation as the one that leads to \eqref{eq: T in terms of M, R} that
		$T^{(k)} = M^{(k)} + \mathcal{O}\left(\frac{1}{n}\right)$ as $n\to\infty$ for $k = 1, 2$, so we have that, using \eqref{eq: recurrence coeffs in terms of T}
		\begin{equation}\label{eq: recurrence in terms of M}
		\alpha_n = \frac{M^{(2)}_{12}}{M^{(1)}_{12}} - M^{(1)}_{22}+\mathcal{O}\left(\frac{1}{n}\right), \qquad 
		\beta_n = M^{(1)}_{12}M^{(1)}_{21}+\mathcal{O}\left(\frac{1}{n}\right), \quad \text{ as } \quad n \to \infty.
		\end{equation}
		\begin{remark}
			In order to compute higher order terms in the expansion of the recurrence coefficients in the genus $1$ regime, one would again need to write the jump matrix for $R$ as a perturbation of the identity. This would involve writing the jump matrix on $\partial D_\lambda$ in terms of the appropriate local parametrix used at $\lambda$. 
			One could again carry out the process detailed in Section~\ref{sec: asymptotics I} to obtain higher order terms in the genus $1$ regime, but we just concern ourselves with the leading term. 
		\end{remark}
		By Lemma~\ref{lem: model solution genus 1}, as $n\in \N(s, \epsilon)$, the global parametrix is defined as 
		\begin{equation}
			M(z) = e^{n\tilde{g}(\infty)\sigma_3} \mathcal{L}^{-1}(\infty) \mathcal{L}(z) e^{-n\tilde{g}(z)\sigma_3},
		\end{equation}
		where we recall from \eqref{eq: tilde g eqn} and \eqref{eq: L cal defn} that
		\begin{equation}
			\mathcal{L}(z) := \frac{1}{2}\begin{pmatrix}
			\left(\phi(z)+\phi(z)^{-1}\right)\mathcal{M}_1(z,d) & i\left(\phi(z)-\phi(z)^{-1}\right) \mathcal{M}_2(z,d) \\
			-i\left(\phi(z)-\phi(z)^{-1}\right)\mathcal{M}_1(z,-d) & \left(\phi(z)+\phi(z)^{-1}\right) \mathcal{M}_2(z,-d) \\
			\end{pmatrix}
		\end{equation}
		and
		\begin{equation}
				\tilde{g}(z) = \varXi(z)\left[\int_{\gamma_{c,1}} \frac{ \eta_1\, d\zeta}{(\zeta-z)\varXi(\zeta)}-\int_{\gamma_{m,0}}\frac{\Delta_0\, d\zeta}{(\zeta-z)\varXi_+(\zeta)}\right].
		\end{equation}
		Above, $\varXi(z)$ is given by \eqref{varXi} and $\phi$ is defined in \eqref{eq: beta eqn} as
		\begin{equation}
			\phi(z) = \left(\frac{\left(z+1\right)\left(z-\lambda_1\right)}{\left(z-\lambda_0\right)\left(z-1\right)}\right)^{1/4}
			\end{equation}
		with branch cuts on $\gamma_{m,0}$ and $\gamma_{m,1}$ and the branch of the root chosen so that $\phi(\infty)=1$ and the constant $\Delta_0$ was chosen to satisfy
		\begin{equation}
			\int_{\gamma_{c,1}} \frac{ \eta_1\, d\zeta}{\varXi(\zeta)}- \int_{\gamma_{m,0}}\frac{\Delta_0\, d\zeta}{\varXi_+(\zeta)} = 0. 
		\end{equation}
		We see that
		\begin{equation}
		 	\tilde{g}(z) = \tilde{g}(\infty) + \frac{\tilde{g}_1}{z} + \frac{\tilde{g}_2}{z^2}	+ \mathcal{O}\left(\frac{1}{z^3}\right), \qquad z\to\infty, 
		\end{equation}
		where
		\begin{subequations}
			\begin{equation}
				\tilde{g}(\infty) = \delta_1, \qquad \tilde{g}_1= \delta_2 -\frac{\delta_1\left(\lambda_0+\lambda_1\right)}{2}, \qquad \tilde{g}_2 = \delta_3 -\frac{\delta_2\left(\lambda_0+\lambda_1\right)}{2}-\frac{\delta_1\left(4+(\lambda_0-\lambda_1)^2\right)}{8},
			\end{equation}
			and
			\begin{equation}
				\delta_k :=\int_{\gamma_{m,0}}\frac{\zeta^k\Delta_0\, d\zeta}{\varXi_+(\zeta)}- \int_{\gamma_{c,1}} \frac{\zeta^k \eta_1\, d\zeta}{\varXi(\zeta)}.
			\end{equation}
		\end{subequations}
		Therefore, 
		\begin{equation}
			e^{-n\tilde{g}(z)\sigma_3} = \left[I - \frac{n\tilde{g}_1\sigma_3}{z}+\frac{n^2\tilde{g}_1^2I-2n\tilde{g}_2\sigma_3}{2z^2}+\mathcal{O}\left(\frac{1}{z^3}\right)\right]e^{-n\tilde{g}(\infty)\sigma_3}, \qquad z\to\infty. 
		\end{equation}
		Next we turn to the expansion of the matrix $\mathcal{L}$. We have
		\begin{equation}
			\mathcal{L}(z) = \mathcal{L}(\infty)+ \frac{\mathcal{L}_1}{z}+\frac{\mathcal{L}_2}{z^2}+\mathcal{O}\left(\frac{1}{z^3}\right), \qquad z\to\infty.
		\end{equation}
		To calculate $\mathcal{L}_1$ and $\mathcal{L}_2$, we first see that by \eqref{eq: beta eqn} that 
		\begin{equation}
			\phi(z) = 1+ \frac{\phi_1}{z} + \frac{\phi_2}{z^2} + \mathcal{O}\left(\frac{1}{z^3}\right), \qquad z\to\infty, 
		\end{equation}
		where
			\begin{equation}	\label{eq: b1 and b2}
				\phi_1 = \frac{2+\lambda_0-\lambda_1}{4}, \qquad \phi_2 = \frac{4+4\lambda_0+5\lambda_0^2-4\lambda_1-2\lambda_0\lambda_1-3\lambda_1^2}{32}.
			\end{equation}
		This then gives us that
		\begin{subequations}
			\begin{equation}
				\phi(z)+\phi(z)^{-1} = 2 +\frac{\phi_1^2}{z^2} + \mathcal{O}\left(\frac{1}{z^3}\right), \qquad z\to\infty, 
			\end{equation}
			and
			\begin{equation}
				\phi(z)-\phi(z)^{-1}= \frac{2\phi_1}{z} +\frac{2\phi_2-\phi_1^2}{z^2} + \mathcal{O}\left(\frac{1}{z^3}\right), \qquad z\to\infty, 
			\end{equation}
		\end{subequations}
		which implies
		\begin{subequations}
			\begin{equation}
				\mathcal{L}_1= \begin{pmatrix}
					\frac{d}{dz}\mathcal{M}_1\left(\frac{1}{z},d\right)\Big|_{z=0} & i\phi_1 \mathcal{M}_2(\infty, d)\\
					-i\phi_1 \mathcal{M}_1(\infty,- d) & \frac{d}{dz}\mathcal{M}_2\left(\frac{1}{z},-d\right)\Big|_{z=0}
				\end{pmatrix}
			\end{equation}
			and
			\begin{equation}
				\mathcal{L}_2=\begin{pmatrix}
					\frac{1}{2}\mathcal{M}_1\left(\infty,d\right)\phi_1^2+\frac{d^2}{dz^2}\mathcal{M}_1\left(\frac{1}{z},d\right)\Big|_{z=0} & \frac{\phi_1^2-2\phi_2}{2i}\mathcal{M}_2\left(\infty,d\right)+i\phi_1\frac{d}{dz}\mathcal{M}_2\left(\frac{1}{z},d\right)\Big|_{z=0} \\
					\frac{2\phi_2-\phi_1^2}{2i}\mathcal{M}_1\left(\infty,-d\right)-i\phi_1\frac{d}{dz}\mathcal{M}_1\left(\frac{1}{z},-d\right)\Big|_{z=0} & \frac{1}{2}\mathcal{M}_2\left(\infty,-d\right)\phi_1^2+\frac{d^2}{dz^2}\mathcal{M}_2\left(\frac{1}{z},-d\right)\Big|_{z=0}
				\end{pmatrix}
			\end{equation}
		\end{subequations}
		Putting this all together yields
		\begin{subequations}
			\begin{equation}
				M_1 = e^{n\tilde{g}(\infty)\sigma_3}\left[\mathcal{L}^{-1}(\infty)\mathcal{L}_1-n\tilde{g}_1\sigma_3\right]e^{-n\tilde{g}(\infty)\sigma_3}
			\end{equation}
			and
			\begin{equation}
				M_2 = e^{n\tilde{g}(\infty)\sigma_3}\left[\frac{n^2\tilde{g}_1^2\sigma_3^2-2n\tilde{g}_2\sigma_3}{2} -n\tilde{g}_1 \mathcal{L}^{-1}(\infty)\mathcal{L}_1\sigma_3+ \mathcal{L}^{-1}(\infty)\mathcal{L}_2\right]e^{-n\tilde{g}(\infty)\sigma_3}
			\end{equation}
		\end{subequations}
		Using this in \eqref{eq: recurrence in terms of M}, we find that
		\begin{equation}
			\beta_n = \frac{\mathcal{M}_1(\infty, -d)\mathcal{M}_2(\infty, d)}{\mathcal{M}_1(\infty, d)\mathcal{M}_2(\infty, -d)} \phi_1^2 + \mathcal{O}\left(\frac{1}{n}\right), \qquad n\to\infty,
		\end{equation}
		and
		\begin{equation}
			\alpha_n = \frac{\phi_1}{2}-\frac{\phi_2}{\phi_1} + \frac{d}{dz}\left[\log\mathcal{M}_2(1/z,d)-\log\mathcal{M}_2(1/z,-d)\right]\Big|_{z=0}+\mathcal{O}\left(\frac{1}{n}\right).
		\end{equation}
		Using \eqref{eq: b1 and b2}, we arrive at
		\begin{subequations}
			\begin{equation}
				\alpha_n(s) = \frac{\lambda_1^2(s)-\lambda_0^2(s)}{4+2\lambda_0(s)-2\lambda_1(s)} + \frac{d}{dz}\left[\log\mathcal{M}_2(1/z,d)-\log\mathcal{M}_2(1/z,-d)\right]\Big|_{z=0}+\mathcal{O}\left(\frac{1}{n}\right)
			\end{equation}
			and
			\begin{equation}
				\beta_n(s) = \frac{(2+\lambda_0(s)-\lambda_1(s))^2}{16}\frac{\mathcal{M}_1(\infty, -d)\mathcal{M}_2(\infty, d)}{\mathcal{M}_1(\infty, d)\mathcal{M}_2(\infty, -d)}  + \mathcal{O}\left(\frac{1}{n}\right),
			\end{equation}
		\end{subequations}
		as $n\to\infty$, completing the proof of Theorem~\ref{thm: recurrence coeffs genus 1}.
		
\section{Double Scaling Limit near Regular Breaking Points}\label{sec: DS Reg}
Having determined the behavior of the recurrence coefficients as $n\to\infty$ with $s\in \mathfrak{G}_0\cup\mathfrak{G}_1^\pm$, we turn our attention to the behavior of these coefficients for critical values of $s_* \in\mathfrak{B}$ where $s_*\not\in\R$. Below, the double scaling limit describes the asymptotics of the recurrence coefficients as both $n\to\infty$ and $s\to s_*$ simultaneously at an appropriate scaling rate.

\subsection{Definition of the Double Scaling Limit}
In the remainder of this section, we will assume that $s$ approaches $s_*$ within the region $\mathfrak{G}_0$. In particular, we fix $s_*\in\mathcal{B}\setminus\left((-\infty,-2]\cup[2,\infty)\right)$ and take
\begin{equation}\label{eq: def of double scaling}
s = s_* + \frac{L_1}{n}, \qquad L_1\in\C,
\end{equation}
where the constant $L_1$ is chosen so that $s\in\mathfrak{G}_0$ for all $n$ large enough. Furthermore, we impose that $\Im s_* < 0$, so that $\Im \frac{2}{s_*}>0$; this requirement is for ease of exposition, and the case where $\Im s_* > 0$ can be handled similarly. As $s\to s_*$ within $\mathfrak{G}_0$, we have that $\Omega(s)=\gamma_{c,0}\cup\gamma_{m,0}(s)$. Furthermore, there exists a genus $0$ $h$-function which satisfies \eqref{eq: RHP for h} with $L = 0$.
As $s_*$ is a regular breaking point, we now have that $\Re(h(2/s_*;s_*))=0$, by definition, and a more detailed local analysis will be needed in the vicinity of this point. 

As the first transformation is the same as the first transformation in Section~\ref{sec: steepest descent}, we briefly restate it below. We recall that $Y$ defined in \eqref{eq: Y eqn} solves the Riemann-Hilbert problem \eqref{eq: Y RHP}. By setting
\begin{equation}\label{eq: T def ds}
T(z) := e^{-nl \sigma_3/2} Y(z) e^{-\frac{n}{2}\left[h(z)+f(z)\right]\sigma_3},
\end{equation}
we then have that $T$ defined above solves the Riemann-Hilbert problem \eqref{eq: T RHP}. 

\subsection{Opening of the Lenses}
In order to address some of the more technical issues which arise when attempting to open lenses, we turn again to the theory of quadratic differentials. Recall that $\gamma_{m,0}(s)$ is defined to be the trajectory of the quadratic differential 
\begin{equation}
\varpi_s = -\frac{(2-sz)^2}{z^2-1}\, dz^2
\end{equation}
which connects $-1$ and $1$, whose existence is assured due to Lemma~\ref{lem: QD in genus 0}. Moreover, we also have that four trajectories $\varpi_s$ emanate from $z=2/s$ at equal angles of $\pi/2$, as described in Section~\ref{sec: Quadratic Differentials} above. Finally, an application of Teichm\"uller's Lemma (c.f. \cite[Theorem~14.1]{strebel1984quadratic}) shows that the trajectories define two infinite sectors and one finite sector whose boundary is formed by a closed trajectory from $z=2/s$ which encircles both $\pm 1$. Moreover, at the critical value $s_*$, we have that two trajectories go to infinity from $z=2/s_*$, and the other two connect $z=2/s_*$ with $\pm 1$. Another application of Teichm\"uller's Lemma shows that the two infinite trajectories tend to infinity in opposite directions. The depictions of these critical graphs are given in Figure~\ref{fig: critical graphs}; for more details on the precise structure of the the critical graph we refer the reader to \cite[Section~3.2]{deano2014large}.

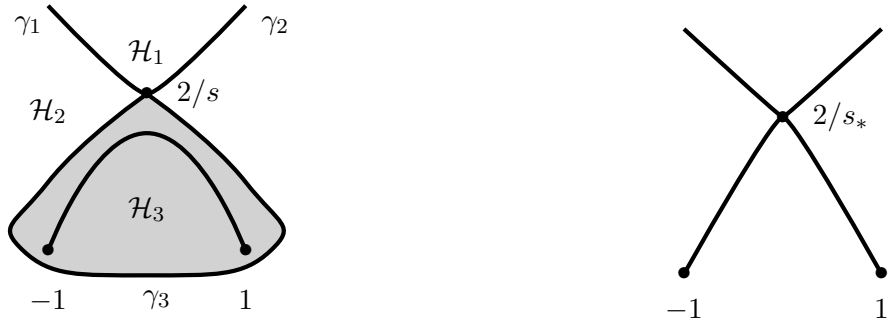
\begin{figure}[h!]
	\centering
	\begin{subfigure}{.45\textwidth}
		\centering
		\begin{tikzpicture}[scale=1.3]
		\draw [ultra thick] plot [smooth, tension =0.3] coordinates {(-1,2.5) (0,1.6) (1,2.5) };
		\draw [ultra thick, fill=gray!35] plot [smooth, tension =1] coordinates{(0,1.6)(-1,0.7)(-1.25, 0) (0,-0.25)(1.25,0)(1,0.7) (0,1.6)};
		\draw [ultra thick] plot [smooth, tension=1] coordinates {(-1,0) (0,1.2) (1,0)};
		\foreach \Point in {(-1,0), (1,0), (0,1.6)}{
			\node at \Point {\textbullet};
		}
		\node [below] at (-1, -0.3) {$-1$};
		\node [below] at (1, -0.3) {$1$};
		\node [right] at (0.2,1.6) {$2/s$};
		\node [above] at (0,1.8) {$\mathcal{H}_1$};
		\node [above] at (-1,1.2) {$\mathcal{H}_2$};
		\node [above] at (0,0.2) {$\mathcal{H}_3$};
		\node [above] at (-1.2,2.1) {$\gamma_1$};
		\node [above] at (1.3,2.1) {$\gamma_2$};
		\node [below] at (0.1,-0.3) {$\gamma_3$};
		\end{tikzpicture}
		\caption{The critical graph of $\varpi_s$ when $s=-it \in \mathfrak{G}_0$ with $t>0$. The figure depicts the situation when $s$ is close to $s_*$. The shaded region is $\mathcal{H}_3$.}
		\label{fig:sub1}
	\end{subfigure}%
	\qquad
	\begin{subfigure}{.45\textwidth}
		\centering
		\begin{tikzpicture}[scale=1.3]
		\draw [ultra thick] plot [smooth, tension =0.05] coordinates {(-1,2.5) (0,1.6) (1,2.5) };
		\draw [ultra thick] plot [smooth, tension =0.2] coordinates {(-1,0) (0,1.6) (1,0) };
		\foreach \Point in {(-1,0), (1,0), (0,1.6)}{
			\node at \Point {\textbullet};
		}
		\node [below] at (-1, -0.15) {$-1$};
		\node [below] at (1, -0.15) {$1$};
		\node [right] at (0.2,1.6) {$2/s_*$};
		\end{tikzpicture}
		\caption{The critical graph of $\varpi_s$ when $s=s_*$ where $s_*\in\mathfrak{B}~\cap~ i\R_-$.}
		\label{fig:sub2}
	\end{subfigure}
	\caption{The critical graphs of $\varpi_s$ for $s$ close to $s_*$ and for $s=s_*$}
	\label{fig: critical graphs}
\end{figure}
Recall that the key to the opening of lenses is that the jump matrices decay exponentially quickly to the identity along the lips of the lens. In the sections above this immediately followed from the inequality \eqref{eq: main arc inequality} which stated that sign of the real part of $h$ was greater than zero. However, at the critical value of $s_*$, this will no longer be true above the critical point $2/s_*$, and a more detailed local analysis will be needed. We label the trajectories emanating from $z=2/s$ as $\gamma_i$, $i = 1,2, 3$, and the regions bounded by these trajectories as $\mathcal{H}_j$, $j=1, 2, 3$, as in Figure~\ref{fig: critical graphs}.

To understand the sign of the real part of $h$, consider the function
\begin{equation}
\Upsilon(z;s) = \int_{2/s}^{z} \frac{2-s u}{\left(u^2-1\right)^{1/2}}\,du,
\end{equation}
with the branch cut taken on $\gamma_{m,0}(s)$ and branch chosen so that $\Upsilon(z;s)=-sz+\mathcal{O}\left(1\right)$ as $z\to\infty$. In terms of the $h$-function, we may write
\begin{equation}\label{eq: h in terms of upsilon}
h(z;s) = h(2/s;s) + \Upsilon(z;s).
\end{equation}
We may now state the following lemma.
\begin{lemma}
	Fix $s\in\mathfrak{G}_0$ so that $\Im s < 0$. Then, 
	\begin{equation}\tag{{i}}\label{eq: lemma i}
	\Re h\left(\frac{2}{s}; s\right) > 0, 
	\end{equation}
	\begin{equation}\tag{{ii}} \label{eq: lemma ii}
	\Re h(z;s) > 0, \qquad z \in \mathcal{H}_2 \cup \mathcal{H}_3. 
	\end{equation}
\end{lemma}
\begin{proof}
	By the basic theory (c.f \cite[Appendix~B]{martinez2016critical}, \cite[Chapter~3]{jenkins2013univalent}) the domains $\mathcal{H}_1$ and $\mathcal{H}_2$ are half plane domains which are conformally mapped by $\Upsilon$ to either the left or right half planes. As $\Im s < 0$, there exists some $t_0>0$ so that $z=-i t \in \mathcal{H}_2$ for all $t >t_0$. Recalling that 
	\begin{equation*}
	\Upsilon(z;s) = -s z + \mathcal{O}(1), \qquad z\to\infty, 
	\end{equation*}
	we may use that $\Im s < 0$ to conclude that $\Re \Upsilon(z;s)> 0$ for $z = -i t$, where $t>t_0$. Therefore, we must have that $\Upsilon$ conformally maps $\mathcal{H}_2$ to the right half plane and as such
	\begin{equation}\label{eq: upsilon sign in h2}
	\Re \Upsilon(z;s) > 0, \qquad z\in \mathcal{H}_2. 
	\end{equation}
	Similarly, as $\Upsilon$ is analytic around $z=2/s$ and has a double zero at $z=2/s$, we can conclude that $\Re \Upsilon(z;s)<0$ for $z$ in $\mathcal{H}_1\cup\mathcal{H}_3$ in close proximity to $z=2/s$. As $\mathcal{H}_1$ is a half plane domain, we immediately have that 
	\begin{equation}
	\Re \Upsilon(z;s) < 0, \qquad z\in \mathcal{H}_1. 
	\end{equation}
	Again following the theory laid out in \cite[Appendix~B]{martinez2016critical}, it follows that $\mathcal{H}_3$ is a ring domain. Therefore there exists some $c>0$ so that the function $z\mapsto \exp\left(c\Upsilon(z;s)\right)$ maps $\mathcal{H}_1$ conformally to an annulus
	\begin{equation}
	R = \left\{w\in\C: r_1 < \left|w\right| < 1\right\}.
	\end{equation}
	In particular we have that 
	\begin{equation}\label{eq: upsilon in h3}
	0 > \Re\Upsilon(z;s) > \Re \Upsilon(1,s), \qquad z\in\mathcal{H}_3.
	\end{equation}
	As $\Upsilon(1;s) = - h(2/s;s)$, we have proven \eqref{eq: lemma i}. Furthermore, \eqref{eq: lemma ii} now follows directly from \eqref{eq: h in terms of upsilon}, \eqref{eq: upsilon sign in h2}, and \eqref{eq: upsilon in h3}. 
\end{proof}
We now open lenses as depicted in Figure~\ref{fig: lens opening double scaling}.
\begin{figure}[h]
	\centering
	\centering
	\begin{tikzpicture}[scale=1.6]
	\draw [thick, dashed] plot [smooth, tension =0.3] coordinates {(-1,2.5) (0,1.6) (1,2.5) };
	\draw [thick, dashed] plot [smooth, tension =1] coordinates{(0,1.6)(-1,0.7)(-1.25, 0) (0,-0.25)(1.25,0)(1,0.7) (0,1.6)};
	\draw [ultra thick, postaction={mid arrow=black}] plot [smooth, tension = 1.5] coordinates {(-1,0)(0,1.6)(1,0)};
	\draw [ultra thick, postaction={mid arrow=black}] plot [smooth, tension = 1] coordinates {(-1,0)(0,0.2)(1,0)};
	\draw [ultra thick, postaction={mid arrow=black}] plot [smooth, tension=1] coordinates {(-1,0) (0,1.2) (1,0)};
	\foreach \Point in {(-1,0), (1,0), (0,1.6)}{
		\node at \Point {\textbullet};
	}
	\node [below] at (-1, -0.3) {$-1$};
	\node [below] at (1, -0.3) {$1$};
	\node [above] at (0.,1.8) {$2/s$};
	\node [below] at (0,0.25) {$\gamma_{m,0}^-$};
	\node [right] at (0.4,1.6) {$\gamma_{m,0}^+$};
	\node [above] at (0,0.65) {$\gamma_{m,0}$};
	\end{tikzpicture}
	\caption{Opening of lenses in the double scaling regime near a regular breaking point. The trajectories of $\varpi_s$ are indicated by dashed lines.}
	\label{fig: lens opening double scaling}
\end{figure}
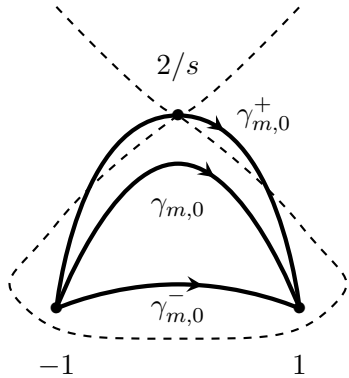
Note that the upper lip of the lens, $\gamma_{m,0}^+$ passes through $z=2/s$ and both $\gamma_{m,0}^\pm$ remain entirely within $\mathcal{H}_2\cup\mathcal{H}_3$. As before, we define $\mathcal{L}_0^\pm$ to be the region bounded between the arcs $\gamma_{m,0}$ and $\gamma_{m,0}^\pm$, respectively, and set $\hat{\Sigma} := \Sigma \cup \gamma_{m,0}^+ \cup \gamma_{m,0}^-$. We can now define the third transformation of the steepest descent process as 
\begin{equation}\label{eq: S to T}
S(z):= \begin{cases}
T(z)\begin{pmatrix}
1 & 0 \\
\mp e^{-nh(z)} & 1
\end{pmatrix}, \qquad &z \in \mathcal{L}_0^\pm, \\
T(z), \qquad &\text{otherwise}.
\end{cases}
\end{equation}
We then consider the model Riemann-Hilbert problem formed by disregarding the jumps on $\gamma_{m,0}^\pm$. In particular, we seek $M$ such that
\begin{subequations}
	\label{eq: M RHP genus 0 DS}
	\begin{alignat}{2}
	&M(z) \text{ is analytic for } z\in \C\setminus \gamma_{m,0}(s), \qquad && \\
	&M_+(z) = M_-(z) \begin{pmatrix}
	0 & 1 \\
	-1 & 0
	\end{pmatrix},  \qquad &&z \in \gamma_{m,0}, \label{eq: genus 0 model problem jump DS}\\
	&M(z) = I + \mathcal{O}\left(\frac{1}{z}\right), \qquad && z \to \infty.
	\end{alignat}
\end{subequations}

The solution to this Riemann-Hilbert problem was provided in Section~\ref{sub: Global parametrix genus 0}, see \eqref{eq: global parametrix eq genus 0}.

Note that the jump on $\gamma_{m,0}^+(s)$ is no longer exponentially decaying to the identity as $s\to s_*$ in a neighborhood of $z=2/s$. Moreover, the matrix $M$ is not bounded near the endpoints $z=\pm 1$. Therefore, we define $D_c:=D_\delta(2/s)$,  $D_{-1}:=D_\delta(-1)$, and $D_{1}:=D_\delta(1)$ to be discs of radius $\delta$ centered at $z=2/s, -1,$ and $1$, respectively. We take $\delta$ small enough so that $D_c \cap\gamma_{m,0}^-=\emptyset$. Note that for $s$ near $s_*$, the trajectory $\gamma_{m,0}(s)$ is close to $2/s_*$, so that for $n$ large enough we must have that $D_c \cap \gamma_{m,0}(s)\not=\emptyset$. In each $D_k$, $k \in \{c, -1, 1\}$, we seek a local parametrix $P^{(k)}$ such that
\begin{subequations}
	\label{eq: P RHP DS}
	\begin{alignat}{2}
	&P^{(k)}(z) \text{ is analytic for } z\in D_\lambda\setminus \hat{\Sigma}, \qquad && \\
	&P^{(k)}_+(z) = P^{(k)}_-(z) j_S(z),  \qquad &&z \in D_k\cap \hat{\Sigma} \\
	&P^{(\lambda)}(z)= M(z)\left(I + o(1)\right), \qquad && n \to \infty, \quad z\in \partial D_k
	\end{alignat}
\end{subequations}
As shown in Section~\ref{sec: Local Parametrices}, $P^{(1)}$ and $P^{(-1)}$ are given by
\begin{subequations}
	\begin{equation}
	\begin{aligned}
	P^{(1)}(z) &= E_n^{(1)}(z) B\left(f_{n,B}(z)\right)e^{-\frac{n}{2}h(z)\sigma_3}\\
	P^{(-1)}(z) &= E_n^{(-1)}(z) \tilde{B}\left(\tilde{f}_{n,B}(z)\right)e^{-\frac{n}{2}h(z)}
	\end{aligned}
	\end{equation}
\end{subequations}
where $\tilde{h}(z)=h(z)-2\pi i$, $B$ is the Bessel parametrix defined in \eqref{eq: Bessel parametrix}, and $\tilde{B}(z)=\sigma_3B(z)\sigma_3$. Above, 
\begin{subequations}
	\begin{equation}
	f_{n,B}(z)= \frac{h(z)^2}{16}, \qquad \tilde{f}_{n,B}(z) = \frac{\tilde{h}(z)^2}{16},
	\end{equation}
	\begin{equation}
	E_n^{(1)}(z) = M(z) L_n^{(1)}(z)^{-1}, 
	\qquad			
	L_n^{(1)}(z) := \frac{1}{\sqrt{2}} \left(2\pi n\right)^{-\sigma_3/2} f_B(z)^{-\sigma_3/4}
	\begin{pmatrix}
	1 & i \\
	i & 1
	\end{pmatrix},
	\end{equation}
	and
	\begin{equation}
	E_n^{(-1)}(z) = M(z) L_n^{(-1)}(z)^{-1}, 
	\qquad			
	L_n^{(-1)}(z) := \frac{1}{\sqrt{2}} \left(2\pi n\right)^{-\sigma_3/2} \tilde{f}_B(z)^{-\sigma_3/4}
	\begin{pmatrix}
	-1 & i \\
	i & -1
	\end{pmatrix}.
	\end{equation}
\end{subequations}
We will now move on to the construction of the local parametrix $P^{(c)}$ within $D_c$.

\subsection{Parametrix around the Critical Point}
We consider a disc $D_c$ around $z=2/s$ of small radius $\delta$. We partition $D_c$ into $D_c^+$ and $D_c^-$ as in Figure~\ref{fig: setups}, so that $D_c^+$ is the region within $D_c$ that lies to the left of $\gamma_{m,0}$ and $D_c^-$ is the region which lies to the right. We define the following function in $D_c^+$:
\begin{equation}
\tilde{h}_c(z;s) = \int_{2/s_*}^z \frac{2-s u}{\left(u^2-1\right)^{1/2}}\, du, \qquad z\in D_c^+,
\end{equation}
where the path of integration does not cross $\gamma_{m,0}(s)$. Note that  $\tilde{h}_c(z;s)$ is analytic within $D_c^+$. Next, denote by $h_c$ the analytic continuation of $\tilde{h}_c$ into $D_c^-$.

\begin{figure}[h]
	\centering
	\begin{tikzpicture}[scale=3.5]
	\draw[ultra thick, postaction={mid arrow=black}] (0,0) arc (0:-360:0.6 and 0.6);
	\foreach \Point in { (-.6,-0), (-.6,-.15)}{
		\node at \Point {\textbullet};
	}
	\draw [ultra thick, fill = gray!35 ] (0,0) arc (0:-140:0.6 and 0.6) to[bend left =20] (-0.07,-0.28) ;
	\draw [ultra thick, postaction = {mid arrow=black}] (-1.06,-.385) to[bend left=20] (-0.07, -0.28);
	\draw [ultra thick, postaction = {mid arrow=black}] plot [smooth, tension =1] coordinates{(-1.175,-0.2) (-.6,0) (0,0)};
	\node [above] at (-0.6,0.05) {$2/s$};
	\node [right] at (-.575,-.12) {$2/s_*$};
	\node [above] at (-0.7,0.3) {$D_c^+$};
	\node [below] at (-0.5,-0.3) {$D_c^-$};
	\node [right] at (-1.3,-.45) {$\gamma_{m,0}$};
	\node [right] at (0,0) {$\gamma_{m,0}^+$};
	\end{tikzpicture}
	\caption{Definitions of the regions $D_c^\pm$ within $D_c$. The region $D_c^-$ is shaded in the figure. }
	\label{fig: setups}
\end{figure}
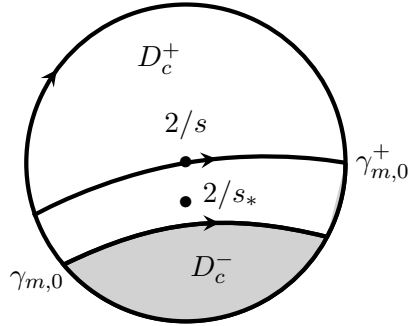

In terms of the $h$ function, we may write
\begin{align}
h_c(z;s) = \begin{cases}
h(z;s)-h\left(\frac{2}{s_*};s\right), \qquad &z\in D_c^+,\\
-h(z;s)-h\left(\frac{2}{s_*};s\right), \qquad & z\in D_c^-.
\end{cases}
\end{align}
We now have the following lemma, following the lines laid out in \cite[Proposition~4.5]{bertola2016asymptotics}.
\begin{lemma}\label{lem: conformal map}
	There exists a jointly analytic function $\zeta(z;s)$ which is univalent in a fixed neighborhood of $z=2/s_*$, with $s$ in a neighborhood of $s_*$, and an analytic function $K(s)$ near $s=s_*$ so that
	\begin{equation}\label{eq: quadratic}
	h_c(z;s)=\frac{1}{2}\zeta^2(z;s)+K(s)\zeta(z;s), 
	\end{equation}
	where $K(2/s_*)=0$ and
	\begin{equation}
		\zeta\left(\frac{2}{s_*},s\right)\equiv 0
	\end{equation}
	for $s$ in a neighborhood of $s_*$. 
\end{lemma}
\begin{proof}
	Define $h_{cr}(s):=h_c(2/s;s)$. Then, we have that 
	\begin{equation}
	h_{cr}(s) = \frac{2}{s_*^3\left(\frac{4}{s_*^2}-1\right)^{1/2}}(s-s_*)^2 \left[1+\mathcal{O}\left(s-s_*\right)\right]. 
	\end{equation}
	Therefore, we may write
	\begin{equation}
	h_{cr}(s) = -\frac{1}{2}K^2(s)
	\end{equation}
	where $K(s)$ is analytic near $s=s_*$ and satisfies
	\begin{equation}
		K(s) = k_1(s-s_*) + \mathcal{O}\left(s-s_*\right)^2, 
	\end{equation}
	where
	\begin{equation}\label{eq: k_1 eqn}
		k_1 = \frac{2i}{s_*^{3/2}}\left(\frac{4}{s_*^2}-1\right)^{-1/4}. 
	\end{equation}
	Moreover, we can calculate that
	\begin{equation}\label{eq: f_c behavior}
	h_c(z;s)-h_{cr}(s) = -\frac{s}{2}\left(\frac{4}{s^2}-1\right)^{-1/2}\left(z-\frac{2}{s}\right)^{2}\left[1+\mathcal{O}\left(z-\frac{2}{s}\right)\right].
	\end{equation}
	Next define
	\begin{equation}\label{eq: zeta def}
	\frac{\zeta(z;s)}{\sqrt{2}}:= \sqrt{h_c(z;s)+\frac{K^2(s)}{2}}-\frac{K(s)}{\sqrt{2}}
	\end{equation}
	We immediately have that $\zeta$ satisfies \eqref{eq: quadratic}, is conformal map in a neighborhood of $z=2/s$ and satisfies $\zeta(2/s_*,s)\equiv 0$.
\end{proof}
We now specify that the size of the disc $D_c$ is chosen to be small enough so that $\zeta(z;s)+K(s)$ is conformal for $n$ large enough (or equivalently, when $s$ is close to $s_*$), which is possible via the lemma above. Moreover, we also impose that the arc $\gamma_{m,0}^+$ is mapped to the real line via $\zeta(z;s)+K(s)$ within $D_c$. 

From the proof of Lemma~\ref{lem: conformal map}, we see that
\begin{equation}
	K(s) = \frac{2i}{s_*^{3/2}}\left(\frac{4}{s_*^2}-1\right)^{-1/4}(s-s_*) + \mathcal{O}\left(s-s_*\right)^2, 
\end{equation}
Therefore, we note that the double scaling limit \eqref{eq: def of double scaling} can be equivalently stated by taking $n\to\infty$ and $s\to s_*$ so that
\begin{equation}\label{eq: alternate double scaling}
	\lim\limits_{n\to\infty, \, s\to s_*} nK(s) =\frac{2iL_1}{s_*^{3/2}}\left(\frac{4}{s_*^2}-1\right)^{-1/4} = L_1 k_1, 
\end{equation}
where $k_1$ is given in \eqref{eq: k_1 eqn}. We may obtain the local parametrix about $z=2/s$ by solving the following Riemann-Hilbert problem: 
\begin{subequations}\label{eq: Pc RHP}
	\begin{alignat}{2}
	&P^{(c)}(z) \text{ is analytic for } z\in D_{c}\setminus \hat{\Sigma}, \qquad &&\label{eq: Pc RHP1} \\
	&P^{(c)}_+(z) = P^{(c)}_-(z) j_S(z),  \qquad &&z \in D_{c}\cap \hat{\Sigma}, \label{eq: Pc RHP2}\\
	&P^{(c)}(z)= \left(I + o\left(1\right)\right)M(z), \qquad && n \to \infty, \quad z\in \partial D_{c},\label{eq: matching condition Pc}
	\end{alignat}
\end{subequations}
We recall that the jumps in \eqref{eq: Pc RHP2} are given by
\begin{align}
P^{(c)}_+(z) = P^{(c)}_-(z) \begin{cases}
\begin{pmatrix}
1 & 0 \\
e^{-nh(z;s)} & 1
\end{pmatrix}, \qquad & z\in D_c\cap\gamma_{m,0}^+(s),\\
\begin{pmatrix}
0 & 1 \\
-1 & 0
\end{pmatrix}, \qquad & z\in D_c\cap\gamma_{m,0}(s)
\end{cases}
\end{align}
We solve for $P^{(c)}$ by first defining $U^{(c)}$ so that
\begin{equation}
P^{(c)}(z) = U^{(c)}(z) e^{-\frac{n}{2}h(z)\sigma_3}. 
\end{equation}
Then, $U^{(c)}$ is also analytic for $z\in D_{c}\setminus \hat{\Sigma}$ and satisfies the following jump conditions within $D_c$:
\begin{align}
U^{(c)}_+(z) = U^{(c)}_-(z) \begin{cases}
\begin{pmatrix}
1 & 0 \\
1 & 1
\end{pmatrix}, \qquad & z\in D_c\cap\gamma_{m,0}^+(s),\\
\begin{pmatrix}
0 & 1 \\
-1 & 0
\end{pmatrix}, \qquad & z\in D_c\cap\gamma_{m,0}(s)
\end{cases}
\end{align}
We may solve for $U^{(c)}$ using the error function parametrix presented in \cite[Section~7.5]{bleher2017topological}. We introduce 
\begin{equation}
C(\zeta) := \begin{pmatrix}
e^{\zeta^2} & 0 \\
b(\zeta) & e^{-\zeta^2}
\end{pmatrix},
\end{equation}
where 
\begin{align}
b(\zeta) := \frac{1}{2}e^{-\zeta^2} \begin{cases}
\textrm{erfc}\left(-i\sqrt{2}\zeta\right), \qquad & \Im \zeta >0, \\
-\textrm{erfc}\left(i\sqrt{2}\zeta\right), \qquad & \Im \zeta <0.
\end{cases}
\end{align}
Then, $C(\zeta)$ is analytic for $\zeta\in\C\setminus\R$ and satisfies
\begin{equation}
C_+(\zeta) = C_-(\zeta)\begin{pmatrix}
1 & 0 \\
1 & 1
\end{pmatrix}, \qquad \zeta\in\R
\end{equation}
and moreover possess the following asymptotic expansion, uniform in the upper and lower half planes:
\begin{equation}
C(\zeta) = \left(I + \sum_{k=0}^\infty\begin{pmatrix}
0 & 0 \\
b_k & 0
\end{pmatrix} \zeta^{-2k-1}\right) e^{\zeta^2\sigma_3}, \qquad \zeta\to\infty, 
\end{equation}
where
\begin{equation}
b_k = \frac{i}{\sqrt{2\pi}}\frac{\Gamma\left(k+\frac{1}{2}\right)}{2^{k+1}\Gamma\left(\frac{1}{2}\right)}.
\end{equation}
Next define, 
\begin{equation}
f_{n,C}(z;s) = \left(\frac{n}{2}\right)^{1/2} f_C(z;s), \qquad f_C(z;s)=\frac{1}{\sqrt{2}}\left(\zeta(z;s)+K(s)\right), 
\end{equation}
where $\zeta$ and $K$ are as defined via Lemma~\ref{lem: conformal map}. Using the proof of Lemma~\ref{lem: conformal map}, we see that $f_{C}(z;s)$ conformally maps a neighborhood of $z=2/s$ to a neighborhood of $z=0$. If we define
\begin{align}
J(z) = \begin{cases}
I, \qquad & z\in D_c^+, \\
\begin{pmatrix}
0 & -1 \\
1 & 0
\end{pmatrix}, \qquad &z\in D_c^-, 
\end{cases}
\end{align}
we see that
\begin{equation}
P^{(c)}(z) = E_n^{(c)}(z) C\left(f_{n,C}(z)\right) J(z) e^{-\frac{n}{2}h(z)\sigma_3},
\end{equation}
where $E_n^{(c)}$ is any matrix which is analytic throughout $D_c$, solves \eqref{eq: Pc RHP1} and \eqref{eq: Pc RHP2}. We now choose $E_n^{(c)}$ so that $P^{(c)}$ satisfies \eqref{eq: matching condition Pc}. As $n\to\infty$ for $z\in D_c^+$, we have
\begin{equation}
P^{(c)}(z) = E^{(c)}_n(z)\left(I +\sum_{k=0}^\infty \begin{pmatrix}
0 & 0 \\
b_k & 0
\end{pmatrix}\left(\frac{2}{n}\right)^{k+1/2}\left(f_C(z;s)\right)^{-2k-1}\right)e^{\frac{n}{2}\left[f_C^2(z;s)-h(z;s)\right]\sigma_3}
\end{equation}

Similarly, we have that as $n\to\infty$ for $z\in D_c^-$, 
\begin{equation}
P^{(c)}(z) = E^{(c)}_n(z)\left(I +\sum_{k=0}^\infty \begin{pmatrix}
0 & 0 \\
b_k & 0
\end{pmatrix}\left(\frac{2}{n}\right)^{k+1/2}\left(f_C(z;s)\right)^{-2k-1}\right)e^{\frac{n}{2}\left[f_C^2(z;s)+h(z;s)\right]\sigma_3}J(z)
\end{equation}
Therefore, if we set
\begin{equation}
E_n^{(c)}(z) = M(z) J^{-1}(z)e^{-\frac{n}{2}\left[K^2(s)/2- h(2/s_*;s)\right]\sigma_3 } \qquad z\in D_c,
\end{equation} 
we see that $P_n^{(c)}(z)$ satisfies the matching condition \eqref{eq: matching condition Pc}. It is easy enough to see that $E_n^{(c)}$ is analytic within $D_c$ as both $M$ and $J$ have the same jumps over $\gamma_{m,0}$ and are bounded within $D_c$. Moreover, we see that
\begin{equation}\label{eq: expansion PC fake}
P^{(c)}(z) =  \left(I+ n^{-1/2}\sum_{k=0}^\infty \frac{P_{k,n}(z;s)}{n^k} \right)M(z), \qquad n\to\infty, 
\end{equation}
where
\begin{align}\label{eq: Pk defn}
	P_{k,n}(z;s) = \frac{2^{k+1/2}}{f_C(z;s)^{2k+1}}e^{\frac{n}{2}\left(K^2(s)-2h(2/s_*;s)\right)}\begin{cases}
	\begin{pmatrix}
	0 & 0 \\
	b_k & 0
	\end{pmatrix}, \qquad &z \in D_c^+,\\
	\begin{pmatrix}
	0 & -b_k \\
	0 & 0
	\end{pmatrix}, \qquad &z \in D_c^-.
	\end{cases}
\end{align}

Now, as $s\to s_*$, 
\begin{subequations}
\begin{align}
	K^2(s)-2h(2/s_*;s)&=-2 h(2/s_*;s_*) + 2\left(\frac{4}{s_*^2}-1\right)^{1/2}(s-s_*) + k_1^2(s-s_*)^2+\mathcal{O}(s-s_*)^3 \\
	&=-2 h(2/s_*;s_*) + 2L_1\left(\frac{4}{s_*^2}-1\right)^{1/2}\frac{1}{n} + \frac{L_1^2k_1^2}{n^2}+\mathcal{O}\left(\frac{1}{n^3}\right) .
\end{align}
\end{subequations}
Moreover, as $s_*$ is a regular breaking point, we have that $h(2/s_*;s_*)=i \kappa$, where $\kappa\in\R$. Then, as $n\to\infty$ (and as such $s\to s_*$), 
\begin{equation}
	e^{\frac{n}{2}\left(K^2(s)-2h(2/s_*;s)\right)} = e^{-in\kappa}\exp\left(L_1\left(\frac{4}{s_*^2}-1\right)^{1/2}\right)\left(1+\frac{L^2k_1^2}{2n}+\mathcal{O}\left(\frac{1}{n}\right)\right).
\end{equation}
We then have that
\begin{equation}\label{eq: expansion PC}
	P^{(c)}(z) =\left(I+ n^{-1/2}\sum_{k=0}^\infty \frac{P_{k}(z;s)}{n^k} \right) M(z) , \qquad n\to\infty, 
\end{equation}
where $P_0$ is given by
\begin{align}\label{eq: P0 def}
	P_0(z;s) = \frac{\sqrt{2}\delta_n(L_1)}{f_C(z;s)}\begin{cases}
	\begin{pmatrix}
	0 & 0 \\
	\frac{i}{2\sqrt{2\pi}} & 0
	\end{pmatrix}, \qquad &z \in D_c^+,\\
	\begin{pmatrix}
	0 & -\frac{i}{2\sqrt{2\pi}}  \\
	0 & 0
	\end{pmatrix}, \qquad &z \in D_c^-.
	\end{cases}
\end{align}
where for ease of notation we have defined
\begin{equation}\label{eq: delta n}
	\delta_n(L_1) := e^{-in\kappa}\exp\left(L_1\left(\frac{4}{s_*^2}-1\right)^{1/2}\right).
\end{equation}
Note above that $\left|e^{-in\kappa}\right|=1$ as
\begin{equation}\label{kappa}
	\kappa = \Im h(2/s_*;s_*).
\end{equation}

\subsection{Proof of Theorem~\ref{thm: rec coeff double scaling regular}}
The final transformation is 
\begin{align}\label{eq: final transformation crit}
R(z)= S(z) \begin{cases}
M(z)^{-1}, \qquad & z\in \C\setminus\overline{\left(D_{-1}\cup D_1 \cup  D_c\right)} \\
P^{(-1)}(z)^{-1}, \qquad & z\in D_{-1} \\
P^{(1)}(z)^{-1}, \qquad & z\in D_{1} \\
P^{(c)}(z)^{-1}, \qquad & z\in D_c 
\end{cases}
\end{align}
We write the jump matrix $j_R(z) = I + \Delta(z)$, where
\begin{equation}\label{eq: Delta expansion critical}
\Delta(z) = \sum_{k=1}^\infty \frac{\Delta_{k/2}(z)}{n^{k/2}}.
\end{equation}
As before, we have that $\Delta_k(z)=0$ for $z\in\Sigma_R\setminus\left(\partial D_{-1} \cup\partial D_1 \cup\partial D_c \right)$, as the jump matrix decays exponentially quickly to the identity off of the boundaries of the discs $D_{-1}$, $D_1$, and $D_c$. From \eqref{eq: Delta k 1}, \eqref{eq: Delta k -1}, and \eqref{eq: expansion PC}, we have for $k\in \N$ that
\begin{subequations}
	\begin{align}\label{eq: Delta k int}
	\Delta_k(z) = \begin{cases}
	\displaystyle
	\frac{(-1)^{k-1}\prod_{j=1}^{k-1}(2j-1)^2}{4^{k-1}(k-1)!\tilde{h}(z)^k} M(z)\begin{pmatrix}
	\frac{(-1)^k}{k} \left(\frac{k}{2}-\frac{1}{4}\right) & i \left(k-\frac{1}{2}\right) \\
	(-1)^{k+1} i \left(k-\frac{1}{2}\right) &\frac{1}{k}\left(\frac{k}{2}-\frac{1}{4}\right)
	\end{pmatrix}M^{-1}(z), \quad &z\in D_{-1}\\
	\displaystyle
	\frac{(-1)^{k-1}\prod_{j=1}^{k-1}(2j-1)^2}{4^{k-1}(k-1)!h(z)^k} M(z)\begin{pmatrix}
	\frac{(-1)^k}{k} \left(\frac{k}{2}-\frac{1}{4}\right) & -i \left(k-\frac{1}{2}\right) \\
	(-1)^k i \left(k-\frac{1}{2}\right) &\frac{1}{k}\left(\frac{k}{2}-\frac{1}{4}\right)
	\end{pmatrix}M^{-1}(z), \quad &z\in D_{1} \\
	0, \quad & z\in D_c, 
	\end{cases}
	\end{align}
	and
	\begin{align}\label{eq: Delta k half int}
	\Delta_{k+\frac{1}{2}}(z) = \begin{cases}
	0  \qquad &z\in D_1\cup D_{-1} \\
	M(z)P_k(z;s)M^{-1}(z), \qquad &z\in D_c, 
	\end{cases}
	\end{align}
\end{subequations}
where we have used \eqref{eq: expansion PC}. As $\Delta(z)$ possesses the expansion \eqref{eq: Delta expansion critical}, we may again use the arguments presented in \cite[Section~7]{deift1999strong} and \cite[Section~8]{arnojacobi} to conclude that $R$ has an asymptotic expansion in inverse powers of $n^{1/2}$ of the form
\begin{equation}\label{eq: R expansion c}
R(z) = \sum_{k=0}^{\infty} \frac{R_{k/2}(z)}{n^{k/2}}, \qquad n\to\infty, 
\end{equation}
where each $R_{k/2}$ solves the following Riemann-Hilbert problem:
\begin{subequations}
	\label{eq: R_k RHP c}
	\begin{alignat}{2}
	&R_{k/2}(z) \text{ is analytic for } z \in \C\setminus  \left(\partial D_{-1} \cup \partial D_{-1}\cup \partial D_c\right) \qquad &&\\
	& R_{k/2,+}(z) = R_{k/2,-}(z) + \sum_{j=1}^{k-1}R_{(k-j)/2,-}\Delta_{j/2}(z), \qquad &&z\in \partial D_{-1} \cup \partial D_{-1}\cup \partial D_c\\
	& R_{k/2}(z) = \frac{R_{k/2}^{(1)}}{z}+\frac{R_{k/2}^{(2)}}{z^2}+\mathcal{O}\left(\frac{1}{z}\right), \qquad &&z \to\infty.
	\end{alignat}
\end{subequations}
Above, we have $R_0(z)\equiv I$. Following \cite{arnojacobi}, we have the following lemma.
\newline
\newline

\begin{lemma}\label{lem: behavior of Delta}
	$\,$
	\begin{enumerate}[(i)]
		\item The restriction of $\Delta_1$ to $\partial D_{-1}$ has a meromorphic continuation to a neighborhood of $D_{-1}$. This continuation is analytic, except at $-1$, where $\Delta_1$ has a pole of order $1$. 
		\item The restriction of $\Delta_1$ to $\partial D_{1}$ has a meromorphic continuation to a neighborhood of $D_{1}$. This continuation is analytic, except at $1$, where $\Delta_1$ has a pole of order at most $1$. 
		\item The restriction of $\Delta_{1/2}$ to $\partial D_{c}$ has a meromorphic continuation to a neighborhood of $D_{c}$. This continuation is analytic, except at $2/s$, where $\Delta_{1/2}$ has a pole of order at most $1$. 
	\end{enumerate}
\end{lemma}
\begin{proof}
	$(i)$ and $(ii)$ are given in \cite[Lemma~8.2]{arnojacobi}, so we prove $(iii)$. As both $M$ and $P_k(z;s)$ are analytic within $D_c^\pm$, we have that $\Delta_{1/2}(z)$ is analytic in both $D_c^\pm$. Furthermore, it is straightforward to check using \eqref{eq: P0 def} and \eqref{eq: genus 0 model problem jump DS} that
	\begin{equation}
	\Delta_{1/2,+}(z)= \Delta_{1/2,-}(z), \qquad z\in \gamma_{m,0}, 
	\end{equation}
	so that $\Delta_{1/2}(z)$ is analytic in $D_c\setminus\{2/s\}$. As $f_C(z;s) = \mathcal{O}\left(z-2/s\right)$ as $z\to 2/s$, we have by \eqref{eq: Pk defn} that the isolated singularity is pole of order $1$.  
\end{proof}
By \eqref{eq: S to T} and \eqref{eq: final transformation crit} we have that $T(z)= R(z)M(z)$ for $z$ outside of the lens. Using \eqref{eq: R expansion c}, we then have that
\begin{subequations}
	\label{eq: T expansion}
	\begin{equation}
	T^{(1)} = M^{(1)} + \frac{R_{1/2}^{(1)}}{n^{1/2}}+\frac{R_1^{(1)}}{n}+\mathcal{O}\left(\frac{1}{n^{3/2}}\right), \qquad n\to\infty, 
	\end{equation}
	\begin{equation}
	T^{(2)}= M^{(2)}+\frac{R_{1/2}^{(1)}M^{(1)}+R_{1/2}^{(2)}}{n^{1/2}}+\frac{R_1^{(1)}M^{(1)}+R_1^{(2)}}{n}+\mathcal{O}\left(\frac{1}{n^{3/2}}\right), \qquad n\to\infty, 
	\end{equation}
\end{subequations}
where $M^{(1)}$ and $M^{(2)}$ were calculated in \eqref{eq: M1 and M2} as
\begin{equation}
M^{(1)} = \begin{pmatrix}
0 & \frac{i}{2} \\
-\frac{i}{2} & 0
\end{pmatrix}, \qquad 
M^{(2)} = \begin{pmatrix}
\frac{1}{8} & 0\\
0& \frac{1}{8} 
\end{pmatrix}.
\end{equation}
We first solve for $R_{1/2}(z)$. Using Lemma~\ref{lem: behavior of Delta}, we may write
\begin{equation}
\Delta_{1/2}(z) = \frac{C^{(1/2)}}{z-2/s}, \qquad z\to 2/s, 
\end{equation}
for some constant matrix $C^{(1/2)}$. Using the explicit expression \eqref{eq: Delta k half int} for $\Delta_{1/2}$, we can calculate $C^{(1/2)}$ as
\begin{equation}\label{eq: c 1/2}
C^{(1/2)} = \frac{\delta_n(L_1)}{2s\sqrt{\pi}}\begin{pmatrix}
	1 & -\frac{s\left(\frac{4}{s^2}-1\right)^{1/2}-2}{is} \\
	\frac{s\left(\frac{4}{s^2}-1\right)^{1/2}+2}{is} & -1
\end{pmatrix}
\end{equation}
where we have used \eqref{eq: zeta def} to calculate that
\begin{equation}
f_C(z;s) = -\frac{s}{2}\left(\frac{4}{s^2}-1\right)^{-1/2}\left(z-\frac{2}{s}\right)+\mathcal{O}\left(z-\frac{2}{s}\right)^2
\end{equation}
Then
\begin{align}
R_{1/2}(z) := \begin{cases}
\displaystyle\frac{C^{(1/2)}}{z-2/s}, \qquad &z\in\C\setminus D_c, \\[2mm]
\displaystyle\frac{C^{(1/2)}}{z-2/s} - \Delta_{1/2}(z), \qquad &z\in D_c,
\end{cases}
\end{align}
solves \eqref{eq: R_k RHP c} with $k=1$. Next, as shown in \eqref{eq: A1 and B1 def} and \eqref{eq: A1 and B1}, 
\begin{align}
\Delta_1(z) = \begin{cases}
\displaystyle\frac{A^{(1)}}{z-1} + \mathcal{O}\left(1\right), \qquad & z\to 1, \\[2mm]
\displaystyle\frac{B^{(1)}}{z+1}  + \mathcal{O}\left(1\right), \qquad & z\to -1,
\end{cases}
\end{align}
where
\begin{equation}
A^{(1)} = \frac{1}{8(s-2)}\begin{pmatrix}
-1 & i \\
i & 1
\end{pmatrix}, \qquad B^{(1)} = \frac{1}{8(s+2)}\begin{pmatrix}
-1 &- i \\
-i & 1
\end{pmatrix}.
\end{equation}
We can then compute that
\begin{align}
R_{1/2}(z)\Delta_{1/2}(z)+ \Delta_1(z) = \begin{cases}
\displaystyle\frac{A^{(1)}}{z-1} + \mathcal{O}\left(1\right), \qquad & z\to 1, \\[2mm]
\displaystyle\frac{B^{(1)}}{z+1}  + \mathcal{O}\left(1\right), \qquad & z\to -1, \\[2mm]
\displaystyle\frac{C^{(1)}}{z-2/s} + \mathcal{O}\left(1\right), \qquad & z\to 2/s,
\end{cases}
\end{align}
where
\begin{equation}
C^{(1)} = -\frac{\delta_n^2(L_1)}{4\pi s^2\left(\frac{4}{s^2}-1\right)^{1/2}}\begin{pmatrix}
1 & -\frac{s\left(\frac{4}{s^2}-1\right)^{1/2}-2}{is} \\
\frac{s\left(\frac{4}{s^2}-1\right)^{1/2}+2}{is} & -1
\end{pmatrix}
\end{equation}
Then, 
\begin{align}
R_1(z) = \begin{cases}
\displaystyle
\frac{A^{(1)}}{z-1} +\frac{B^{(1)}}{z+1}+ \frac{C^{(1)}}{z-2/s}, \qquad &z\in\C\setminus\left(D_{-1}\cup D_1\cup D_c\right),\\[2mm]
\displaystyle
\frac{A^{(1)}}{z-1} +\frac{B^{(1)}}{z+1}+ \frac{C^{(1)}}{z-2/s} - R_{1/2}(z)\Delta_{1/2}(z)-\Delta_1(z), \qquad & z\in D_{-1}\cup D_1\cup D_c,
\end{cases}
\end{align}
solves the Riemann-Hilbert problem \eqref{eq: R_k RHP c} with $k=2$. As we now have explicit expressions for $R_{1/2}$ and $R_1$, we may expand at infinity to get
\begin{subequations}
	\begin{equation}
	R_{1/2}^{(1)} = C^{(1/2)}, \qquad 
	R_{1/2}^{(2)} = \frac{2}{s} C^{(1/2)}
	\end{equation}
	\begin{equation}
	R_{1}^{(1)} = A^{(1)}+B^{(1)}+C^{(1)}, \qquad
	R_{1}^{(2)} = A^{(1)}-B^{(1)}+\frac{2}{s} C^{(1)}
	\end{equation}
\end{subequations}

Using \eqref{eq: recurrence coeffs in terms of T} and \eqref{eq: T expansion}, we may now calculate that
\begin{subequations}
	\begin{equation}
	\alpha_n(s)=\frac{\delta_n\left(s^2+2s\left(\frac{4}{s^2}-1\right)^{1/2}-4\right)}{\sqrt{\pi}s^3}\frac{1}{n^{1/2}}+\frac{2\delta_n^2\left(s^2+4s\left(\frac{4}{s^2}-1\right)^{1/2}-8\right)}{\pi s^5}\frac{1}{n}+\mathcal{O}\left(\frac{1}{n^{3/2}}\right)
	\end{equation}
	and
	\begin{equation}
	\beta_n(s) = \frac{1}{4} + \frac{\delta_n}{2\sqrt{\pi}s}\left(\frac{4}{s^2}-1\right)^{1/2}\frac{1}{n^{1/2}}-\frac{\delta_n^2}{2\pi s^2}\frac{1}{n} + \mathcal{O}\left(\frac{1}{n^{3/2}}\right),
	\end{equation}
\end{subequations}
as $n\to\infty$, where we recall that 
\begin{equation}
	\delta_n =\delta_n(L_1)= e^{-in\kappa}\exp\left(L_1\left(\frac{4}{s_*^2}-1\right)^{1/2}\right).
\end{equation}

\section{Double Scaling Limit near a Critical Breaking Point}\label{sec: DS Crit}
We now take $s$ in a double scaling regime near the critical point $s=2$ as
\begin{equation}\label{eq: ds critical}
s = 2 + \frac{L_2}{n^{2/3}}, 
\end{equation}
where $L_2<0$. Note that as $L_2<0$, we have that $s\in\mathfrak{G}_0$ for large enough $n$.

\subsection{Outline of Steepest Descent}
Although we are now considering the case where $s$ depends on $n$ via the double scaling limit \eqref{eq: ds critical}, the first two transformations of steepest descent remain unchanged to the previous analysis, and as such, we summarize the steps briefly and refer the reader to Section~\ref{sec: steepest descent} for full details.

As $s \in \mathfrak{G}_0$ for $n$ large enough, we have immediately that there is a genus $0$ $h$-function satisfying \eqref{eq: RHP for h}, with $L=0$, and \eqref{eq: both inequalities}. 
Finally, we remark that as we are in the genus $0$ regime, we have an explicit formula for the $h$ function, given in \eqref{eq: genus 0 h function explicit} as
\begin{equation}
h(z;s) = 2 \log\left(z+\left(z^2-1\right)^{1/2}\right)-s\left(z^2-1\right)^{1/2}.
\end{equation}

We recall that $Y$ defined in \eqref{eq: Y eqn} solves the Riemann-Hilbert problem \eqref{eq: Y RHP}. By making the transformations $Y \mapsto T \mapsto S$ as described in Section \ref{sec: steepest descent} we arrive at a matrix $S$ that satisfies
%
\begin{subequations}
	\label{eq: S RHP ds}
	\begin{alignat}{2}
	&S(z) \text{ is analytic for } z\in \C\setminus \hat{\Sigma}, \qquad && \\
	&S_+(z) = S_-(z) j_S(z), \qquad && z \in \hat{\Sigma}, \\
	&S(z) = I + \mathcal{O}\left(\frac{1}{z}\right), \qquad && z \to \infty, 
	\end{alignat}
\end{subequations}
where
\begin{align}
j_S(z) = \begin{cases}
\begin{pmatrix}
1 & 0 \\
e^{-nh(z)} & 1
\end{pmatrix}, \qquad z\in \gamma_{m,0}^\pm, \\
\begin{pmatrix}
0 & 1\\
-1 & 0
\end{pmatrix}, \qquad z\in\gamma_{m,0}.
\end{cases}
\end{align}

To complete the process of nonlinear steepest descent, we must find suitable global and local parametrices, $M(z)$ and $P^{(\pm 1)}(z)$. We have seen in Section~\ref{sub: Global parametrix genus 0} that $M(z)$ is given by \eqref{eq: global parametrix eq genus 0}. 
Moreover, we have that the local parametrix $P^{(-1)}(z)$ is given by \eqref{eq: p-1 Bessel}.

The main difference between the case of regular points and the critical breaking point at $s=2$ comes in the analysis about $z=1$. Note that the map 
\begin{equation}
f_{n,B}(z;s) = \frac{h(z;s)^2}{16}
\end{equation}
defined in \eqref{eq: Bessel conformal map} is no longer conformal when $s=2$. Indeed, 
\begin{equation}
f_{n,B}(z;s) = \frac{(s-2)^2}{8} (z-1) + \frac{(s-2)(3s+2)}{48}(z-1)^2+\mathcal{O}\left(\left(z-1\right)^3\right), \qquad z\to 1, 
\end{equation}
so that $f_{n,B}(z,2) = \mathcal{O}\left(\left(z-1\right)^3\right)$ as $z\to 1$. Therefore, a different analysis will be needed in $D_1$ in the double scaling limit \eqref{eq: ds critical}.

\subsection{Local parametrix at $z=1$.}
We consider a disc, $D_1$, around $z=1$ of fixed radius $\delta>0$. The local parametrix about $z=1$ solves the following Riemann-Hilbert problem
\begin{subequations}\label{eq: P RHP ds critical}
	\begin{alignat}{2}
	&P^{(1)}(z) \text{ is analytic for } z\in D_{1}\setminus \hat{\Sigma}, \qquad && \\
	&P^{(1)}_+(z) = P^{(1)}_-(z) j_S(z),  \qquad &&z \in D_{1}\cap \hat{\Sigma}, \\
	&P^{(1)}(z)= \left(I + o(1)\right)M(z), \qquad && n \to \infty, \quad z\in \partial D_{1},\label{eq: matching condition ds critical}
	\end{alignat}
\end{subequations}
We will solve for $P^{(1)}$ by setting $P^{(1)}(z) = U^{(1)}(z) e^{-\frac{n}{2}h(z)\sigma_3},$ where $U^{(1)}$ has the following jumps over $\hat{\Sigma}$ within $D_1$, 
\begin{align}
U^{(1)}_+(z) = U^{(1)}_-(z)\begin{cases}
\begin{pmatrix}
1 & 0 \\
1 & 1
\end{pmatrix}, \qquad &z\in D_1\cap\left(\gamma_{m,0}^+\cup\gamma_{m,0}^-\right),\\
\begin{pmatrix}
0 & 1\\
-1 & 0
\end{pmatrix}, \qquad &z\in D_1\cap\gamma_{m,0}
\end{cases}
\end{align}

We will solve this local problem using a parametrix related to the Painlev\'e II and Painlev\'e XXXIV differential equations.

\subsubsection{The Painlev\'e XXXIV Parametrix}	
Let $q=q(w)$ be a solution of the Painlev\'e II equation 
\begin{equation}
q''=wq+2q^3-\alpha, \qquad \alpha\in\mathbb{C}.
\end{equation}
We define the following function $D=D(w)$, which is closely related to the Hamiltonian function for Painlev\'e II:
\begin{equation}\label{DD}
D=(q')^2-q^4-wq^2+2\alpha q.
\end{equation}
Next, we consider the following Riemann--Hilbert problem, which appears in \cite{IKO_critical,IKO_asymptotics,WuXuZhao,xu2011painleve,yattselev2016Angelesco}. This problem appears in works related to orthogonal polynomials on the real line and Hermitian random matrix ensembles with a Fisher--Hartwig singularity or with critical behavior at the edge of the spectrum.	

Let $\Gamma=\Gamma_1 \cup \Gamma_2 \cup \Gamma_3 \cup \Gamma_4$, where $\Gamma_1 =\left\{\arg \zeta =-\frac{2\pi}{3}\right\}$, $\Gamma_2 =\left\{\arg \zeta =0\right\}$, $\Gamma_3 =\left\{\arg \zeta =\frac{2\pi}{3}\right\}$, and $\Gamma_4 =\left\{\arg \zeta =\pi\right\}$, with orientation as in Figure~\ref{fig:RHPsi}, and define the sectors $\Omega_j$ as in Figure~\ref{fig:RHPsi}.

\begin{figure}[h]
	\centering
	\begin{tikzpicture}[scale=1.5]
	\draw[ultra thick,postaction={mid arrow=black}](-1.15,2) to (0,0);
	\draw[ultra thick,postaction={mid arrow=black}](-2.5,0) to (0,0);
	\draw[ultra thick,postaction={mid arrow=black}](-1.15,-2) to (0,0);
	\draw[ultra thick,postaction={mid arrow=black}](0,0) to (2.5,0);
	\draw [thick,domain=0:120] plot ({0.3*cos(\x)}, {0.3*sin(\x)});
	\node at (1,1) {$\Omega_3$};
	\node at (-1.5,1) {$\Omega_4$};
	\node at (-1.5,-1) {$\Omega_1$};
	\node at (1,-1) {$\Omega_2$};
	
	\node at (-0.55,1.5) {$\Gamma_3$};
	\node at (-1.5,0.2) {$\Gamma_4$};
	\node at (-0.55,-1.5) {$\Gamma_1$};
	\node at (2,0.2) {$\Gamma_2$};
	
	\node at (0.45,0.25) {$\frac{2\pi}{3}$};
	
	\draw [dashed] (-2,0) to (2,0);
	\draw [dashed] (0,-2) to (0,2);
	
	\end{tikzpicture}
	\caption{Contour for the RH problem for $\Psi_\alpha(\zeta;w)$. }
	\label{fig:RHPsi}
\end{figure}
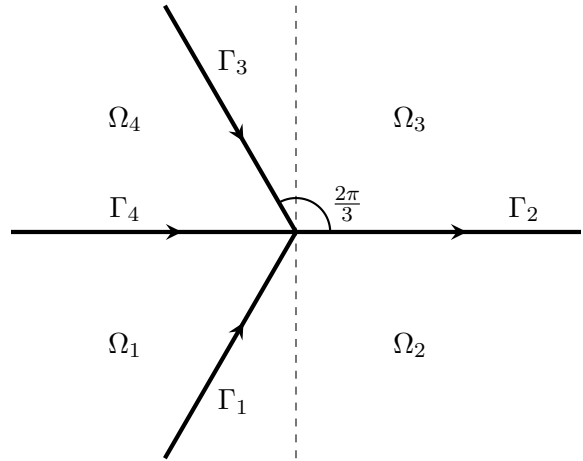
Consider the following Riemann-Hilbert problem for $\Psi(\zeta, w)$ posed on $\Gamma$. 
\begin{subequations}\label{eq: Psi RHP}
	\begin{alignat}{2}
	&\Psi(\zeta, w) \text{ is analytic for } \zeta\in\C\setminus\left(\Gamma_1\cup\Gamma_3\cup\Gamma_4\right), \qquad && \\
	&\Psi_{+}(\zeta, w) = \Psi_{-}(\zeta, w)  \begin{cases}
	\begin{pmatrix}
	1 & 0\\
	1 & 1
	\end{pmatrix}, \qquad & \zeta \in \Gamma_1\cup\Gamma_3,\\
	\begin{pmatrix}
	1 & a_2\\
	0 & 1
	\end{pmatrix}, \qquad & \zeta \in \Gamma_2,\\
	\begin{pmatrix}
	0 & 1\\
	-1 & 0
	\end{pmatrix}, \qquad &\zeta \in \Gamma_4,
	\end{cases} \\
	&\Psi(\zeta, w)= \left(1+\frac{\Psi_{1}(w)}{\zeta}+\mathcal{O}\left(\frac{1}{\zeta^2}\right)\right) \zeta^{-\sigma_3/4} \left(\frac{I + i\sigma_1}{\sqrt{2}}\right)e^{-\left(\frac{4}{3}\zeta^{3/2}-w\zeta^{1/2}\right)\sigma_3}, \qquad \zeta \to\infty,\label{eq: Psi infinity}\\
	&\Psi(\zeta,w) = \begin{cases}
	\mathcal{O}\begin{pmatrix}
	1 & \log \zeta \\
	1 & \log \zeta
	\end{pmatrix}, \qquad &\zeta\in \Omega_2\cup\Omega_3, \\
	\mathcal{O}\begin{pmatrix}
	\log \zeta & \log \zeta\\
	\log \zeta & \log \zeta
	\end{pmatrix}, \qquad & \zeta\in\Omega_1\cup\Omega_4,
	\end{cases}\label{eq: behavior at 0 Psi}
	\end{alignat}
\end{subequations}
where
\begin{equation}
\sigma_1=\begin{pmatrix}
0 & 1\\
1 & 0
\end{pmatrix}.
\end{equation}

In \cite[Section~2]{xu2011painleve}, it is shown{\footnote{Our $\Psi$ function corresponds to $\Psi_0$ in their notation.}}, via a vanishing lemma (Lemma 1), that this Riemann--Hilbert problem has a unique solution for all real values of $w$ if $a_2\in\mathbb{C}\setminus(-\infty,0)$. In the present case, we are taking $a_2=0$ (therefore, no jump on $\Sigma_2$), so the result applies. This existence result also follows from \cite[Proposition 2.3]{IKO_critical}, identifying $\Psi(\zeta,w)$ with the function $\Psi^{(spec)}(\zeta,s)$ in their notation. 


In order to calculate the entries of the matrix $\Psi_{1}(w)$ in \eqref{eq: Psi infinity}, that will be needed later to obtain the asymptotics of the recurrence coefficients, we use the fact that this Riemann--Hilbert problem originates from a folding procedure of the Flaschka--Newell one for Painlev\'e II. Applying formulas (25) and (37) in \cite{xu2011painleve}, we have
\begin{equation}\label{PsiPhi}
\Psi(\zeta,w)
=
\begin{pmatrix}
1 & 0 \\
\displaystyle -\frac{D+q}{2i} & 1
\end{pmatrix}
\zeta^{-\frac{\sigma_3}{4}}
\frac{1}{\sqrt{2}}
\begin{pmatrix}
1 & i \\
i & 1
\end{pmatrix}
\Phi(i\zeta^{\frac{1}{2}},w),	
\end{equation}
where $\Phi(\lambda,w)$ solves a Riemann--Hilbert problem corresponding to Painlev\'e II, see \cite[Section 2]{xu2011painleve} and also \cite[Theorem 5.1 and (5.0.51)]{FIKN2007}. Here $q=q(w)$ solves Painlev\'e II and $D=D(w)$ is given by \eqref{DD}. Furthermore, we observe that the solution $\Psi(\zeta,w)$ that we study corresponds to the Stokes multipliers $b_1=0$ and $b_2=b_4=1$, in the notation used in \cite[\S 1.3]{IKO_asymptotics}, and therefore $a_2=0$ and $a_1=a_3=-i$ in terms of the Stokes multipliers for Painlev\'e II, see \cite[(A.10)]{IKO_asymptotics}. This is in fact the generalized Hastings--McLeod solution to Painlev\'e II, with parameter $\alpha=1/2$, which is characterized by the following asymptotic behavior:
\begin{equation}\label{asympHM}
\begin{aligned}
q_{\rm HM}(x)&=\sqrt{-\frac{x}{2}}+\mathcal{O}(x^{-1}), \qquad x\to-\infty,\\
q_{\rm HM}(x)&=\frac{\alpha}{x}+\mathcal{O}(x^{-4})
=
\frac{1}{2x}+\mathcal{O}(x^{-4}), \qquad x\to+\infty.
\end{aligned}
\end{equation}
Further properties of the Painlev\'e functions associated to $\Psi(\zeta,w)$ are proved in \cite[Lemma 3.5]{IKO_critical}.

As $\lambda\to\infty$, we have the expansion 
\begin{equation}\label{asympPsiPII}
\Phi(\lambda,w)
=
\left(I+\frac{m_1(w)}{\lambda}+\frac{m_2(w)}{\lambda^2}+\mathcal{O}(\lambda^{-3})\right)
e^{-i\left(\frac{4}{3}\lambda^3+w\lambda\right)\sigma_3},
\end{equation}
where the entries of the matrices $m_1(w)$ and $m_2(w)$ are given explicitly in formula (21) in \cite{xu2011painleve}, see also \cite[(5.0.7)]{FIKN2007}, again in terms of $u$, $u'$ and $D$ (we omit the dependence on $w$ for brevity):
\begin{equation}\label{m1m2}
m_1(w)
=
\frac{1}{2}
\begin{pmatrix}
-iD & q\\
q & iD
\end{pmatrix}, \qquad
m_2(w)
=
\frac{1}{8}
\begin{pmatrix}
q^2-D^2 & 2i(qD+q')\\
-2i(qD+q') & q^2-D^2
\end{pmatrix}.
\end{equation}

Combining \eqref{PsiPhi}, \eqref{asympPsiPII} and \eqref{m1m2}, we arrive at the following formulas for the entries of the matrix $\Psi_1(w)$ in \eqref{eq: Psi infinity}:
\begin{equation}\label{entries_Psi1}
\Psi_{1,11}=\frac{D^2-q^2}{8}-\frac{qD+q'}{4},\qquad \Psi_{1,22}=-\frac{D^2-q^2}{8}+\frac{qD+q'}{4}, \qquad \Psi_{1,12}=\frac{i}{2}(D-q).
\end{equation}			



\subsubsection{Construction of the Local Parametrix}
We now continue to build the local parametrix in the disc $D_1$. First, we have the following lemma, following the ideas laid out in 
\cite[Proposition~4.5]{bertola2018painleve}, see also \cite[\S 9.5.1]{yattselev2016Angelesco} and \cite[Lemma 7.6]{ady2020JacobiAngelesco}.
\begin{lemma}\label{lem: conformal map crit}
	There exists a function $\zeta(z;s)$ which is conformal in a fixed neighborhood of $z=1$, with $s$ close to $2$, and an analytic function $A(s)$, such that
	\begin{equation}
	-\frac{h(z)}{2} = \frac{4}{3} \zeta(z;s)^{3/2} - A(s)\zeta(z;s)^{1/2}, 
	\end{equation}
	and
	\begin{equation}
	\zeta(1,s)\equiv 0, \qquad A(2) = 0. 
	\end{equation}
\end{lemma} 
\begin{proof}
	As $h$ has a critical point at $z = \frac{2}{s}$, we write
	\begin{equation}\label{eq: h cubic}
	h_{cr}(s) = h\left(\frac{2}{s},s\right) = 	2\log\left(\frac{2}{s}+\left(\frac{4}{s^2}-1\right)^{1/2}\right)-s\left(\frac{4}{s^2}-1\right)^{1/2}.
	\end{equation}
	Near $s=2$, we see that $h_{cr}(s) = \mathcal{O}\left((s-2)^{3/2}\right)$ and $h_{cr}(s) <0$ for $s<2$, so that
	\begin{equation}\label{eq: A def}
	h_{cr}(s) = \frac{2}{3} A^{3/2}(s), 
	\end{equation}
	for some $A(s)$ analytic in a neighborhood of $s=2$ satisfying $A(s) =\mathcal{O}(s-2)$ as $s\to 2$ and $A(s) >0$ for $s <2$. More precisely,
	\begin{equation}\label{eq: a1 eqn}
	A(s) = -(s-2) +\mathcal{O}\left((s-2)^2\right), \qquad s\to 2.
	\end{equation}	
	Next, define 
	\begin{equation}\label{eq:def_xi}
	\xi(z;s) = -3h(z;s) + \left(-4A^3(s)+9h^2(z;s)\right)^{1/2}
	\end{equation}
	where the square root has a branch cut for $z\in [2/s,\infty)$ and maps $\mathbb{R}^-$ into $i\mathbb{R}^-$. Since $h_+(x) = -h_-(x)$ for $x \in (-1, 1)$, it follows that 
	\begin{equation}
		\label{eq:xi jumps}
	(\xi_+\xi_-)(x) = \left \{ \begin{array}{lr}
		-4A^3(s), & x <1,\\
		4A^3(s), & x >2/s.
	\end{array} \right.
	\end{equation} Set
	\begin{equation}
	u(z;s) = u_1(z;s) + u_2(z;s),
	\end{equation}
	where
	\begin{equation}\label{eq:u1u2}
	u_1(z;s) = \frac{A(s)}{2^{2/3}\xi^{1/3}(z;s)}, \qquad u_2(z;s) = \frac{\xi^{1/3}(z;s)}{2^{4/3}}.
	\end{equation}
	In this last equation, we choose the branch of the cubic root that maps $\mathbb{R}^-$ into $\mathbb{R}^-$ and $i\mathbb{R}^-$ into $i\mathbb{R}^+$, with a cut on the positive real axis.  
	Then, $u$ solves
	\begin{equation}\label{eq: u cubic}
	\frac{4}{3}u^3(z;s) - A(s) u(z;s)=-\frac{h(z;s)}{2}.
	\end{equation}
	Using \eqref{eq:xi jumps}--\eqref{eq:u1u2} we can check that $u(z;s)$ is analytic in a neighborhood of $z=1$ off of $z<1$ and $u_+(x;s) = -u_-(x;s)$ for $x <1$.
	\begin{align}
	\xi(z;s) &= 2 (-A(s))^{3/2}+ 3\sqrt{2}(s-2)(z-1)^{\frac{1}{2}} + \frac{9(s-2)^2}{2(-A(s))^{3/2}}(z-1) \notag \\
	&\qquad +\frac{2+3s}{2\sqrt{2}}(z-1)^{\frac{3}{2}} + \mathcal{O}\left((z-1)^2\right).
	\end{align}
	From this, we then have that
	\begin{align}
	u_1(z) &= -\frac{(-A(s))^{1/2}}{\sqrt{2}}-\frac{s-2}{2\sqrt{2}A(s)} (z-1)^{\frac{1}{2}} -\frac{(s-2)^2}{8(-A(s))^{5/2}}(z-1)\notag\\
	&-\frac{A^3(s)(3s+2)+8(s-2)^3}{24\sqrt{2}A^4(s)}(z-1)^{\frac{3}{2}}\notag \\
	&+ \frac{(s-2)\left(35(s-2)^3+4A^3(s)(3s+2)\right)}{192 (-A(s))^{\frac{11}{2}}}(z-1)^2 
	+\mathcal{O}\left((z-1)^\frac{5}{2}\right)
	\end{align}
	and
	\begin{align}
	u_2(z) &= \frac{(-A(s))^{1/2}}{\sqrt{2}}-\frac{s-2}{2\sqrt{2}A(s)} (z-1)^{\frac{1}{2}} +\frac{(s-2)^2}{8(-A(s))^{5/2}}(z-1)\notag\\
	&-\frac{A^3(s)(3s+2)+8(s-2)^3}{24\sqrt{2}A^4(s)}(z-1)^{\frac{3}{2}}\notag \\
	&- \frac{(s-2)\left(35(s-2)^3+4A^3(s)(3s+2)\right)}{192 (-A(s))^{\frac{11}{2}}}(z-1)^2
	+\mathcal{O}\left((z-1)^\frac{5}{2}\right).
	\end{align}
	Combining these two, we have that
	\begin{align}
	\label{eq:u asymptotics z=1}
	u(z;s) &= -\frac{(s-2)}{\sqrt{2} A(s)}\left(z-1\right)^{1/2}-\frac{A^3(s)(3s+2)+8(s-2)^3}{12\sqrt{2}A^4(s)}(z-1)^{3/2}+\mathcal{O}\left((z-1)^{\frac{5}{2}}\right).
	\end{align}
	Combining \eqref{eq:u asymptotics z=1} and the jump relation $u_+(x;s) = -u_-(x;s)$ for $x <1$ yields the representation $u(z) =  g(z)(z - 1)^{\frac{1}{2}}$, where $g(z)$ is analytic in a small neighborhood of $z=1$. Making the change of variables $u^2 \mapsto \zeta$, we have that
	\begin{equation}\label{eq: zeta series}
	\zeta(z;s) = 				\frac{(s-2)^2}{2A^2(s)}(z-1)+\mathcal{O}\left((z-1)^2\right), 
	\end{equation}
	so that $\zeta$ is a conformal map in a neighborhood of $z=1$ when $s$ is in a neighborhood of $2$. Note that when $s=2$, we have that
	\begin{equation}
	\zeta(z, 2) = \frac{1}{2}(z-1) + \mathcal{O}\left((z-1)^2\right), 
	\end{equation}
	where we have used \eqref{eq: a1 eqn}, so that $\zeta$ is still conformal when $s=2$. Finally, it is immediate from \eqref{eq: u cubic} that $\zeta$ solves \eqref{eq: h cubic}, which completes the proof. 
\end{proof}
Using \eqref{eq: a1 eqn} and \eqref{eq: zeta series}, we may compute
\begin{equation}
\zeta(z,s) = \zeta_1(s) (z-1) + \mathcal{O}\left((z-1)^2\right), \qquad z\to 1, 
\end{equation}
where
\begin{equation}\label{eq: zeta 1 eqn}
\zeta_1(s) =  \frac{1}{2} + \mathcal{O}(s-2), \qquad s\to 2.
\end{equation}
As $s\in\R$, and for $x <1$, we can write $u_+(x) = -u_-(x) = 2^{-4/3}(\xi_+^{1/3} - \xi_-^{1/3})(x)$, where the last quantity is purely imaginary; to see this, we note that \eqref{eq:def_xi} and \eqref{eq:xi jumps} imply that $\xi_{\pm}\in i\mathbb{R}^-$, and by the choice of the cubic root in \eqref{eq:u1u2}, we have that 
 $\left(\xi^{1/3}\right)_{\pm}\in i\mathbb{R}^+$, therefore $\gamma_{m,0}$ is mapped to the ray $\Gamma_4$ by the conformal map $\zeta$. Moreover, we now choose the lips of the lens, $\gamma_{m,0}^\pm$, within the disc so that they are mapped by $\zeta$ to the rays $\Gamma_3$ and $\Gamma_1$, respectively.

Next, we set 
\begin{equation}\label{eq: En eqn ds}
E_n^{(1)}(z) = M(z)\left(\frac{I+i\sigma_1}{\sqrt{2}} \right)^{-1} \left(n^{2/3} \zeta(z;s)\right)^{\sigma_3/4}, 
\end{equation}
where the branch cut for $\zeta^{1/4}$ is taken on $\gamma_{m,0}(s)$. As
\begin{equation}
M_+(z) = M_-(z)\begin{pmatrix}
0 & 1\\
-1 & 0
\end{pmatrix}, \qquad \zeta_+^{1/4}(z,s) = i\zeta_-^{1/4}(z,s) , \qquad z\in\gamma_{m,0}(s), 
\end{equation}
we see that $E_n^{(1)}(z)$ has no jumps within $D_1$. By \eqref{eq: global parametrix eq genus 0} each entry of $M(z)$ is $\mathcal{O}\left((z-1)^{1/4}\right)$ as $z\to 1$, so the singularity of $E_n^{(1)}$ at $z=1$ is removable. Therefore, we see that $E_n^{(1)}(z)$ is analytic in $D_1$. We may then conclude that 
\begin{equation}\label{eq: P1 def}
P^{(1)}(z) = E_n^{(1)}(z) \Psi\left(n^{2/3}\zeta(z;s), n^{2/3} A(s)\right) e^{-\frac{n}{2}h(z)\sigma_3}
\end{equation}
solves \eqref{eq: P RHP ds critical}. Indeed, as $\zeta(z;s)$ maps $\gamma_{m,0}$, $\gamma_{m,0}^+$, and $\gamma_{m,0}^-$ to $\Gamma_4$, $\Gamma_3$, and $\Gamma_1$, respectively, we see that $P^{(1)}$ is analytic in $D_1\setminus\hat{\Sigma}$. Next, using Lemma~\ref{lem: conformal map crit} and \eqref{eq: Psi infinity}, we see that $P^{(1)}$ satisfies \eqref{eq: matching condition ds critical}. Finally, we note that as $P^{(1)}$ and $S$ have the same jumps within $D_1$, the combination $S(z) P^{(1)}(z)^{-1}$ is analytic on $D_1\setminus\{1\}$. Also note that the behavior of $S$ and $P^{(1)}$ are the same as $z\to 1$, so that the singularity is removable. 

\subsection{Proof of Theorem~\ref{thm: rec coeff double scaling critical}}
The final transformation is 
\begin{align}\label{eq: final transformation crit 2}
R(z)= S(z) \begin{cases}
M(z)^{-1}, \qquad & z\in \C\setminus\overline{\left(D_{-1}\cup D_1 \right)} \\
P^{(-1)}(z)^{-1}, \qquad & z\in D_{-1} \\
P^{(1)}(z)^{-1}, \qquad & z\in D_{1} \\
\end{cases}
\end{align}
As before, we want to write the jump matrix as $I + \Delta(z)$, where $\Delta(z)$ has an expansion in inverse powers of $n^\alpha$, for some $\alpha$ to be determined. We recall \eqref{eq: Delta k -1}, where we showed that 
\begin{equation}
\Delta(z) = \sum_{k=1}^\infty \frac{\Delta_k(z)}{n^k}, \qquad n\to\infty, \quad z\in D_{-1}, 
\end{equation} 
where
\begin{equation} \label{eq: Delta k -1 ds}
\Delta_k(z)= \frac{(-1)^{k-1}\prod_{j=1}^{k-1}(2j-1)^2}{4^{k-1}(k-1)!\tilde{h}(z)^k} M(z)\begin{pmatrix}
\frac{(-1)^k}{k} \left(\frac{k}{2}-\frac{1}{4}\right) & i \left(k-\frac{1}{2}\right) \\
(-1)^{k+1} i \left(k-\frac{1}{2}\right) &\frac{1}{k}\left(\frac{k}{2}-\frac{1}{4}\right)
\end{pmatrix}M^{-1}(z),
\end{equation}
and $\tilde{h}(z) = h(z)-2\pi i$. 

To compute the jumps over $\partial D_1$, we first recall that 
\begin{equation}
\Psi(\zeta, w) = \left(1+\frac{\Psi_1(w)}{\zeta}+\mathcal{O}\left(\frac{1}{\zeta^2}\right)\right) \zeta^{-\sigma_3/4} \left(\frac{I+i\sigma_1}{\sqrt{2}}\right)e^{-\left(\frac{4}{3}\zeta^{3/2}-w\zeta^{2/3}\right)\sigma_3}, \qquad \zeta \to\infty.
\end{equation}
We may then use \eqref{eq: Psi infinity}, \eqref{eq: En eqn ds}, and \eqref{eq: P1 def} to see that
\begin{equation}
P^{(1)}(z)M^{-1}(z) = M(z) \left(I + \frac{\tilde{\Psi}_{1/3}(z,s)}{n^{1/3}} + \frac{\tilde{\Psi}_{2/3}(z,s)}{n^{2/3}} + \mathcal{O}\left(\frac{1}{n}\right)\right) M^{-1}(z), \qquad n\to\infty, 
\end{equation}
where
\begin{subequations}\label{eq: Psi 1 Psi 2}
	\begin{equation}
	\tilde{\Psi}_{1/3}(z,s) = \frac{\Psi_{1,12}(w)}{2\zeta^{1/2}(z,s)}\begin{pmatrix}
	i & 1 \\
	1 & -i
	\end{pmatrix},
	\end{equation}
	and
	\begin{equation}
	\tilde{\Psi}_{2/3}(z,s) = \frac{1}{2\zeta(z,s)}\begin{pmatrix}
	\Psi_{1,11}(w)+\Psi_{1,22}(w) & i \left(\Psi_{1,11}(w)-\Psi_{1,22}(w)\right) \\
	-i \left(\Psi_{1,11}(w)-\Psi_{1,22}(w)\right) & \Psi_{1,11}(w)+\Psi_{1,22}(w)
	\end{pmatrix},
	\label{eq: Psi 23}\end{equation}
\end{subequations}
where $\Psi_{1, ij}$ refers to the $(i,j)$ entry of the matrix $\Psi_1$. Moreover, above we have defined
\begin{equation}\label{eq: w}
w = w(s) = n^{2/3}A(s),
\end{equation}
where $A$ is the analytic function given in Lemma~\ref{lem: conformal map}. By the double scaling limit \eqref{eq: ds critical} and \eqref{eq: a1 eqn}, we also have that
\begin{equation}\label{eq: w L2}
w = -L_2 + \mathcal{O}\left(\frac{1}{n^{2/3}}\right), 
\qquad n\to\infty.
\end{equation}

It is now straightforward to see that $\Delta$ can be written in inverse powers of $n^{1/3}$ as
\begin{equation}
\Delta(z) = \sum_{k=1}^\infty \frac{\Delta_{k/3}(z)}{n^{1/3}}, \qquad n\to\infty, z\in\Sigma_R, 
\end{equation}
where $\Delta_{k/3}(z) \equiv 0$ for $z\in\Sigma_R\setminus\left(\partial D_1 \cup \partial D_{-1}\right)$, 
\begin{align*}
\Delta_{k/3}(z) = \begin{cases}
0, \qquad & \frac{k}{3}\not\in \N, \\
\displaystyle
\frac{(-1)^{k-1}\prod_{j=1}^{k-1}(2j-1)^2}{4^{k-1}(k-1)!\tilde{h}(z)^k} M(z)\begin{pmatrix}
\frac{(-1)^k}{k} \left(\frac{k}{2}-\frac{1}{4}\right) & i \left(k-\frac{1}{2}\right) \\
(-1)^{k+1} i \left(k-\frac{1}{2}\right) &\frac{1}{k}\left(\frac{k}{2}-\frac{1}{4}\right)
\end{pmatrix}M^{-1}(z), \qquad & \frac{k}{3}\in\N
\end{cases}
\end{align*}
for $z\in \partial D_1$, and 
\begin{equation}
\Delta_{k/3}(z) = M(z)\tilde{\Psi}_{k/3}(z,s)M^{-1}(z), \qquad z\in\partial D_1,
\end{equation}
where the $\tilde{\Psi}_{k/3}$ can be computed using the expansion of $\Psi$ in \eqref{eq: Psi infinity} along with the definitions of the conformal maps and analytic prefactor given in Lemma~\ref{lem: conformal map} and \eqref{eq: En eqn ds}, respectively. We recall that both $\tilde{\Psi}_{1/3}$ and $\tilde{\Psi}_{2/3}$ are given in \eqref{eq: Psi 1 Psi 2}. 

Now, we may again use the arguments presented in \cite[Section~7]{deift1999strong} and \cite[Section~8]{arnojacobi} to conclude that $R$ has an asymptotic expansion in inverse powers of $n^{1/3}$ of the form
\begin{equation}\label{eq: R expansion c 2}
R(z) = \sum_{k=0}^{\infty} \frac{R_{k/3}(z)}{n^{k/3}}, \qquad n\to\infty, 
\end{equation}
where each $R_{k/3}$ solves the following Riemann-Hilbert problem:
\begin{subequations}
	\label{eq: R_k RHP c 2}
	\begin{alignat}{2}
	&R_{k/3}(z) \text{ is analytic for } z \in \C\setminus  \left(\partial D_{-1} \cup \partial D_{1}\right) \qquad &&\\
	& R_{k/3,+}(z) = R_{k/3,-}(z) + \sum_{j=1}^{k-1}R_{(k-j)/3,-}\Delta_{j/3}(z), \qquad &&z\in \partial D_{-1} \cup \partial D_{1}\\
	& R_{k/3}(z) = \frac{R_{k/3}^{(1)}}{z}+\frac{R_{k/3}^{(2)}}{z^2}+\mathcal{O}\left(\frac{1}{z^3}\right), \qquad &&z \to\infty.
	\end{alignat}
\end{subequations}
By \eqref{eq: final transformation crit 2}, we have that $T(z)= S(z) = R(z)M(z)$ for $z$ outside of the lens. Using \eqref{eq: R expansion c 2}, we then have that
\begin{subequations}
	\label{eq: T expansion 2}
	\begin{equation}
	T^{(1)} = M^{(1)} + \frac{R_{1/3}^{(1)}}{n^{1/3}}+\frac{R_{2/3}^{(1)}}{n^{2/3}}+\mathcal{O}\left(\frac{1}{n}\right), \qquad n\to\infty, 
	\end{equation}
	\begin{equation}
	T^{(2)}= M^{(2)}+\frac{R_{1/3}^{(1)}M^{(1)}+R_{1/3}^{(2)}}{n^{1/3}}+\frac{R_{2/3}^{(1)}M^{(1)}+R_{2/3}^{(2)}}{n^{2/3}}+\mathcal{O}\left(\frac{1}{n}\right), \qquad n\to\infty, 
	\end{equation}
\end{subequations}
where $M^{(1)}$ and $M^{(2)}$ were calculated in \eqref{eq: M1 and M2}.
We therefore turn our attention to computing the first few terms of the expansions of both $R_{1/3}$ and $R_{2/3}$. Before doing so, we first present the following lemma.
\begin{lemma}
	The restrictions of $\Delta_{1/3}$ and $\Delta_{2/3}$ to $\partial D_1$ have meromorphic continuations to a neighborhood of $D_1$. These continuations are analytic, except at $1$, where they have poles of order $1$.
\end{lemma}
\begin{proof}
	We first consider $\Delta_{1/3}$, defined as
	\begin{equation}
	\Delta_{1/3}(z) = M(z) \tilde{\Psi}_{1/3}(z,s) M^{-1}(z), 
	\end{equation}
	where
	\begin{equation*}
	\tilde{\Psi}_{1/3}(z,s) = \frac{\Psi_{1,12}(w)}{2\zeta^{1/2}(z,s)}\begin{pmatrix}
	i & 1 \\
	1 & -i
	\end{pmatrix},
	\end{equation*}
	where the branch cut of $\zeta^{1/2}$ is taken to be $\gamma_{m,0}(s)$. Next, as 
	\begin{equation}
	M_+(z) = M_-(z)\begin{pmatrix}
	0 & 1\\
	-1 & 0
	\end{pmatrix}, \qquad \zeta_+^{1/2}(z,s)=-\zeta_-^{1/2}(z,s), \qquad z\in\gamma_{m,0}(s), 
	\end{equation}
	we see that $\Delta_{1/3,+}(z) = \Delta_{1/3,-}(z)$ for $z\in\gamma_{m,0}$ so that $\Delta_{1/3}$ is analytic in $D_1\setminus\{1\}$. As 
	\begin{equation}
	M(z) \begin{pmatrix}
	i & 1 \\
	1 & -i
	\end{pmatrix}M^{-1}(z) = \sqrt{2}\begin{pmatrix}
	i & 1 \\
	1 & -i
	\end{pmatrix}\frac{1}{\left(z-1\right)^{1/2}}+\mathcal{O}\left((z-1)^{1/2}\right), \qquad z\to 1, 
	\end{equation}
	and $\zeta(z,s) =\zeta_1(s)(z-1)+\mathcal{O}\left(z-1\right)^2$, where $\zeta_1(s)\not=0$ as $\zeta$ is a conformal mapping from $1$ to $0$, we see that the isolated singularity at $z=1$ is a simple pole. 
	
	In the case, of $\Delta_{2/3}$, we note that
	\begin{equation}
	M(z)\tilde{\Psi}_{2/3}(z,s)M^{-1}(z) = \tilde{\Psi}_{2/3}(z,s), 
	\end{equation}
	so that the lemma follows immediately from \eqref{eq: Psi 23}.
\end{proof}
In light of the lemma above, we may write that
\begin{equation}
\Delta_{1/3}(z) = \frac{C^{(1/3)}}{z-1}, \qquad z\to 1.
\end{equation}
Using that $\zeta(z, s) = \zeta_1(s)(z-1)+ \mathcal{O}\left(z-1\right)^2$ as $z\to 1$, we compute that
\begin{equation}
C^{(1/3)} = \frac{\Psi_{1,12}(w)}{\sqrt{2}\zeta_1^{1/2}(s)}\begin{pmatrix}
i & 1 \\
1 & -i
\end{pmatrix}.
\end{equation}
By direct inspection, we see that
\begin{align}
R_{1/3}(z) = \begin{cases}
\displaystyle\frac{C^{(1/3)}}{z-1}, \qquad & z\in\C\setminus D_1, \\[2mm]
\displaystyle\frac{C^{(1/3)}}{z-1} - \Delta_{1/3}(z), \qquad & z\in D_1,
\end{cases}
\end{align}
solves the Riemann-Hilbert problem \eqref{eq: R_k RHP c 2} when $k=1$, so that 
\begin{equation}\label{eq: R13 expansion}
R_{1/3}^{(1)} = R_{1/3}^{(2)} =C^{(1/3)}.
\end{equation}

We analogously solve for the terms in the expansion of $R_{2/3}$ by writing 
\begin{equation}
R_{1/3}(z)\Delta_{1/3}(z) + \Delta_{2/3}(z) = \frac{C^{(2/3)}}{z-1}, 
\end{equation}
where we may compute that
\begin{equation}
C^{(2/3)} = \frac{1}{2 \zeta_1(s)} \begin{pmatrix}
\Psi_{1,11}(w)+\Psi_{1,22}(w)  & i \left(\Psi_{1,11}(w)-\Psi_{1,22}(w)\right)\\
-i \left(\Psi_{1,11}(w)-\Psi_{1,22}(w)\right)& \Psi_{1,11}(w)+\Psi_{1,22}(w) 
\end{pmatrix}.
\end{equation}
Then, 
\begin{align}
R_{2/3}(z) = \begin{cases}
\displaystyle\frac{C^{(2/3)}}{z-1}, \qquad & z\in\C\setminus D_1, \\[2mm]
\displaystyle\frac{C^{(2/3)}}{z-1} - R_{1/3}(z)\Delta_{1/3}(z)-\Delta_{2/3}(z), \qquad & z\in D_1,
\end{cases}
\end{align}
solves \eqref{eq: R_k RHP c 2}, and we may compute that the terms in the large $z$ expansion of $R_{2/3}$ are given by
\begin{equation}\label{eq: R23 expansion}
R^{(1)}_{2/3} = R^{(2)}_{2/3}= C^{(2/3)}. 
\end{equation}

Now, combining the previous equations (in particular \eqref{eq: recurrence coeffs in terms of T}, \eqref{eq: T expansion 2}, 
\eqref{eq: R13 expansion}, and \eqref{eq: R23 expansion}), we have
\begin{subequations}
	\begin{equation}
	\alpha_n(s) = \frac{\Psi_{1,11}(w)-\Psi_{1,22}(w)+\Psi_{1,12}^2(w)}{\zeta_1(s)} \frac{1}{n^{2/3}} + \mathcal{O}\left(\frac{1}{n}\right), \qquad n\to\infty, 
	\end{equation}
	and
	\begin{equation}
	\beta_n(s) = \frac{1}{4} + \frac{\Psi_{1,11}(w)-\Psi_{1,22}(w)+\Psi_{1,12}^2(w)}{2\zeta_1(s)} \frac{1}{n^{2/3}}+ \mathcal{O}\left(\frac{1}{n^{4/3}}\right), \qquad n\to\infty.
	\end{equation}
\end{subequations}
Next, using \eqref{eq: zeta 1 eqn} and the double scaling limit \eqref{eq: ds critical}, along with the formula for $w$ in \eqref{eq: w}, we have that
\begin{subequations}
	\begin{equation}
	\alpha_n(s) = 2\left(\Psi_{1,11}(w)-\Psi_{1,22}(w)+\Psi_{1,12}^2(w)\right)\frac{1}{n^{2/3}} + \mathcal{O}\left(\frac{1}{n}\right), \qquad n\to\infty, 
	\end{equation}
	and
	\begin{equation}
	\beta_n(s) = \frac{1}{4} + \left(\Psi_{1,11}(w)-\Psi_{1,22}(w)+\Psi_{1,12}^2(w)\right) \frac{1}{n^{2/3}}+ \mathcal{O}\left(\frac{1}{n}\right), \qquad n\to\infty.
	\end{equation}
\end{subequations}
Using \eqref{entries_Psi1}, we can simplify the previous combination of entries of $\Psi_1(w)$:
\[
\Psi_{1,11}(w)-\Psi_{1,22}(w)+\Psi_{1,12}^2(w)
=
-\frac{1}{2}(q^2(w)+q'(w)),
\]
so that by using \eqref{eq: w L2} we have that
\begin{subequations}
	\begin{equation}
	\alpha_n(s)= -\left(q^2(-L_2)+q'(-L_2)\right)\frac{1}{n^{2/3}} + \mathcal{O}\left(\frac{1}{n}\right)
	\end{equation}
	and
	\begin{equation}
	\beta_n(s) = \frac{1}{4} -\frac{q^2(-L_2)+q'(-L_2)}{2}\frac{1}{n^{2/3}} + \mathcal{O}\left(\frac{1}{n}\right)
	\end{equation}
\end{subequations}
as $n\to\infty$. Finally, the fact that the function $U(w)=q^2(w)+q'(w)$ is free of poles for $w\in\mathbb{R}$ follows from \cite[Lemma 3.5]{IKO_critical}, as well as from \cite[Lemma 1, Corollary 1]{xu2011painleve}; in this last reference, the theorem is a consequence of the vanishing lemma applied to the Painlev\'e XXXIV Riemann--Hilbert problem, and then translating the result to solutions of Painlev\'e II. This completes the proof of Theorem~\ref{thm: rec coeff double scaling critical}.

\bibliographystyle{plain}

\end{document}